\documentclass[letterpaper,11pt]{amsart}

\usepackage[margin=1.2in]{geometry}
\usepackage{amsmath,amsthm,amssymb}
\usepackage{xspace,xcolor}
\usepackage[breaklinks,colorlinks,citecolor=teal,linkcolor=teal,urlcolor=teal,pagebackref,hyperindex]{hyperref}
\usepackage[alphabetic]{amsrefs}
\usepackage[all]{xy}
\usepackage{color}
\usepackage{comment} 

\usepackage{tikz-cd}

\theoremstyle{plain}
\newtheorem{thm}{Theorem}[section]

\newtheorem{lem}[thm]{Lemma}
\newtheorem{prop}[thm]{Proposition}
\newtheorem{cor}[thm]{Corollary}

\theoremstyle{definition}
\newtheorem{defi}[thm]{Definition}

\newtheorem{eg}[thm]{Example}

\theoremstyle{remark}
\newtheorem{rmk}[thm]{Remark}

\def\ZZ{{\mathbf Z}}
\def\QQ{{\mathbf Q}}
\def\RR{{\mathbf R}}
\def\CC{{\mathbf C}}

\def\AA{{\mathbf A}}
\def\C{{\mathbf C}}

\def\cD{\mathcal{D}}
\def\cE{\mathcal{E}}
\def\cF{\mathcal{F}}

\def\cH{\mathcal{H}}

\def\cJ{\mathcal{J}}
\def\cK{\mathcal{K}}

\def\cM{\mathcal{M}}
\def\cN{\mathcal{N}}
\def\cO{\mathcal{O}}

\def\cR{\mathcal{R}}

\def\cT{\mathcal{T}}

\def\.{\cdot}
\def\^{\widehat}

\def\de{\partial}

\def\({\left(}
\def\){\right)}

\renewcommand{\and}{ \ \ \text{ and } \ \ }

\def\gr{{\mathrm{Gr}}}

\DeclareMathOperator{\coker} {coker}

\DeclareMathOperator{\lct} {lct}

\makeatletter
\@namedef{subjclassname@2020}{\textup{2020} Mathematics Subject Classification}
\makeatother

\usepackage{soul}

\begin{document}

\author[Q.~Chen]{Qianyu Chen}

\address{Department of Mathematics, University of Michigan, 530 Church Street, Ann Arbor, MI 48109, USA}

\email{qyc@umich.edu}

\author[B.~Dirks]{Bradley Dirks}

\address{Department of Mathematics, Stony Brook University, Stony Brook, NY 11794, USA}

\email{bradley.dirks@stonybrook.edu}

\author[M.~Musta\c{t}\u{a}]{Mircea Musta\c{t}\u{a}}

\address{Department of Mathematics, University of Michigan, 530 Church Street, Ann Arbor, MI 48109, USA}

\email{mmustata@umich.edu}

\thanks{Q.C. was partially supported by NSF grant DMS-1952399, B.D. was partly supported by NSF grant DMS-2001132 and NSF-MSPRF grant DMS-2303070, 
 and M.M. was partially supported by NSF grants DMS-2301463 and DMS-1952399.}


\begin{abstract}
We give an introduction to the theory of $V$-filtrations of Malgrange and Kashiwara. After discussing the basic properties of this construction (in the case
of a smooth hypersurface and, later, in the general case), we describe the connection with the theory of $b$-functions. As an example, we treat the case
of weighted homogeneous isolated singularities. We discuss the compatibility of $V$-filtrations with proper push-forward and duality and the connection with
nearby and vanishing cycles via the Riemann-Hilbert correspondence. We end by describing some invariants of singularities via
the $V$-filtration. 
\end{abstract}

\title[An introduction to $V$-filtrations] {An introduction to $V$-filtrations}

\maketitle

\tableofcontents

\section{Introduction}

Let $X$ be a smooth, irreducible complex algebraic variety. The sheaf $\cD_X$ of differential operators of $X$ is the subsheaf of ${\mathbf C}$-algebras 
of ${\mathcal End}_{\mathbf C}(\cO_X)$ generated by $\cO_X$ and the sheaf of derivations ${\mathcal Der}_{\mathbf C}(\cO_X)$. A $\cD$-module on $X$
is simply a sheaf of (left or right) $\cD_X$-modules. The theory of $\cD$-modules started being developed in the 70s in the work of Bernstein (in the algebraic setting,
when $X={\mathbf A}_k^n$, for any field $k$ of characteristic $0$) and of Kashiwara and his collaborators (in the complex analytic setting, when $X$ is a complex
manifold). While one of its original motivations came from the desire to have an algebraic approach to questions concerning systems of linear PDEs (hence the name 
\emph{algebraic analysis}, sometimes used for this field), it turned out that the theory of $\cD$-modules had deep connections with other fields of mathematics, especially
representation theory and singularity theory. For details on the history of the subject and applications to representation theory, we refer to \cite{HTT}. 

A special feature of the field, that takes it outside of the usual algebro-geometric framework, is that $\cD_X$ is a sheaf of \emph{noncommutative} rings. However,
$\cD_X$ is not far from being commutative, and \emph{coherent} $\cD_X$-modules turn out to have associated objects (depending on the choice of a suitable filtration) 
that are coherent sheaves of $\cO$-modules on the cotangent bundle
$T^*X$ (we review this briefly in Section~\ref{section_SS} below). The  \emph{dimension} of a coherent $\cD_X$-module is defined by considering the 
dimension of the support of the corresponding sheaf on $T^*X$ (this support is, in fact, independent of any choices).
A fundamental result
in the theory says that the dimension of a nonzero coherent $\cD_X$-module is always $\geq n=\dim(X)$. Because of this, the $\cD_X$-modules 
that are either zero or of dimension $n$
(called \emph{holonomic} $\cD_X$-modules) enjoy special properties (for example, they form an abelian category in which all objects have finite length). Furthermore, they
are preserved by natural pull-back and push-forward functors. The good properties of holonomic $\cD$-modules are at the origin of many of the striking applications
of this theory.

One of the crown jewels of $\cD$-module theory is the Riemann-Hilbert correspondence, proved independently by Kashiwara \cite{KashiwaraRH} and Mebkhout
\cite{MebkhoutRH}. This gives an equivalence of categories ${\rm DR}_X^{\rm an}\colon {\mathcal D}^b_{rh}(\cD_X)\to {\mathcal D}^b_{c}(X^{\rm an})$ (the \emph{analytic de Rham  functor}). The left-hand side is the bounded derived category of complexes whose cohomology objects are regular holonomic 
$\cD_X$-modules, while the right-hand side is the bounded derived category of complexes of ${\mathbf C}$-vector spaces on $X^{\rm an}$ with constructible cohomology. 
We discuss this briefly in Section~\ref{section_RH}. This equivalence of triangulated categories induces an equivalence of abelian categories between regular holonomic $\cD_X$-modules
and the famous category of perverse sheaves on $X^{\rm an}$, introduced by Beilinson, Bernstein, Deligne, and Gabber \cite{BBDG}. A classical special case of this correspondence 
(say, when $X$ is projective) identifies the category of $\cD_X$-modules on $X$ that are $\cO_X$-coherent (equivalently, these are vector bundles with integrable
connection) with the category of local systems on $X^{\rm an}$. 

The role of the $V$-filtration is to give a description of two important constructions, the \emph{nearby} and \emph{vanishing cycles}, at the level of $\cD_X$-modules.
More precisely, to any ${\mathcal K}\in \cD^b_c(X^{\rm an})$ and any nonzero regular function $f\colon X\to {\mathbf A}^1$, 
one can associate the nearby cycles $\psi_f{\mathcal K}$ and the vanishing cycles
$\varphi_f{\mathcal K}$ in $\cD^b_c(X^{\rm an})$. By a result of Gabber, if ${\mathcal K}$ is a perverse sheaf, then both
$\psi_f{\mathcal K}[-1]$ and $\varphi_f{\mathcal K}[-1]$ are perverse sheaves. If ${\mathcal K}={\rm DR}_X^{\rm an}(\cM)$, for a regular holonomic 
$\cD_X$-module $\cM$, then the $V$-filtration provides a $\cD$-module theoretic description of $\psi_f{\mathcal K}[-1]$ and $\varphi_f{\mathcal K}[-1]$.
This construction was first introduced by Malgrange in \cite{Malgrange} in the case when $\cM=\cO_X$, with the goal of comparing two important 
invariants of $f$, the roots of the Bernstein-Sato polynomial of $f$ and the eigenvalues of the monodromy action on the cohomology of the Milnor fiber of $f$.
It was later generalized to arbitrary holonomic $\cD_X$-modules by Kashiwara in \cite{Kashiwara3}.

The $V$-filtration and its interplay with the Hodge filtration plays a key role in Saito's theory of Hodge modules, see \cite{Saito-MHP} and \cite{Saito-MHM}. 
First, the description of Hodge modules is done by induction on dimension and going from objects in dimension $n$ to objects in dimension $n-1$ is done 
via the graded pieces with respect to the $V$-filtration. Second, Hodge modules come with a canonical filtration (the \emph{Hodge filtration}) and this is required
to satisfy a suitable compatibility with the $V$-filtrations associated to arbitrary functions. It follows that in the setting of Hodge modules, it is important to keep track
of the $V$-filtration and the Hodge filtration at the same time. This is an important topic, but beyond the scope of these notes.

Another topic that we do not touch upon is the $V$-filtration associated to several functions. This more general version is also due to Kashiwara  \cite{Kashiwara3},
but it is more subtle and less understood than in the special case of one function. While we do not treat the general case here for reasons of space, we refer to \cite{BMS} for its
description and relation to $b$-functions 
and to \cite{CD} for its connection to the theory of Hodge modules.

Our goal here is to give an introduction 
to $V$-filtrations. The following is an outline of what we cover.
In the next section we give an overview of some basic notions and results from the theory of $\cD$-modules that we need. In Section~\ref{section_Vfiltration_smooth}
we discuss the $V$-filtration on general $\cD$-modules in the special case when the hypersurface defined by $f\in\cO_X(X)$ is smooth.
The case of an arbitrary function $f$ is handled by a trick due to Kashiwara: if $\iota\colon X\hookrightarrow X\times {\mathbf A}^1$ is the closed immersion giving the
graph of $f$, then the $V$-filtration corresponding to a $\cD_X$-module $\cM$ is obtained by taking the $V$-filtration on the push-forward $\iota_+(\cM)$ with respect to the function
$t$ on $X\times {\mathbf A}^1$ coming from the coordinate on ${\mathbf A}^1$. We discuss this in Section~\ref{section_Vfilt_general}, while in
Section~\ref{sect_Vfil_bfcn} we relate the $V$-filtration to $b$-functions. From this point on, the different sections can be read independently.
 We give a nontrivial example of a $V$-filtration on $\iota_+(\cO_X)$ when $f$
has weighted homogeneous isolated singularities in Section~\ref{section_weighted_homogeneous}. In Sections~\ref{section_behavior_push_forward} 
and \ref{section_behavior_push_forward2} 
we discuss
the compatibility of the $V$-filtration with duality and, respectively, proper push-forward, while in Section~\ref{section_comparison} we explain the description of the nearby 
and vanishing cycles via $V$-filtrations. In the final Section~\ref{section_b_invar_sing} we discuss the connection between the $V$-filtration on 
$\iota_+(\cO_X)$ and certain invariants of singularities of the given hypersurface.

\subsection*{Acknowledgment} We are indebted to the anonymous referee for several comments and suggestions on a preliminary version of this paper.

\section{A brief review of $\cD$-module theory}\label{section_review}

In this section we review some basic definitions and results on $\cD$-modules. For details and proofs we refer to \cite{HTT} and \cite{Bernstein}.
Unless explicitly mentioned otherwise, we work over an algebraically closed field $k$, of characteristic $0$. 

\subsection{The sheaf of differential operators}\label{sect_D_X}
Let $X$ be a smooth, irreducible, $n$-dimensional algebraic variety over $k$. The \emph{sheaf of differential operators} on $X$ is the subsheaf of $k$-algebras
$\cD_X\subseteq {\mathcal End}_k(\cO_X)$ generated by $\cO_X$ (with $\cO_X$ acting on itself by multiplication) and the sheaf of $k$-derivations 
${\mathcal T}_X={\mathcal Der}_k(\cO_X)$. If $x_1,\ldots,x_n$ form an algebraic system of coordinates on an open subset $U\subseteq X$ (that is,
$dx_1,\ldots,dx_n$ trivialize $\Omega_U$ and we have the dual basis of ${\mathcal T}_U$ given by $\partial_1,\ldots,\partial_n$), then $\cD_X\vert_U=\cD_U$
is a free left (or right) $\cO_U$-module, with basis given by $\partial^{\alpha}:=\partial_1^{\alpha_1}\cdots\partial_n^{\alpha_n}$ for $\alpha=(\alpha_1,\ldots,\alpha_n)
\in {\mathbf Z}_{\geq 0}^n$. 

The sheaf $\cD_X$ carries an increasing filtration $(F_p\cD_X)_{p\geq 0}$ by order of differential operators: if $U$ and $x_1,\ldots,x_n$ are as above,
then $F_p\cD_U$ is generated as a (left or right) $\cO_U$-module by $\partial^{\alpha}$, with $|\alpha|:=\alpha_1+\ldots+\alpha_n\leq p$. 
It is easy to see that we have $F_p\cD_X\cdot F_q\cD_X\subseteq F_{p+q}\cD_X$ and $[F_p\cD_X,F_q\cD_X]\subseteq F_{p+q-1}\cD_X$ for all $p,q\geq 0$.
We may thus form the sheaf of graded rings ${\rm Gr}^F_{\bullet}(\cD_X)=\bigoplus_{p\geq 0}F_p\cD_X/F_{p-1}\cD_X$ and this is, in fact, a sheaf of
\emph{commutative} rings. Since ${\rm Gr}^F_1(\cD_X)={\mathcal T}_X$, it is clear that we have an induced morphism of sheaves of commutative rings
$${\mathcal Sym}_{\cO_X}^{\bullet}(\cT_X)\to {\rm Gr}^F_{\bullet}(\cD_X)$$
and by looking in local coordinates we see that this is an isomorphism. It is then easy to deduce that
for every affine open subset $U\subseteq X$, the ring $\cD_X(U)$ is both left and right Noetherian.

A left (right) $\cD$-module on $X$ is a sheaf of left (respectively, right) $\cD_X$-modules.
Two basic examples of left $\cD$-modules are $\cD_X$ and $\cO_X$ (note that we have a tautological action of $\cD_X$ on $\cO_X$ since $\cD_X$
is a subsheaf of ${\mathcal End}_k(\cO_X)$). 

\subsection{The de Rham complex}\label{section_de_Rham}
Recall that a connection $\nabla$ on an $\cO_X$-module $\cM$ is a $k$-linear map $\nabla\colon \cM\to \Omega_X\otimes_{\cO_X}\cM$ that satisfies the Leibniz rule: 
$\nabla(fu)=f\cdot \nabla(u)+df\otimes u$ for every local sections $f\in\cO_X$ and $u\in\cM$. Given such a connection $\nabla$, we get induced maps
$$\nabla\colon \Omega_X^i\otimes_{\cO_X}\cM\to \Omega_X^{i+1}\otimes_{\cO_X}\cM$$
that satisfy again a version of Leibniz rule. The connection is \emph{integrable} if $\nabla\circ\nabla=0$ on $\cM$, in which case also the higher order compositions
vanish, so we have the de Rham complex
$$0\to\cM\overset{\nabla}\longrightarrow\Omega_X\otimes_{\cO_X}\cM\overset{\nabla}\longrightarrow \ldots\overset{\nabla}\longrightarrow
\Omega_X^n\otimes_{\cO_X}\cM\to 0,$$
placed in degrees $0,\ldots,n$.
It follows easily from the definition that giving a left $\cD_X$-module structure on $\cM$ is equivalent to giving an $\cO_X$-module structure on $\cM$,
together with an integrable connection. We denote by ${\rm DR}_X(\cM)$ the corresponding shifted de Rham complex $\Omega_X^{\bullet}\otimes_{\cO_X}\cM[n]$;
note that the maps in this complex are
not $\cO_X$-linear. One can show (see \cite[Theorem~1.4.10]{HTT}) that a left $\cD_X$-module is coherent as an $\cO_X$-module if and only if it is a locally
free $\cO_X$-module of finite rank; in other words, it is a vector bundle with an integrable connection. 

\subsection{The left-right equivalence}\label{section_equivalence}
In what follows we will mostly consider left $\cD_X$-modules, but right $\cD_X$-modules sometimes naturally enter the picture (for example, when considering duality
and direct image of $\cD$-modules). A useful fact is that there is a canonical equivalence of categories between left and right $\cD_X$-modules, as follows. If we denote 
by $\cM^r$ the right $\cD_X$-module corresponding to the left $\cD_X$-module $\cM$, then as $\cO_X$-modules we have
$\cM^r=\cM\otimes_{\cO_X}\omega_X$, where $\omega_X=\Omega_X^n$, so $\cM\simeq {\mathcal Hom}_{\cO_X}(\omega_X,\cM^r)$. 
The right action of $\cD_X$ on $\cM\otimes_{\cO_X}\omega_X$ is easy to describe in a system
of coordinates $x_1,\ldots,x_n$, as follows. The map given by the \emph{classical adjoint} 
\begin{equation}\label{eq_classical_adjoint}
P = \sum_{\alpha \in \mathbf Z^n_{\geq 0}} h_{\alpha} \de_x^{\alpha}\mapsto {}^t P: = \sum_{\alpha \in \mathbf Z^n_{\geq 0}} (-\de_x)^\alpha h_{\alpha}
\end{equation}
is an automorphism $\cD_X\simeq \cD_X^{\rm op}$ and if $\eta=dx_1\wedge\ldots\wedge dx_n$, then
\[ (m \otimes \eta) \cdot P = ({}^t P \cdot m)\otimes \eta.\]
For example, the left $\cD_X$-module $\cD_X^{\ell}$ corresponding to the standard right $\cD_X$-module structure on $\cD_X$ is isomorphic to 
$\cD_X$ via the map given by the classical adjoint. 
For the fact that in general we get a global equivalence of categories, see \cite[Section~1.2]{HTT}.

\subsection{Singular support and holonomic $\cD$-modules}\label{section_SS}
From now on, unless explicitly mentioned otherwise, all our $\cD_X$-modules are left $\cD_X$-modules. Let $\cM$
be a \emph{coherent} such $\cD_X$-module: this means that $\cM$ is quasi-coherent as an $\cO_X$-module and it is locally finitely
generated over $\cD_X$. In this case one can show that there is a \emph{good filtration} on $\cM$, that is, an increasing filtration $(F_p\cM)_{p\in\ZZ}$ with the following properties:
\begin{enumerate}
\item[i)] $F_p\cM=0$ for $p\ll 0$ and $\cM=\bigcup_{p\in\ZZ}F_p\cM$.
\item[ii)] $F_p\cD_X\cdot F_q\cM\subseteq F_{p+q}\cM$ for all $p,q\in\ZZ$.
\item[iii)] ${\rm Gr}^F_{\bullet}(\cM):=\bigoplus_{p\in\ZZ}F_p\cM/F_{p-1}\cM$ is locally finitely generated over ${\rm Gr}^F_{\bullet}(\cD_X)$.
\end{enumerate}
Note that if $\pi\colon T^*X\to X$ is the cotangent bundle of $X$, then we have a coherent sheaf $\cF$ on $T^*X$ such that $\pi_*(\cF)\simeq {\rm Gr}^F_{\bullet}(\cM)$.
While $\cF$ depends on the choice of filtration, its support ${\rm Supp}(\cF)$ (viewed as a subset of $T^*X$) is independent of the filtration.
This is the \emph{characteristic variety} (or \emph{singular support}) ${\rm SS}(\cM)$. The \emph{dimension} $\dim(\cM)$ of $\cM$ is, by definition,
the dimension of ${\rm SS}(\cM)$. Note that $\dim(\cM)\leq\dim(T^*X)=2n$. 

A fundamental result in the theory of $\cD_X$-modules is \emph{Bernstein's inequality}: if $\cM$ is a coherent $\cD_X$-module, then every irreducible component of ${\rm SS}(\cM)$ has dimension $\geq n$.
In particular, if $\cM\neq 0$, then 
$\dim(\cM)\geq n$. 
A coherent $\cD_X$-module $\cM$ is \emph{holonomic} if $\cM=0$ or $\dim(\cM)=n$. 

Given a short exact sequence of coherent $\cD_X$-modules
\begin{equation}\label{SES_coh1}
0\to\cM'\to\cM\to\cM''\to 0,
\end{equation}
after choosing a good filtration on $\cM$ and taking the induced filtrations on $\cM'$ and $\cM''$,
we obtain the short exact sequence 
\begin{equation}\label{SES_coh2}
0\to {\rm Gr}^F_{\bullet}(\cM')\to {\rm Gr}^F_{\bullet}(\cM)\to {\rm Gr}^F_{\bullet}(\cM'')\to 0,
\end{equation}
so $\dim(\cM)=\max\big\{\dim(\cM'),\dim(\cM'')\big\}$. In particular, $\cM$ is holonomic if and only if both $\cM'$ and $\cM''$ are holonomic.
It is then easy to see that the holonomic $\cD_X$-modules form an abelian category. 

A more refined invariant than the characteristic variety is the characteristic cycle of a coherent $\cD_X$-module. Given a good filtration on the coherent $\cD_X$-module 
$\cM$, if $\cF$ is the coherent sheaf on $T^*X$ corresponding to ${\rm Gr}^F_{\bullet}(\cM)$, the \emph{characteristic cycle} ${\rm CC}(\cM)$ is
$${\rm CC}(\cM):=\sum_{Z}\ell(\cF_{\eta_Z})[Z],$$
where the sum is over the irreducible components $Z$ of ${\rm SS}(\cM)$ and $\eta_Z$ is the generic point of $Z$. 
It is a bit more subtle, but one can show that the characteristic cycle, too, is independent of the choice of good filtration.
Given a short exact sequence (\ref{SES_coh1}) of holonomic $\cD_X$ modules, we deduce from the corresponding exact sequence (\ref{SES_coh2})
that ${\rm CC}(\cM)={\rm CC}(\cM')+{\rm CC}(\cM'')$. Since we clearly have ${\rm CC}(\cM)=0$ if and only if $\cM=0$, it is easy to deduce that
in the category of holonomic $\cD_X$-modules, every object has finite length. 

\subsection{Functors on $\cD$-modules}\label{section_functors}
We next recall the definition of pull-back and push-forward for $\cD$-modules. Let $f\colon X\to Y$ be a morphism of smooth, irreducible algebraic varieties over $k$. 
If $\cM$ is a left $\cD_Y$-module, then the $\cO_X$-module $f^*(\cM)=\cO_X\otimes_{f^{-1}(\cO_Y)}f^{-1}(\cM)$ has a natural structure of left $\cD_X$-module:
the connection $\nabla_{\cM}$ on $\cM$ induces the following connection on $f^*(\cM)$:
$$\nabla_{f^*(\cM)}\colon f^*(\cM)\to f^*(\Omega_Y)\otimes_{\cO_X}f^*(\cM)\to \Omega_X\otimes_{\cO_X}f^*(\cM)$$
such that if $\nabla_{\cM}(u)=\sum_idg_i\otimes v_i$, then 
$$\nabla_{f^*(\cM)}\big(h\otimes f^{-1}(u)\big)=dh\otimes f^{-1}(u)+\sum_ihd(g_i\circ f)\otimes f^{-1}(v_i)\quad\text{for}\quad h\in\cO_X.$$
We thus get a right-exact functor from $\cD_Y$-modules to $\cD_X$-modules
that extends to an exact functor between triangulated categories ${\mathbf L}f^*\colon
\cD^b(\cD_Y)\to \cD^b(\cD_X)$. Note that we denote by $\cD^b(\cD_X)$, respectively $\cD^b(\cD_X^r)$, the bounded derived category of left, respectively right,
$\cD_X$-modules. 

In particular, we see that $f^*(\cD_Y)$ is a left $\cD_X$-module; in fact, it is a $\cD_X-f^{-1}(\cD_Y)$-bimodule, using also the right $\cD_Y$-module structure of 
$\cD_Y$. As such, it is denoted by $\cD_{X\to Y}$.
It follows that if $\cM$ is a \emph{right} $\cD_X$-module, then $\cM\otimes_{\cD_X}\cD_{X\to Y}$ is a right $f^{-1}(\cD_Y)$-module, and thus 
$f_*(\cM\otimes_{\cD_X}\cD_{X\to Y})$ is a right $\cD_Y$-module. Since this construction combines a left exact functor with a right-exact functor, the correct functor
is defined at the level of derived categories. More precisely, we have the exact functor between triangulated categories
$$f_+={\mathbf R}f_*(-\otimes^L_{\cD_X}\cD_{X\to Y})\colon \cD^b(\cD_X^r)\to \cD^b(\cD_Y^r).$$
Finally, the equivalence of categories between left and right $\cD$-modules induces equivalences
$\cD^b(\cD_X)\simeq \cD^b(\cD_X^r)$ and $\cD^b(\cD_Y)\simeq \cD^b(\cD_Y^r)$. Using these, we obtain the push-forward functor for left $\cD$-modules
$f_+\colon \cD^b(\cD_X)\to \cD^b(\cD_Y)$. 

If $g\colon Y\to Z$ is another such morphism, then we have a canonical isomorphism $(g\circ f)_+\simeq g_+\circ f_+$. Since every morphism $f\colon X\to Y$
admits the decomposition $X\overset{i}\hookrightarrow X\times Y\overset{p}\longrightarrow Y$, where $p$ is the projection and $i$ is the closed immersion 
given by $i(x)=\big(x,f(x)\big)$, one can reduce understanding the push-forward via general morphisms to the case of projections and closed immersions. 

We consider first the case of a closed immersion $i\colon X\hookrightarrow Y$ between smooth varieties. Working locally, we may assume that we have 
algebraic coordinates $x_1,\ldots,x_n,y_1,\ldots,y_m$ on $Y$ such that $X$ is defined by $(y_1,\ldots,y_m)$. In this case we have
$$\cD_{X\to Y}=\cD_Y/(y_1,\ldots,y_m)\cD_Y=\bigoplus_{\beta\in {\mathbf Z}_{\geq 0}^m}\cD_X\partial_y^{\beta},$$
where the right multiplication by $y_i$ and $\partial_{y_i}$ are given by
$$P\partial_y^{\beta}\cdot\partial_{y_i}=P\partial_y^{\beta+e_i}\quad\text{and}\quad P\partial_y^{\beta}\cdot y_i=\beta_iP\partial_y^{\beta-e_i}.$$
In particular, we see that $\cD_{X\to Y}$ is a free left $\cD_X$-module. Since both functors $i_*$ and $-\otimes_{\cD_X}\cD_{X\to Y}$ are exact,
we see that $i_+$ is induced at the level of derived categories by an exact functor at the level of abelian categories. 
An important result in this setting is Kashiwara's Equivalence Theorem (see \cite[Section~1.6]{HTT}): this says that $i_+$ gives an equivalence between the category 
of coherent (holonomic) $\cD_X$-modules and the category of coherent (respectively, holonomic) $\cD_Y$-modules with support contained in $X$;
the inverse functor is ${\mathbf L}i^*[-m]$. More generally, there is a natural isomorphism $i_+\circ {\mathbf L}i^*[-m]\simeq {\mathbf R}\Gamma_X$
of functors on $\cD^b(\cD_Y)$,
where ${\mathbf R}\Gamma_X$ is the derived functor of the functor that maps $\cM$ to the subsheaf $\Gamma_X\cM$ of sections supported on $X$. 
In particular, we have 
\begin{equation}\label{eq_local_coho}
i_+(\cO_X)\simeq \cH^m_X(\cO_Y),
\end{equation}
where the right-hand side is the $m^{\rm th}$ local cohomology sheaf of $\cO_Y$ with respect to $X$.


Suppose now that $Y$ and $Z$ are smooth varieties and $p\colon X=Y\times Z\to Y$ is the projection onto the first factor. In this case we have 
$\cD_{X\to Y}\simeq \cD_Y\otimes_k\cO_Z$. 
It is easy to see that the de Rham complex ${\rm DR}_Z(\cD_Z)$ is a complex of right $\cD_Z$-modules and the corresponding
complex of left $\cD_Z$-modules gives a
free resolution of $\cO_Z$. This implies that if $Y$ is a point and if 
$\cM$ is a left $\cD_Z$-module, then
$$p_+(\cM)={\mathbf R}\Gamma\big({\rm DR}_Z(\cM)\big),$$
so the cohomology of $p_+(\cM)$ gives, up to a shift, the de Rham cohomology of $\cM$. For arbitrary $Y$, we have
\begin{equation} \label{pushforwardSmooth} 
p_+(\cM)\simeq {\mathbf R}p_*\big({\rm DR}_{X/Y}(\cM)\big),
\end{equation}
where ${\rm DR}_{X/Y}(\cM)$ is the shifted relative de Rham complex $\Omega_{X/Y}^{\bullet}\otimes_{\cO_X}\cM[\dim(Z)]$.

Let $\cD^b_{h}(\cD_X)$ denote the bounded derived category of complexes of $\cD$-modules on $X$ with holonomic cohomology.
It is an important result of the theory that pull-back and push-forward of $\cD$-modules preserves holonomicity: if $f\colon X\to Y$ is a morphism
of smooth, irreducible varieties, then we have induced functors
$${\mathbf L}f^*\colon \cD^b_h(\cD_Y)\to\cD^b_h(\cD_X)\quad\text{and}\quad f_+\colon\cD^b_h(\cD_X)\to\cD^b_h(\cD_Y).$$
An important example is the following: suppose that $X$ is the open subset of $Y$ where $g\neq 0$, for some nonzero $g\in\cO(Y)$. Note that 
in this case $f_+$ coincides with the usual push-forward functor at the level of $\cO$-modules. We thus see that if $\cM$ is a holonomic $\cD_X$-module,
then $f_+(\cM)=f_*(\cM)$ is a holonomic $\cD_Y$-module. In particular, we see that $\cO_Y[1/g]$ is a holonomic $\cD_Y$-module (note that, a priori, it is not
even clear that this is a coherent $\cD_Y$-module). The key ingredient in proving the preservation of holonomicity by the push-forward via an open immersion
is the notion of $b$-function, that will play an important role also in Section~\ref{sect_Vfil_bfcn} below. 

Another important functor on holonomic $\cD$-modules is provided by duality. More generally, we have a duality functor on the bounded derived category of left $\cD_X$-modules
with coherent cohomology $\cD^b_c(\cD_X)$. Note that for every left $\cD_X$-module $\cM$, we have a right $\cD_X$-module structure on ${\mathcal Hom}_{\cD_X}(\cM,\cD_X)$
induced by the canonical right $\cD_X$-module structure on $\cD_X$, and thus a left $\cD_X$-module structure on ${\mathcal Hom}_{\cD_X}(\cM,\cD_X)\otimes_{\cO_X}\omega_X^{-1}$. We thus obtain a \emph{duality functor}
\begin{equation}\label{eq_def_duality_functor}
\cD^b_c(\cD_X)\to \cD^b_c(\cD_X),\quad \cM\mapsto \cM^*:={\mathbf R}{\mathcal Hom}_{\cD_X}(\cM,\cD_X)\otimes_{\cO_X}\omega_X^{-1}[n],
\end{equation}
where $n=\dim(X)$. 
One can show that there is a natural isomorphism $(\cM^*)^*\simeq\cM$; in particular, the functor $\cM\mapsto\cM^*$ is an anti-equivalence of categories.
The reason for the shift by $n$ in the definition of duality is due to the following two properties:
if $\cM$ is a coherent $\cD_X$-module, then
\begin{enumerate}
\item[i)] ${\mathcal Ext}^i_{\cD_X}(\cM,\cD_X)=0$ for all $i>n$.
\item[ii)] If $\cM\neq 0$, then $\min\{i\geq 0 \mid {\mathcal Ext}_{\cD_X}^i(\cM,\cD_X)\neq 0\}=2n-\dim(\cM)$.
\end{enumerate}
In particular, if $\cM$ is holonomic, we see that ${\mathcal Ext}_{\cD_X}^i(\cM,\cD_X)=0$ for $i\neq n$,
hence $\cM^*$ is a $\cD_X$-module; in fact, one can show that it is a holonomic $\cD_X$-module. We thus get an induced anti-equivalence 
of the abelian category of holonomic $\cD_X$-modules and also an anti-equivalence of $\cD^b_h(\cD_X)$.

\subsection{The Riemann-Hilbert correspondence}\label{section_RH}
We now assume that $k={\mathbf C}$ and write $X^{\rm an}$ for the complex
manifold corresponding to the smooth complex algebraic variety $X$. Recall that a sheaf $\cF$ of ${\mathbf C}$-vector spaces on $X^{\rm an}$ is \emph{constructible} if
there is a finite decomposition $X=\bigsqcup_iX_i$, with each $X_i$ being a locally closed algebraic subvariety,
such that the restriction of $\cF$ to each $X_i$ is a local system (that is, it is a locally constant sheaf of finite rank). 
Let us write $\cD^b_c(X^{\rm an})$ for the bounded derived category of complexes of sheaves of ${\mathbf C}$-vector spaces on $X^{\rm an}$ whose 
cohomology sheaves are constructible. These categories satisfy a 6-functor formalism. In particular, we have the \emph{Verdier duality} functor ${\mathbf D}_{X^{\rm an}}
={\mathbf R}{\mathcal Hom}_{\mathbf C}\big(-,\underline{\mathbf C}_X[2n]\big)$ on $\cD^b_c(X^{\rm an})$ and if $f\colon X\to Y$ is a morphism of smooth, irreducible complex
algebraic varieties, we have exact functors
$${\mathbf R}f_*, {\mathbf R}f_!\colon \cD^b_c(X^{\rm an})\to \cD^b_c(Y^{\rm an})\quad\text{and}\quad 
f^{-1},f^!\colon \cD^b_c(Y^{\rm an})\to \cD^b_c(X^{\rm an})$$
that satisfy certain adjointness properties and compatibilities with duality. 

Beilinson, Bernstein, Deligne, and Gabber defined in \cite{BBDG} an abelian subcategory ${\rm Perv}(X^{\rm an})$ of $\cD^b_c(X^{\rm an})$ whose
objects are known as \emph{perverse sheaves}. 
They are complexes with constructible cohomology that satisfy certain bounds on the dimensions of the supports of the cohomology sheaves
and whose Verdier duals satisfy the same conditions (in particular, the subcategory is preserved by Verdier duality).

If $\cM$ is a left $\cD_X$-module, then we also consider the 
analytification of the de Rham complex ${\rm DR}_X(\cM)$, that we denote by 
${\rm DR}_X^{\rm an}(\cM)$. In his doctoral thesis, Kashiwara proved that ${\rm DR}^{\rm an}_{X}(\cM)\in\cD^b_c(X^{\rm an})$.
In fact, he showed that ${\rm DR}_X^{\rm an}(\cM)$ is a perverse sheaf, though this notion was formalized only later. 

In order to obtain an equivalence of categories one needs to impose one extra condition: \emph{regular singularities}. We do not give the definition since it is somewhat involved and we will not need it. For a thorough discussion, see \cite[Chapter~6]{HTT}. 
If we denote by $\cD^b_{rh}(\cD_X)$ the subcategory of $\cD^b_h(\cD_X)$ whose cohomology sheaves are regular holonomic $\cD_X$-modules, then
$\cD^b_{rh}(\cD_X)$ is preserved by ${\mathbf D}_X$ and 
for a morphism $f\colon X\to Y$ between smooth, irreducible complex algebraic varieties, we have induced functors
$$f_+\colon\cD^b_{rh}(\cD_X)\to \cD^b_{rh}(\cD_Y)\quad\text{and}\quad f^{\dagger}={\mathbf L}f^*[\dim(X)-\dim(Y)]\colon \cD^b_{rh}(\cD_Y)
\to \cD^b_{rh}(\cD_X).$$

The Riemann-Hilbert correspondence, proved independently by Kashiwara and Mebkhout, says that the analytic de Rham functor
induces an equivalence of triangulated categories
$${\rm DR}_X^{\rm an}\colon \cD^b_{rh}(\cD_X)\simeq \cD^b_c(X^{\rm an}).$$
Moreover, via this equivalence, the functor $(-)^*$ defined in (\ref{eq_def_duality_functor}) corresponds to ${\mathbf D}_{X^{\rm an}}$ and if $f\colon X\to Y$ is a morphism
of smooth, irreducible algebraic varieties, then $f_+$ and $f^{\dagger}$ correspond, respectively, to ${\mathbf R}f_*$ and $f^!$. 
Finally, ${\rm DR}_X^{\rm an}$ induces an equivalence between the abelian category of regular holonomic $\cD_X$-modules and 
${\rm Perv}(X^{\rm an})$. For a detailed discussion and for proofs, see \cite[Chapter~7]{HTT}.

\section{The $V$-filtration with respect to a smooth hypersurface}\label{section_Vfiltration_smooth}

Our presentation of $V$-filtrations in this section is based on the one in \cite[Section~3.1]{Saito-MHP}.
Let $X$ be a smooth, irreducible algebraic variety over $k$
and let $t\in\cO_X(X)$ be nonzero, defining a smooth, irreducible hypersurface $H$. 
We put $\dim(X)=n+1$, so $\dim(H)=n$.

\subsection{The filtration $V^{\bullet}\cD_X$}\label{section_VD}
For every $m\in \ZZ$, we put
$$V^m\cD_X=\{P\in\cD_X\mid P\cdot (t)^q\subseteq (t)^{q+m}\,\,\text{for all}\,\,q\in\ZZ\}$$
(with the convention that $(t)^i=\cO_X$ if $i\leq 0$). 
It follows from the definition that we have $V^{m_1}\cD_X\cdot V^{m_2}\cD_X\subseteq V^{m_1+m_2}\cD_X$ for all $m_1,m_2\in\ZZ$.

\begin{rmk}
Note that the decreasing filtration $V^{\bullet}\cD_X$ only depends on $H$, but not on the function $t$ defining $H$.
\end{rmk}

It is easy to describe $V^m\cD_X$ locally. Note first that if $U\subseteq X$ is an open subset such that $H\cap U=\emptyset$, then $V^m\cD_X\vert_U=\cD_U$ for all $m\in\ZZ$.
On the other hand, if $U$ is an affine open subset with algebraic coordinates $x_1,\ldots,x_n,t$, then
given $P\in \Gamma(U,\cD_X)$, we write $P=\sum_{\alpha,j}P_{\alpha,j}\partial_x^{\alpha}\partial_t^j$, where 
the sum is over $\alpha\in\ZZ_{\geq 0}^n$ and $j\in\ZZ_{\geq 0}$, and
all $P_{\alpha,j}$ lie in $\cO_X(U)$. In this case, it is straightforward to see that $P\in\Gamma(U,V^m\cD_X)$ if and only if 
$P_{\alpha,j}\in (t^{m+j})$ for all $\alpha$ and $j$.

Given coordinates, as above, it is convenient to also describe $V^m\cD_X$ in terms of the Euler operator $t\partial_t$. In fact, for reasons that will become clearer later, we will use instead 
$s:=-\partial_tt=-t\partial_t-1$. For future reference, we collect in the next lemma some easy identities involving $s$, $t$, and $\partial_t$.

\begin{lem}\label{lem_s}
The operator $s=-\partial_tt$ satisfies the following properties:
\begin{enumerate}
\item[i)] For every $P\in k[s]$ and every $m\geq 0$, we have
$P(s)t^m=t^mP(s-m)$ and $P(s)\partial_t^m=\partial_t^mP(s+m)$.
\item[ii)] For every $m\in\ZZ_{>0}$, we have
$$t^m\partial_t^m=(-1)^m\prod_{j=1}^m(s+j)\quad\text{and}\quad \partial_t^mt^m=(-1)^m\prod_{j=0}^{m-1}(s-j).$$
\end{enumerate}
\end{lem}

\begin{proof}
All assertions follow easily by induction on $m$. 
\end{proof}

Using the above lemma, we see that $t^i\partial_t^j$ lies in $k[s]\cdot t^{i-j}$ if $i\geq j$ and it lies in $k[s]\cdot\partial_t^{j-i}$ if $i\leq j$.
We thus conclude that, given coordinates $x_1,\ldots,x_n,t$ as above, 
\begin{equation}\label{formula_description_V0}
V^0\cD_U=\cO_U\langle\partial_{x_1},\ldots,\partial_{x_n},s\rangle.
\end{equation}
We also have 
\begin{equation}\label{eq1_V_D}
V^m\cD_U=V^0\cD_U\cdot t^m=t^m\cdot V^0\cD_U\quad\text{for all}\quad m\geq 0\quad\text{and}
\end{equation}
\begin{equation}\label{eq2_V_D}
V^{-m}\cD_U=\sum_{i=0}^mV^0\cD_U\cdot \partial_t^i=\sum_{i=0}^m\partial_t^i\cdot V^0\cD_U\quad\text{for all}\quad m\geq 0.
\end{equation}

\begin{rmk}\label{s_under_change_coord}
Note that the operators $\partial_t$ and $s=-\partial_tt$ depend on the choice of coordinates. 
More precisely, if we choose a different set of local coordinates $x'_1,\ldots,x'_n,t$ and write
$\partial_t'$ for the corresponding derivation with respect to this system of coordinates, we have
$\partial_{t}'=\partial_t+\sum_{i=1}^n\partial'_{t}(x_i)\partial_{x_i}$, hence 
$$\partial'_t-\partial_t\in V^0\cD_X\quad\text{and}-\partial_t't+\partial_tt\in V^1\cD_X.$$
We also note that if we replace $t$ by $t'=ut$, for some invertible function $u$, then a similar computation shows that, with respect to the same coordinates $x_1,\ldots,x_n$,
we have $\partial_{t'}-\partial_tu^{-1}\in V^0\cD_X$ and thus $\partial_{t'}t'-\partial_tt\in V^1\cD_X$.
\end{rmk}

\begin{rmk}\label{V0_Noeth}
Note that all $V^i\cD_X$ are quasi-coherent sheaves of $\cO_X$-modules and for every affine open subset $U\subseteq X$, the ring $\Gamma(U,V^0\cD_X)$ 
is both left and right Noetherian. In fact, if we consider the 
Rees-type sheaf of rings
$${\mathcal R}_+(V^{\bullet}\cD_X):=\bigoplus_{i\geq 0}V^i\cD_Xz^i,$$
then $\Gamma\big(U, {\mathcal R}_+(V^{\bullet}\cD_X)\big)$ is both left and right Noetherian.
This is standard to check, using the filtration induced on ${\mathcal R}_+(V^{\bullet}\cD_X)$ by the order filtration on $\cD_X$ and showing that
the corresponding graded ring is Noetherian.
\end{rmk}

\subsection{The $V$-filtration with respect to $t$}\label{section_def_V}
We now come to the definition of the $V$-filtration. The role of this filtration is, roughly speaking, to put the operator 
$\partial_tt$ on a given $\cD$-module in upper triangular form. 

\begin{defi}\label{defi_V_filtration}
A \emph{weak}\footnote{This is not standard terminology. We will only use it in a few remarks in this section, in order to emphasize the different roles of the conditions in the definition of a $V$-filtration.} \emph{$V$-filtration} on a coherent $\cD_X$-module $\cM$ (with respect to $t\in\cO_X(X)$) is a decreasing filtration 
$V^{\bullet}\cM=(V^{\alpha}\cM)_{\alpha\in\QQ}$ 
by quasi-coherent $\cO_X$-modules, 
indexed by rational numbers, which is exhaustive\footnote{This means that
$\cup_{\alpha\in\QQ}V^{\alpha}\cM=\cM$.}, discrete and left-continuous\footnote{These conditions mean that there is a positive integer $\ell$ such that $V^{\alpha}\cM$ takes a constant value 
for $\alpha$ in an interval of the form $\big(i/\ell,(i+1)/\ell\big]$, where $i$ is any integer.}, and satisfies the following properties:
\begin{enumerate}
\item[i)] We have $V^i\cD_X\cdot V^{\alpha}\cM\subseteq V^{\alpha+i}\cM$ for every $\alpha\in\QQ$ and $i\in\ZZ$.
\item[ii)] For every $\alpha\in\QQ$, the operator $(s+\alpha)$ is nilpotent on ${\rm Gr}_V^{\alpha}(\cM)$.
\end{enumerate}
We say that $V^{\bullet}\cM$ is a \emph{$V$-filtration} if it satisfies, in addition, the following two conditions:
\begin{enumerate}
\item[iii)] Each $V^{\alpha}\cM$ is locally finitely generated over $V^0\cD_X$.
\item[iv)] We have $t\cdot V^{\alpha}\cM=V^{\alpha+1}\cM$ for all $\alpha>0$.
\end{enumerate}
\end{defi}

Note that, in the above definition, we put ${\rm Gr}_V^{\alpha}(\cM):=V^{\alpha}\cM/V^{>\alpha}\cM$, where 
$$V^{>\alpha}\cM=\bigcup_{\beta>\alpha}V^{\beta}\cM=V^{\alpha+\epsilon}\cM\quad\text{for}\quad 0<\epsilon\ll 1.$$
By assumption, there is a positive integer $\ell$ such that ${\rm Gr}_V^{\alpha}(\cM)=0$ unless $\ell\alpha\in\ZZ$
(we will refer to this property by saying that the filtration is \emph{discrete}).

\begin{rmk}[Restriction of $V$-filtration to an open subset]\label{restriction_Vfilt_open_subset}
It is clear from the definition that if $V^{\bullet}\cM$ is a (weak) $V$-filtration on $\cM$ with respect to $t$, then for every open subset
$U\subseteq X$, the restriction $V^{\bullet}\cM\vert_U$ is a (weak) $V$-filtration on $\cM\vert_U$ with respect to $t\vert_U$. 
\end{rmk}

\begin{rmk}[Behavior on $X\smallsetminus H$]\label{D_mod_str_V0}
If $V^{\bullet}\cM$ is a weak $V$-filtration on $\cM$, then
it follows from condition i) in Definition~\ref{defi_V_filtration} that for every $\alpha\in\QQ$, the quotient $V^{\alpha}\cM/V^{\alpha+1}\cM$ is annihilated by $(t)$;
in fact, it is naturally a $\cD_H$-module. In particular, each ${\rm Gr}^{\alpha}_V(\cM)$ is naturally a $\cD_H$-module.
Since the filtration is exhaustive and discrete, and ${\rm Gr}^{\alpha}_V(\cM)\vert_{X\smallsetminus H}=0$ for all $\alpha$, we see that $V^{\alpha}\cM\vert_{X\smallsetminus H}=\cM\vert_{X\smallsetminus H}$
for all $\alpha\in\QQ$.
\end{rmk}

\begin{rmk}[Independence of coordinates and of the equation $t$]\label{remark_Vfil_invertible_multiple}
Since ${\rm Gr}_V^{\alpha}(\cM)$ is supported on $H$, we only put the condition in ii) in Definition~\ref{defi_V_filtration} at the points in $H$; 
note that the operator $s$ is defined around these points. While this operator is only defined locally and depends on the choice of local coordinates,
its action on ${\rm Gr}_V^{\alpha}(\cM)$ is well-defined: indeed,
as we have noticed in Remark~\ref{s_under_change_coord}, if $s'$ is the same operator corresponding to a different choice of coordinates, then
$s-s'\in V^1\cD_X$, hence the actions of $s$ and $s'$ on ${\rm Gr}_V^{\alpha}(\cM)$ agree by condition i). 

Furthermore, it follows from the same remark that if we replace $t$ by $ut$, for some invertible function $u$, and if we denote by $s'$ the new operator, then
the actions of $s$ and $s'$ on ${\rm Gr}_V^{\alpha}(\cM)$ again coincide. Therefore the notion of $V$-filtration on $\cM$ only depends on $H$ and it does not depend on the choice of function $t$.
\end{rmk}

\begin{rmk}
In the presence of local coordinates $x_1,\ldots,x_n,t$, it follows from the formulas (\ref{formula_description_V0}), (\ref{eq1_V_D}) and (\ref{eq2_V_D}) that the condition in i) in Definition~\ref{defi_V_filtration} is equivalent to the fact 
that each $V^{\alpha}\cM$ is preserved by the action of $\cO_X\langle\partial_{x_1},\ldots,\partial_{x_n}\rangle$ and 
$$t\cdot V^{\alpha}\cM\subseteq V^{\alpha+1}\cM\quad\text{and}\quad \partial_t\cdot V^{\alpha}\cM\subseteq V^{\alpha-1}\cM\quad\text{for all}\quad \alpha\in\QQ.$$
\end{rmk}

\begin{rmk}[Coherence of ${\rm Gr}_V^{\alpha}(\cM)$]\label{D_mod_str_V}
If $V^{\bullet}\cM$ is a $V$-filtration on $\cM$, then
it follows from condition iii) in the definition that each ${\rm Gr}_V^{\alpha}(\cM)$
 is locally finitely generated over $\cD_H[s]$. Since $(s+\alpha)$ is nilpotent by condition ii) in the definition, we conclude that ${\rm Gr}_V^{\alpha}(\cM)$ is a coherent
$\cD_H$-module.
\end{rmk}

\begin{prop}\label{rmk1_Vfiltration}
If $V^{\bullet}\cM$ is a weak $V$-filtration on the coherent $\cD_X$-module $\cM$, then for every $\alpha\neq 0$, the maps
\begin{equation}\label{eq_rmk1_Vfiltration}
{\rm Gr}_V^{\alpha}(\cM)\overset{t\cdot}\longrightarrow {\rm Gr}_V^{\alpha+1}(\cM)\overset{\partial_t\cdot}\longrightarrow{\rm Gr}^{\alpha}_V(\cM)
\end{equation}
are isomorphisms of $\cD_H$-modules.
\end{prop}

\begin{proof}
Let $\nu_t$ and $\nu_{\partial_t}$ be the first, respectively the second, map in (\ref{eq_rmk1_Vfiltration}). 
It is clear that they are morphisms of $\cD_H$-modules, hence we only need to show that they are bijective.
It follows from condition ii) in Definition~\ref{defi_V_filtration} that since $\alpha\neq 0$, the composition $\nu_{\partial_t}\circ\nu_{t}$ is invertible,
hence $\nu_t$ is injective and $\nu_{\partial_t}$ is surjective. On the other hand, since $t\partial_t=\partial_tt-1$ and $\alpha+1\neq 1$,
it follows from condition ii) in Definition~\ref{defi_V_filtration} that $\nu_t\circ\nu_{\partial_t}$ is invertible, hence $\nu_t$ is surjective and $\nu_{\partial_t}$
is injective. Therefore both $\nu_t$ and $\nu_{\partial_t}$ are bijective.
\end{proof}

It follows from the above proposition that the interesting maps are the following morphisms of $\cD_H$-modules:
\begin{equation} \label{defn-can} {\rm can}\colon {\rm Gr}_V^1(\cM)\overset{\partial_t\cdot}\longrightarrow {\rm Gr}_V^0(\cM)\quad\text{and}\end{equation}
\begin{equation} \label{defn-Var}{\rm Var}\colon {\rm Gr}_V^0(\cM)\overset{t\cdot}\longrightarrow {\rm Gr}_V^1(\cM)\end{equation}
(the notation for these two maps is justified by the connection to topological vanishing and nearby cycles, see Remark~\ref{remark_nearby_cycles} below).

\begin{cor}\label{cor_rmk1_Vfiltration}
If $V^{\bullet}\cM$ is a weak $V$-filtration on $\cM$, then the following hold:
\item[i)] If $\alpha\leq 0$ and $u\in\cM$ is such that $tu\in V^{\alpha+1}\cM$, then $u\in V^{\alpha}\cM$. In particular, we have 
$$\{u\in\cM\mid tu=0\}\subseteq V^0\cM.$$
\item[ii)] If $\alpha\leq 1$ and $v\in\cM$ is such that $\partial_tv\in V^{\alpha-1}\cM$, then $v\in V^{\alpha}\cM$.
\item[iii)] For every positive integer $m$ and every $\alpha\in\QQ$ such that
$-m\leq \alpha<-m+1$, we have
\begin{equation}\label{eq_cor_rmk1_Vfiltration}
V^{\alpha}\cM=\partial_t^m\cdot V^{\alpha+m}\cM+\sum_{j=0}^{m-1}\partial_t^j\cdot V^0\cM \quad\text{if}\quad -m\leq \alpha<-m+1.
\end{equation}
\end{cor}

\begin{proof}
In the setting of i), since the $V$-filtration is exhaustive, it follows that there is $\beta\in\QQ$ such that $u\in V^{\beta}\cM$. If $\beta\geq \alpha$, then we are done.
Otherwise, since $tu\in V^{>\beta+1}$ and $\beta+1<\alpha+1\leq 1$, it follows from Proposition~\ref{rmk1_Vfiltration} that $u\in V^{>\beta}\cM$. Since the $V$-filtration
is discrete, after repeating this finitely many steps, we conclude that $u\in V^{\alpha}\cM$. The second assertion in i) follows by taking $\alpha=0$. 

The argument for ii) is similar: let $\gamma$ be such that $v\in V^{\gamma}\cM$. If $\gamma\geq \alpha$, then we are done. Otherwise,
since $\partial_tv\in V^{>\gamma-1}\cM$ and $\gamma-1<\alpha-1\leq 0$, it follows from Proposition~\ref{rmk1_Vfiltration} that $v\in V^{>\gamma}\cM$.
After repeating this finitely many steps, we conclude that $v\in V^{\alpha}\cM$. 

For the assertion in iii), we only need to prove the inclusion ``$\subseteq$". Since $\alpha<0$, for every $w\in V^{\alpha}\cM$,
 it follows from Proposition~\ref{rmk1_Vfiltration} that we can write
$w=\partial_tw'+w''$, where $w'\in V^{\alpha+1}\cM$ and $w''\in V^{\beta}\cM$, for some $\beta>\alpha$. 
Since the filtration is discrete, we can iterate this argument to see that $w\in \partial_t\cdot V^{\alpha+1}\cM+V^{\alpha+1}\cM$ if $\alpha\leq -1$ 
and $w\in \partial_t\cdot V^{\alpha+1}\cM+V^0\cM$ if $-1<\alpha<0$. The formula in (\ref{eq_cor_rmk1_Vfiltration}) then follows by an easy induction on $m$.
\end{proof}

Note that the formula in (\ref{eq_cor_rmk1_Vfiltration}) gives an explicit description of $V^{\alpha}\cM$ for $\alpha<0$ in terms of $V^{\beta}\cM$, for $\beta\in [0,1)$.

\begin{rmk}
If $V^{\bullet}\cM$ is a weak filtration on the coherent $\cD_X$-module $\cM$, then the following are equivalent:
\begin{enumerate}
\item[a)] $V^{\bullet}\cM$ is a $V$-filtration.
\item[b)] The $V^0\cD_X$-module $V^0\cM$ is locally finitely generated and $t\cdot V^i\cM=V^{i+1}\cM$ for $i\in\ZZ$, $i\gg 0$.
\item[c)] The ${\mathcal R}_+(V^{\bullet}\cD_X)$-module ${\mathcal R}_+(V^{\bullet}\cM):=\bigoplus_{i\geq 0}V^{i}\cM z^i$ is locally finitely generated.
\end{enumerate}
We leave the proof as an exercise for the reader, since we will not need this fact.
\end{rmk}

\begin{prop}\label{unique_V_filt}
A coherent $\cD_X$-module $\cM$ admits at most one $V$-filtration with respect to $t$.
\end{prop}

\begin{proof}
Suppose that $V^{\bullet}\cM$ and $\widetilde{V}^{\bullet}\cM$ are two $V$-filtrations on $\cM$. By symmetry, it is enough to show that we have
\begin{equation}\label{eq1_unique_V_filt}
V^{\alpha}\cM\subseteq \widetilde{V}^{\alpha}\cM\quad\text{for all}\quad\alpha\in\QQ.
\end{equation}
This is clear on $X\smallsetminus H$, hence we only need to check it around the points in $H$.
Note first that
\begin{equation}\label{eq2_unique_V_filt}
\frac{V^{\beta}\cM\cap \widetilde{V}^{\gamma}\cM}{(V^{>\beta}\cM\cap \widetilde{V}^{\gamma}\cM)+(V^{\beta}\cM\cap \widetilde{V}^{>\gamma}\cM)}
\end{equation}
is a subquotient of both ${\rm Gr}_V^{\beta}(\cM)$ and ${\rm Gr}^{\gamma}_{\widetilde{V}}(\cM)$. Property iv) in Definition~\ref{defi_V_filtration} thus implies that
both $s+\beta$ and $s+\gamma$ are nilpotent on (\ref{eq2_unique_V_filt}), and thus (\ref{eq2_unique_V_filt}) is $0$ for all 
$\beta,\gamma\in\QQ$, with $\beta\neq\gamma$.

Let us fix now $\alpha\in\QQ$. Note first that since the filtration $\widetilde{V}^{\bullet}\cM$ is an exhaustive filtration by $V^0\cD_X$-submodules
and $V^{\alpha}\cM$ is locally finitely generated over $V^0\cD_X$, there is $\gamma\in\QQ$ such that 
\begin{equation}\label{eq3_unique_V_filt}
V^{\alpha}\cM\subseteq \widetilde{V}^{\gamma}\cM.
\end{equation}
If $\gamma\geq \alpha$, then (\ref{eq1_unique_V_filt}) holds and we are done, hence we may and will assume that $\gamma<\alpha$ and that 
$\gamma$ is maximal such that (\ref{eq3_unique_V_filt}) holds, aiming for a contradiction (we may assume this since the filtration $\widetilde{V}^{\bullet}\cM$ is discrete and left continuous). 
On the other hand, since $V^{\bullet}\cM$ satisfies property iii), it follows that there is $m_0\geq 0$ such that for every $m\in\ZZ_{>0}$, we have
$$V^{\alpha+m_0+m}\cM\subseteq t^m\cdot V^{\alpha+m_0}\cM\subseteq t^m\cdot \widetilde{V}^{\gamma}\cM\subseteq \widetilde{V}^{>\gamma}\cM.$$
Since $V^{\alpha}\cM\not\subseteq \widetilde{V}^{>\gamma}\cM$ and the filtration $V^{\bullet}\cM$ is discrete and left continuous, it follows that there is $\beta\geq\alpha$ such that 
\begin{equation}\label{eq4_unique_V_filt}
V^{\beta}\cM\not\subseteq \widetilde{V}^{>\gamma}\cM\quad\text{and}\quad 
\end{equation}
\begin{equation}\label{eq5_unique_V_filt}
V^{>\beta}\cM\subseteq \widetilde{V}^{>\gamma}\cM.
\end{equation}
Note that $\gamma<\beta$, hence the quotient
 (\ref{eq2_unique_V_filt}) for $\beta$ and $\gamma$ is $0$, and therefore
$$V^{\beta}\cM\subseteq V^{\beta}\cM\cap \widetilde{V}^{\gamma}\cM\subseteq (V^{>\beta}\cM\cap \widetilde{V}^{\gamma}\cM)+(V^{\beta}\cM\cap V^{>\gamma}\cM)
\subseteq V^{>\gamma}\cM,$$
where the first inclusion follows from (\ref{eq3_unique_V_filt}), since $\beta\geq\alpha$, and the last inclusion follows from (\ref{eq5_unique_V_filt}). However, this contradicts 
(\ref{eq4_unique_V_filt}), and thus completes the proof of the proposition.
\end{proof}

\begin{rmk}
The proof of Proposition~\ref{unique_V_filt} shows that if $V^{\bullet}\cM$ is a $V$-filtration on $\cM$ and $\widetilde{V}^{\bullet}\cM$ is a weak $V$-filtration on $\cM$, then
$V^{\alpha}\cM\subseteq\widetilde{V}^{\alpha}\cM$ for all $\alpha\in\QQ$.
\end{rmk}

\begin{cor}\label{cor_unique_V_filt}
Given a coherent $\cD_X$-module $\cM$ and an open cover $X=\bigcup_iU_i$, then $\cM$ has a $V$-filtration with respect to $t$ if and only if the same holds for each
$\cM\vert_{U_i}$ with respect to $t\vert_{U_i}$.
\end{cor}

\begin{proof}
Given a filtration $V^{\bullet}\cM$ on $\cM$ it is clear that this is a $V$-filtration if and only if it induces a $V$-filtration on each $U_i$. The assertion in the corollary now follows from the fact that
by the proposition, given $V$-filtrations for each $\cM\vert_{U_i}$, these glue to a filtration on $\cM$.
\end{proof}

We next give two basic examples of explicit $V$-filtrations with respect to $t$.

\begin{eg}[$V$-filtration on a vector bundle with integrable connection]\label{eg_V_filtration_O}
If $\cM$ is a $\cD_X$-module that is coherent as an $\cO_X$-module (hence locally free), then $\cM$ has a $V$-filtration given by 
$V^{\alpha}\cM=(t)^{\lceil\alpha\rceil-1}\cdot\cM$ for all $\alpha\in\QQ$, with the convention that $(t)^m=\cO_X$ if $m\leq 0$. It is straightforward to check that the conditions in Definition~\ref{defi_V_filtration}
are satisfied. We only note that each $V^{\alpha}\cM$ is clearly locally finitely generated over $V^0\cD_X$ since it is coherent as an $\cO_X$-module. 
Also, condition ii) in the definition holds since if $m\in\ZZ_{>0}$, then $(\partial_tt-m)t^{m-1}u=t^m\partial_tu\in V^{m+1}\cM$, hence 
$(s+m)\cdot {\rm Gr}_V^m(\cM)=0$.
\end{eg}

\begin{eg}[$V$-filtration on a $\cD_X$-module supported on $H$]\label{V_filt_support_H}
If $\cN$ is a coherent $\cD_X$-module supported on $H$, then $\cN$ has a $V$-filtration such that $V^{\alpha}\cN=0$ for all $\alpha>0$
and ${\rm Gr}_V^{\alpha}(\cN)=0$ for all $\alpha\in\QQ_{<0}\smallsetminus\ZZ$. Indeed, note first that by Corollary~\ref{cor_unique_V_filt}, in order to check this,
we may and will assume that $X$ is affine and we have coordinates $x_1,\ldots,x_n,t$ on $X$ such that $H$ is defined by $(t)$. 
It is easy to check that if $\cN_i=\{u\in\cN\mid (s+i)u=0\}$, then the following hold
(these facts are steps in the proof of Kashiwara's Equivalence Theorem, see \cite[Theorem~1.6.1]{HTT}):
\begin{enumerate}
\item[i)] $t\cdot \cN_i\subseteq \cN_{i+1}$ and $\partial_t\cdot \cN_i\subseteq \cN_{i-1}$ for all $i\in\ZZ$.
\item[ii)] $\cN_i=0$ for $i>0$ and $\cN=\bigoplus_{i\leq 0}\cN_i$.
\item[iii)] $\cN_{-m}=\partial_t^m\cdot\cN_0$ for all $m\geq 0$.
\end{enumerate}
Since $\cN$ is a coherent $\cD_X$-module, it follows that $\cN_0$ is a coherent $\cD_H$-module, and thus by iii) above, we see that
every $\bigoplus_{i=0}^m\cN_{-i}$ is a finitely generated $V^0\cD_X$-module for every $m\geq 0$. It is then clear that we get a $V$-filtration on $\cN$
by putting
$$V^{\alpha}\cN=\bigoplus_{i=0}^{-\lceil\alpha\rceil}\cN_{-i}\quad\text{for}\quad \alpha\leq 0$$
and $V^{\alpha}\cN=0$ for $\alpha>0$. 
\end{eg}

\begin{prop}\label{properties_V_filtration1}
Let $\cM$ be a coherent $\cD_X$-module that has a $V$-filtration $V^{\bullet}\cM$ with respect to $t$.
\begin{enumerate}
\item[i)] For every $\alpha>0$, the map $V^{\alpha}\cM\overset{t\cdot}\longrightarrow V^{\alpha+1}\cM$ is bijective.
\item[ii)] If $\cM'$ is a $\cD_X$-submodule of $\cM$ and $p\colon \cM\to \cM''=\cM/\cM'$ is the canonical projection, then $\cM'$ and $\cM''$ have $V$-filtrations
with respect to $t$ given by
\begin{equation}\label{eq_properties_V_filtration1}
V^{\alpha}\cM'=V^{\alpha}\cM\cap\cM'\quad\text{and}\quad V^{\alpha}\cM''=p(V^{\alpha}\cM)\quad\text{for all}\quad \alpha\in\QQ.
\end{equation}
\end{enumerate}
\end{prop}

\begin{proof}
Suppose first that we are in the setting of ii) and we define the filtrations $V^{\bullet}\cM'$ and $V^{\bullet}\cM''$ as in (\ref{eq_properties_V_filtration1}). 
Note that for every $\alpha\in\QQ$, we have a short exact sequence
$$0\to {\rm Gr}_V^{\alpha}(\cM')\to {\rm Gr}^{\alpha}_V(\cM)\to {\rm Gr}^{\alpha}_V(\cM'')\to 0.$$
It is straightforward to check that $V^{\bullet}\cM''$ is a $V$-filtration on $\cM''$. It is also clear that $V^{\bullet}\cM'$ satisfies all the conditions to make 
it a $V$-filtration on $\cM'$ with the exception of condition iv) in Definition~\ref{defi_V_filtration} (the fact that each $V^{\alpha}\cM'$ is locally finitely generated over 
$V^0\cD_X$ follows from the fact that $V^{\alpha}\cM$ has this property and Remark~\ref{V0_Noeth}). However, this condition also follows if 
we know that if $u\in V^{\alpha}\cM$, with $\alpha>0$, is such that $tu\subseteq \cM'$, then $u\in\cM'$; in other words, it is enough to know that 
$\cM''$ satisfies assertion i) in the proposition.

We now prove that every $\cD_X$-module $\cM$ that has a $V$-filtration satisfies the assertion in i). 
We consider the short exact sequence
$$0\to \cM_1\to \cM\to \cM_2\to 0,$$
such that $\cM_1=\Gamma_H(\cM)$. Note that in this case it is clear that if $tu\in\cM_1$, then $u\in\cM_1$. The previous discussion thus implies that 
the $V$-filtration on $\cM$ induces $V$-filtrations on $\cM_1$ and $\cM_2$. On the other hand, since $\cM_1$ is supported on $H$, it follows from the uniqueness
of the $V$-filtration and Example~\ref{V_filt_support_H}
that $V^{\alpha}\cM\cap \cM_1=0$ for all $\alpha>0$. Therefore $V^{\alpha}\cM\simeq V^{\alpha}\cM_2$ for all $\alpha>0$, and since multiplication by $t$ is injective on $\cM_2$,
we see that multiplication by $t$ is injective on $V^{\alpha}\cM$ for all $\alpha>0$. This completes the proof of the proposition.
\end{proof}

\begin{cor}\label{cor_properties_V_filtration1}
It $\varphi\colon \cM\to \cN$ is a morphism of coherent $\cD_X$-modules such that both $\cM$ and $\cN$ have $V$-filtrations with respect to $t$, then
$\varphi$ is a filtered morphism and it is strict\footnote{The fact that $\varphi$ is a filtered morphism means that $\varphi(V^{\alpha}\cM)\subseteq V^{\alpha}\cN$
for all $\alpha\in\QQ$. The fact that it is strict means that, in addition, the filtration on $\varphi(\cM)$ induced from $\cM$ is the same as the one induced from $\cN$, that is,
$\varphi(V^{\alpha}\cM)=\varphi(\cM)\cap V^{\alpha}\cN$ for all $\alpha\in\QQ$.}. In particular, the category of coherent $\cD_X$-modules that carry a $V$-filtration
with respect to $t$ is an abelian category.
\end{cor}

\begin{proof}
It follows from assertion ii) in Proposition~\ref{properties_V_filtration1} that if we put
$$V^{\alpha}\varphi(\cM)=\varphi(V^{\alpha}\cM)\quad\text{and}\quad \widetilde{V}^{\alpha}\varphi(\cM)=V^{\alpha}\cN\cap \varphi(\cM),$$
then both these give a $V$-filtration on $\varphi(\cM)$. Therefore they agree by uniqueness of the $V$-filtration, hence $\varphi$ is a filtered morphism and it is strict. 
The last assertion in the corollary is an immediate consequence of this.
\end{proof}

The following proposition describes the behavior of $V$-filtration with respect to $\cD$-module push-forward via closed immersions.

\begin{prop}\label{prop_Vfil_immersion}
Let $i\colon Z\hookrightarrow X$ be the inclusion map of a smooth, irreducible, closed subvariety such that $t\vert_Z$ defines a smooth, irreducible hypersurface in $Z$.
If $\cN$ is a coherent $\cD_Z$-module and $\cM=i_+(\cN)$, then $\cN$ has a $V$-filtration with respect to $t\vert_Z$ if and only if $\cM$ has a $V$-filtration with respect to $t$.
Moreover, in this case, if $x_1,\ldots, x_m,y_1,\ldots,y_r,t$ are local coordinates on $X$ such that $Z$ is defined by $(y_1,\ldots,y_r)$, so 
$\cM\simeq\cN\otimes_kk[\partial_{y_1},\ldots,\partial_{y_r}]$, then 
$$V^{\alpha}\cM\simeq V^{\alpha}\cN\otimes_kk[\partial_{y_1},\ldots,\partial_{y_r}]\quad\text{for all}\quad \alpha\in\QQ.$$
\end{prop}

\begin{proof}
If $\cM$ has a $V$-filtration with respect to $t$, then by assumption each $V^{\alpha}\cM$ is preserved by the action of $y_i\partial_{y_i}$ for all $i$. 
This implies that if $\sum_{\beta}u_{\beta}\otimes \partial_y^{\beta}\in V^{\alpha}\cM$, then $u_{\beta}\otimes \partial_y^{\beta}\in V^{\alpha}\cM$ for all $\beta$. 
It is then straightforward to see that if we put $V^{\alpha}\cN=\{u\mid u\otimes 1\in V^{\alpha}\cM\}$ for all $\alpha$, then $V^{\bullet}\cN$ is a $V$-filtration on $\cN$ with respect to 
$t\vert_Z$. Conversely, if $\cN$ has a $V$-filtration with respect to $t\vert_Z$ and we put $V^{\alpha}\cM=V^{\alpha}\cN\otimes_kk[\partial_{y_1},\ldots,\partial_{y_r}]$,
it is straightforward to see that this is a $V$-filtration of $\cM$ with respect to $t$.
\end{proof}

\subsection{The maps ${\rm Var}$ and ${\rm can}$}\label{section_Var}
We keep the same notation as before.

\begin{prop}\label{prop_descr_Var}
If $\cM$ has a $V$-filtration with respect to $t$ and $i\colon H\hookrightarrow X$ is the inclusion, then
${\mathbf L}i^*(\cM)[-1]$ is computed by the complex 
$${\rm Gr}_V^0(\cM)\overset{t\cdot}\longrightarrow {\rm Gr}_V^1(\cM),$$
placed in cohomological degrees $0$ and $1$. 
\end{prop}

\begin{proof}
Consider the following commutative diagram with exact rows:
$$\xymatrix{0\ar[r] & V^{>0}\cM\ar[d]_{t\cdot}\ar[r] & V^0\cM\ar[d]_{t\cdot}\ar[r] &  {\rm Gr}_V^0(\cM)\ar[d]_{t\cdot}\ar[r] & 0\\
0\ar[r] & V^{>1}\cM\ar[r] & V^1\cM\ar[r] &  {\rm Gr}_V^1(\cM)\ar[r] & 0.}
$$
Since the first vertical map is an isomorphism by Proposition~\ref{properties_V_filtration1}i), it follows that the second and third columns 
are quasi-isomorphic. Therefore the first assertion in the proposition follows if we show that the inclusion of complexes
\begin{equation}\label{comm_diag_Vfilt}
\xymatrix{V^0\cM\ar[d]_{t\cdot}\ar[r] & \cM\ar[d]^{t\cdot}\\
V^1\cM\ar[r] & \cM}
\end{equation}
is an isomorphism. 

We first show that the induced morphism
$\sigma\colon V^1\cM/t\cdot V^0\cM\to \cM/t\cM$ is an isomorphism.
It follows from Proposition~\ref{rmk1_Vfiltration} that if $u\in V^{\alpha}\cM$, with $\alpha<1$, then there is
$u'\in V^{>\alpha}$ such that $u-u'\in t\cM$. Since the $V$-filtration is discrete, after iterating this finitely many times, we see that $\overline{u}\in \cM/t\cM$
lies in the image of $\sigma$. In order to prove that $\sigma$ is injective, consider $u\in V^1\cM\cap t\cM$. If we write $u=tv$,
then it follows from Corollary~\ref{cor_rmk1_Vfiltration}i) that 
$v\in V^0\cM$, hence $u\in t\cdot V^0\cM$. We have thus proved that the induced morphism between the cokernels of the vertical maps in (\ref{comm_diag_Vfilt})
is an isomorphism.

Finally, we need to show that the induced map
\begin{equation}\label{eq10_prop_descr_Var}
\{u\in V^0\cM\mid tu=0\}\to \{u\in\cM\mid tu=0\}
\end{equation}
is an isomorphism.
This is clearly injective and surjectivity follows from 
Corollary~\ref{cor_rmk1_Vfiltration}i) 
\end{proof}

\begin{cor}\label{cor_prop_descr_Var}
If $V^{\bullet}\cM$ is a $V$-filtration on $\cM$ and $\cM$ has no $t$-torsion, then for every $u\in\cM$ and every $\alpha\in\QQ$, we have $u\in V^{\alpha}\cM$
if and only if $tu\in V^{\alpha+1}\cM$. In particular, if the action of $t$ on $\cM$ is invertible, then $V^{\alpha+1}\cM=t\cdot V^{\alpha}\cM$ for all $\alpha\in\QQ$.
\end{cor}

\begin{proof}
For every $\beta\in \QQ$, the map
$${\rm Gr}_V^{\beta}(\cM)\overset{t\cdot}\longrightarrow {\rm Gr}_V^{\beta+1}(\cM)$$
is injective. Indeed, for $\beta\neq 0$ this follows from Proposition~\ref{rmk1_Vfiltration},
while for $\beta=0$ it follows from the fact that $\cM$ has no $t$-torsion by Proposition~\ref{prop_descr_Var}.
This implies that if $u\in V^{\beta}\cM$ is such that $tu\in V^{\alpha+1}\cM$ for some $\alpha>\beta$, then $u\in V^{>\beta}\cM$.
We thus get the first assertion in the corollary using the fact that the $V$-filtration is discrete. The last assertion in the statement
is an immediate consequence.
\end{proof}

\begin{prop}\label{prop_descr_can}
If the coherent $\cD_X$-module $\cM$ has a $V$-filtration with respect to $t$ and $i\colon H\hookrightarrow X$ is the inclusion, then the following hold:
\begin{enumerate}
\item[i)] The $\cD_X$-submodule $\cM':=\cD_X\cdot V^{>0}\cM$ is the smallest $\cD_X$-submodule of $\cM$ with the property that $\cM/\cM'$ is supported on $H$.
\item[ii)] We have an isomorphism of $\cD_H$-modules
$$\cM/\cM'\simeq i_+{\rm Coker}\big({\rm Gr}_V^1(\cM)\overset{\partial_t\cdot}\longrightarrow {\rm Gr}_V^0(\cM)\big).$$
\end{enumerate}
\end{prop}

\begin{proof}
Given any $u\in\cM$, if $u\in V^{\alpha}\cM$ and $m+\alpha>0$, then $t^mu\in V^{>0}\cM$. Therefore $\cM/\cM'$ is supported on $H$.
On the other hand, if $\cN$ is any $\cD_X$-submodule of $\cM$ such that $\cM/\cN$ is supported on $H$, then it follows from Example~\ref{V_filt_support_H}
and Proposition~\ref{properties_V_filtration1} that $V^{>0}\cM=V^{>0}\cN\subseteq\cN$, hence $\cM'\subseteq\cN$. 

In order to prove the isomorphism in ii), note first that since $\cM/\cM'$ is supported on $H$, it follows from 
Kashiwara's Equivalence Theorem (see \cite[Theorem~1.6.1]{HTT})  and Example~\ref{V_filt_support_H} that
$$\cM/\cM'\simeq i_+i^{\dagger}(\cM/\cM')\simeq i_+{\rm Gr}_V^0(\cM/\cM')\simeq i_+\big(V^0\cM/(V^{>0}\cM+(V^0\cM\cap\cM'))\big).$$
The assertion in ii) thus follows if we show that
\begin{equation}\label{eq2_eq_prop_descr_can}
V^{>0}\cM+(V^0\cM\cap\cM')=V^{>0}\cM+\partial_t\cdot V^1\cM
\end{equation}
(we may and will assume that we are in an open subset of $X$ where we have coordinates $x_1,\ldots,x_n,t$ on $X$, so that
we have the operator $\partial_t$ acting on $\cM$). The inclusion ``$\supseteq$" in (\ref{eq2_eq_prop_descr_can}) is clear, hence we only
need to prove the reverse inclusion. 
Note that since $V^{>0}\cM$ is a $V^0\cD_X$-submodule of $\cM$, it follows that 
$$
\cM'=\sum_{m\geq 0}\partial_t^m\cdot V^{>0}\cM.
$$
Suppose now that $u\in V^0\cM\cap\cM'$. We can write $u=\sum_{i=0}^N\partial_t^iu_i$, with all $u_i\in V^{>0}\cM$. 
If $w=\sum_{i=1}^N\partial_t^{i-1}u_i$, then $\partial_tw=u-u_0\in V^0\cM$, hence $w\in V^1\cM$ by
Corollary~\ref{cor_rmk1_Vfiltration}ii). Therefore $u=u_0+\partial_tw\in V^{>0}\cM+\partial_t\cdot V^1\cM$, which completes the proof of the proposition.
\end{proof}

\section{The $V$-filtration with respect to an arbitrary hypersurface}\label{section_Vfilt_general}

We next discuss the case when the hypersurface is not necessarily smooth. The idea, due to Kashiwara, is to construct the $V$-filtration on the push-forward 
of a given $\cD_X$-module via the graph embedding.
In this section we assume that $X$ is a smooth, irreducible, $n$-dimensional variety, and $f\in\cO_X(X)$ is nonzero, defining the
(possibly singular) hypersurface $H$ in $X$.

We consider the smooth, irreducible variety $X\times\AA^1$, of dimension $n+1$. We denote by $t$ the coordinate on $\AA^1$, defining
the smooth hypersurface $X\times\{0\}$. Consider the closed immersion $\iota=\iota_f\colon X\hookrightarrow X\times\AA^1$ given by $\iota(x)=\big(x,f(x)\big)$. 
Given a coherent $\cD_X$-module $\cM$, we consider the coherent $\cD_{X\times\AA^1}$-module $\iota_+\cM$. With a slight abuse of terminology, we will
say that \emph{$\cM$ has a $V$-filtration with respect to $f$} if $\iota_+(\cM)$ has a $V$-filtration with respect to $t$. 

In fact, we prefer to work with sheaves on $X$, rather than on $X\times\AA^1$. In other words, we will tacitly identify a quasi-coherent $\cD_{X\times\AA^1}$-module $\cF$
with the quasi-coherent $\cD_X\langle t,\partial_t\rangle$-module $p_*(\cF)$, where $p\colon X\times\AA^1\to X$ is the projection onto the first component.

Note that the formulas (\ref{formula_description_V0}), (\ref{eq1_V_D}), and (\ref{eq2_V_D}) become in our setting
\begin{equation}\label{new_formula_description_V0}
V^0\cD_{X\times\AA^1}=\cD_X\langle s,t\rangle\subseteq\cD_X\langle t,\partial_t\rangle,
\end{equation}
\begin{equation}\label{new_eq1_V_D}
V^m\cD_{X\times\AA^1}=V^0\cD_{X\times\AA^1}\cdot t^m=t^m\cdot V^0\cD_{X\times\AA^1}\quad\text{for all}\quad m\geq 0\quad\text{and}
\end{equation}
\begin{equation}\label{new_eq2_V_D}
V^{-m}\cD_{X\times\AA^1}=\sum_{i=0}^mV^0\cD_{X\times\AA^1}\cdot \partial_t^i\quad\text{for all}\quad m\geq 0.
\end{equation}

Let us describe explicitly $\iota_+(\cM)$. We first consider the case when $\cM=\cO_X$. 
Note that it follows from (\ref{eq_local_coho}) that we have an isomorphism
\begin{equation}\label{eq_descr1_Bf}
B_f:=\iota_+(\cO_X)\simeq \cH^1_{\iota(X)}(\cO_{X\times {\mathbf A}^1})\simeq\cO_X[t]_{f-t}/\cO_X[t]
\end{equation}
(recall that we view this as a sheaf of $\cD_X\langle t,\partial_t\rangle$-modules on $X$).

We will denote by $\delta$ the class of $\tfrac{1}{f-t}$ in $B_f$.
Since $\cO_X[t]_{f-t}/\cO_X[t]=\bigoplus_{j\geq 1}\cO_X\tfrac{1}{(f-t)^j}$ and $\tfrac{(j-1)!}{(f-t)^j}=\partial_t^{j-1}\cdot \tfrac{1}{(f-t)}$, 
we conclude that
$$B_f=\bigoplus_{j\geq 0}\cO_X\partial_t^j\delta.$$
The action of $\cO_X$ and $\partial_t$ with respect to this decomposition is clear, while the actions of ${\mathcal Der}_k(\cO_X)$ and of $t$ are given by
\begin{equation}\label{action_on_Bf}
D\cdot h\partial_t^j\delta=D(h)\partial_t^j\delta-hD(f)\partial_t^{j+1}\delta\quad\text{and}\quad t\cdot h\partial_t^j\delta=hf\partial_t^j\delta-jh\partial_t^{j-1}\delta
\end{equation}
for every $D\in {\mathcal Der}_k(\cO_X)$ and $h\in\cO_X$.

Note now that if $\cM$ is an arbitrary coherent $\cD_X$-module, then we have an isomorphism
$$\iota_+(\cM)\simeq\cM\otimes_{\cO_X}B_f=\bigoplus_{j\geq 0}\cM\otimes\partial_t^j\delta,$$
with the actions of ${\mathcal Der}_k(\cO_X)$ and of $t$ being given by the analogues of the formulas in (\ref{action_on_Bf}):
\begin{equation}\label{action_on_Bf2}
D\cdot (u\otimes \partial_t^j\delta)=Du\otimes \partial_t^j\delta-D(f)u\otimes \partial_t^{j+1}\delta\quad\text{and}\quad t\cdot (u\otimes \partial_t^j\delta)=fu\otimes \partial_t^j\delta-ju\otimes\partial_t^{j-1}\delta
\end{equation}
for every $D\in {\mathcal Der}_k(\cO_X)$ and $u\in\cM$. 

\begin{eg}
Besides $B_f$, one important example is that of 
$$B_f':=\iota_+\big(\cO_X[1/f]\big)=\bigoplus_{j\geq 0}\cO_X[1/f]\partial_t^j\delta.$$
\end{eg}

\begin{rmk}[Behavior under restriction to an open subset]\label{rmk_restriction_Vfiltration_open_subset}
If $U$ is an open subset of $X$, we consider the restriction $g=f\vert_U$ of $f$, and let $\iota_f\colon X\hookrightarrow X\times\AA^1$ and $\iota_g\colon U\hookrightarrow U\times\AA^1$
be the corresponding graph embeddings. It is then clear that for every coherent $\cD_X$-module $\cM$, we have a canonical isomorphism
$$(\iota_g)_+(\cM\vert_U)\simeq (\iota_f)_+(\cM)\vert_U$$
(recall that we consider both sides as sheaves on $U$). In this case it follows from Remark~\ref{restriction_Vfilt_open_subset} that if $V^{\bullet}(\iota_f)_+\cM$ is a $V$-filtration
on $(\iota_f)_+(\cM)$, then $V^{\bullet}(\iota_f)_+\cM\vert_U$ gives a $V$-filtration on $(\iota_g)_+(\cM\vert_U)$.
\end{rmk}

\begin{rmk}[The support of ${\rm Gr}_V^{\alpha}(\cM)$]\label{gr_ann_by_f}
It follows from the definition that if $\cM$ has a $V$-filtration with respect to $f$, then $t\cdot {\rm Gr}_V^{\alpha}(\iota_+\cM)=0$ for all $\alpha\in\QQ$. 
A related fact is that as a $\cD_X$-module, ${\rm Gr}_V^{\alpha}(\iota_+\cM)$ is supported on $H$ for all $\alpha\in\QQ$.  Indeed, note that if $u\in V^{\alpha}\iota_+(\cM)\cap \sum_{j\leq p}\cM\otimes\partial_t^j\delta$, then
$$fu=tu+(f-t)u\in V^{>\alpha}\iota_+(\cM)+\left(V^{\alpha}\iota_+(\cM)\cap \sum_{j\leq p-1}\cM\otimes\partial_t^j\delta\right).$$ 
We deduce arguing by induction on $p$ that $f^{p+1}u\in V^{>\alpha}\iota_+(\cM)$.
\end{rmk}

\begin{rmk}[Behavior when $\cM$ has no $f$-torsion]\label{rmk_M_no_ftorsion}
It follows from the formula (\ref{action_on_Bf2}) 
that in $\iota_+(\cM)$ we have
\begin{equation}\label{eq_rmk_M_no_ftorsion}
t\cdot\sum_{i=0}^Nu_i\otimes\partial_t^i\delta=fu_N\otimes\partial_t^{N}\delta+(fu_{N-1}-Nu_N)\otimes\partial_t^{N-1}\delta+\ldots+(fu_0-u_1)\otimes\delta.
\end{equation}
We deduce that multiplication by $t$ on $\iota_+(\cM)$ is injective (or bijective) if and only if multiplication by $f$ on $\cM$
is injective (respectively, bijective). We thus conclude using Corollary~\ref{cor_prop_descr_Var} that if $\cM$ admits a $V$-filtration with respect to $f$ and $\cM$
has no $f$-torsion (or multiplication by $f$ is invertible), then for every $u\in \iota_+(\cM)$ and $\alpha\in\QQ$, we have $u\in V^{\alpha}\iota_+(\cM)$
if and only if $tu\in V^{\alpha+1}\iota_+(\cM)$ (respectively, for every $\alpha\in\QQ$, we have $t\cdot V^{\alpha}\iota_+(\cM)=V^{\alpha+1}\iota_+(\cM)$).
\end{rmk}

\begin{rmk}[Restriction to the complement of $H$]\label{Vfilt_where_f_invertible}
If $U=X\smallsetminus H$, then $\iota(U)$ is contained in the complement of the hypersurface defined by $t$, hence it follows from 
Remarks~\ref{D_mod_str_V0} and \ref{rmk_restriction_Vfiltration_open_subset} that if $\cM$ has a $V$-filtration with respect to $f$, then $V^{\alpha}\iota_+(\cM)\vert_U=\iota_+(\cM)\vert_U$ for all $\alpha\in\QQ$. Moreover, if $f$ is invertible (so $H=\emptyset$), then every $\cM$ has a $V$-filtration with respect to $f$.
\end{rmk}

\begin{eg}[The case of a $\cD_X$-module supported on $H$]\label{eg_case_support_in_H}
Note that ${\rm Supp}(\cM)\subseteq H$ if and only if ${\rm Supp}\big(\iota_+(\cM)\big)\subseteq X\times\{0\}$ and in this case it follows from
Example~\ref{V_filt_support_H} that $\cM$ has a $V$-filtration with respect to $f$. Moreover, we have
$V^{\alpha}\iota_+(\cM)=0$ for all $\alpha>0$ and
$${\rm Gr}_V^0\big(\iota_+(\cM)\big)\simeq \big\{u\in \iota_+(\cM)\mid t\cdot u=0\big\}$$
$$=\left\{\sum_{j\geq 0}\tfrac{f^j}{j!}u_0\otimes\partial_t^j\delta\mid u_0\in \cM\right\},$$
where the last equality follows easily from formula (\ref{eq_rmk_M_no_ftorsion})
(note that $f^ju=0$ for $j\gg 0$ by the assumption on $\cM$).
\end{eg}

\begin{rmk}[Compatibility of definitions smooth hypersurfaces]\label{Vfilt_on_B_f_smooth}
We note that if $f$ defines a smooth hypersurface in $X$, then the two notions of $V$-filtrations determine each other. More precisely, we have a $V$-filtration
on $\cM$ with respect to $f$ if and only if we have a $V$-filtration on $\iota_+(\cM)$ with respect to $t$. Indeed, this follows from Proposition~\ref{prop_Vfil_immersion}, since
$t\circ\iota=f$. Moreover, an easy computation using the proposition implies that the two $V$-filtrations determine each other as follows: 
if we choose local coordinates $x_1,\ldots,x_{n-1},y$ on $X$, with $y=f$, then 
$$V^{\alpha}\iota_+(\cM)=\sum_{j\geq 0}\partial_y^j\cdot (V^{\alpha}\cM\otimes\delta).$$

For example, if we take $\cM=\cO_X$, then it follows from Example~\ref{eg_V_filtration_O} that if the hypersurface defined by $f$ is smooth, then the $V$-filtration on $B_f$ is given by
$$V^{\alpha}\iota_+(\cO_X)=\cD_X\cdot y^{\lceil\alpha\rceil-1}\delta=\sum_{j\geq 0}\cO_X\cdot\partial_y^jy^{\lceil\alpha\rceil-1}\delta\quad\text{for all}\quad\alpha\in\QQ,$$
with the convention that $y^m=1$ if $m\leq 0$. In particular, we see that $V^1\iota_+(\cO_X)=\iota_+(\cO_X)$. 

\end{rmk}

\begin{eg}[The $V$-filtration on $\iota_+(\cO_X)$ when $H$ is an SNC divisor]\label{Vfiltration_SNC}
For a more interesting example, let us consider the case of a simple normal crossing divisor. Suppose that $x_1,\ldots,x_n$ are algebraic coordinates on $X$, and 
$f=\prod_{i=1}^rx_i^{a_i}$, where $a_1,\ldots,a_r$ are positive integers. For every $\lambda\in\QQ$, we put
$$I(f^{\lambda})=\big(x_1^{\lceil\lambda a_1\rceil-1}\cdots x_r^{\lceil\lambda a_r\rceil-1}),$$
with the convention that this is $\cO_X$ for $\lambda\leq 0$
(with the terminology in Section~\ref{section_multiplier} below, for every $\lambda>0$, this is the multiplier ideal $\cJ\big(X,(\lambda-\epsilon)H\big)$, where $0<\epsilon\ll 1$).
 It is clear from the formula that $I(f^{\lambda_1})\subseteq I(f^{\lambda_2})$ if $\lambda_1\geq \lambda_2$.
Note also that we have $f\cdot I(f^{\lambda})\subseteq I(f^{\lambda+1})$, with equality
if $\lambda>0$. 

Let us show that if we put
\begin{equation}\label{eq1_Vfiltration_SNC}
V^{\lambda}\iota_+(\cO_X)=\sum_{j\geq 0}\cD_X\cdot I(f^{\lambda+j})\partial_t^j\delta\quad\text{for all}\quad\lambda\in\QQ,
\end{equation}
then $V^{\bullet}\iota_+(\cO_X)$ is a $V$-filtration with respect to $t$.
A more general result is proved in \cite[Theorem~3.4]{Saito-MHM}, with a rather involved proof. We give a direct proof, by checking that the formula
in (\ref{eq1_Vfiltration_SNC}) satisfies the properties of the $V$-filtration. It is clear that this is a decreasing, exhaustive filtration. Note also that if $N={\rm lcm}(a_1,\ldots,a_r)$,
then $I(f^{\lambda})$ is constant for $\lambda\in (i/N,(i+1)/N]$ for every $i\in\ZZ$, hence the filtration we defined on $\iota_+(\cO_X)$ has the same property, 
and thus it is discrete and left continuous. 

Note first that the sum in the formula (\ref{eq1_Vfiltration_SNC}) can be replaced by a finite sum. Indeed, for every $i$, with $1\leq i\leq r$, 
and every $b\in\ZZ_{\geq 0}^r$, 
we have
$$(x_i\partial_{x_i}-b_i)\cdot x^b\partial_t^j\delta=-a_ix^{a+b}\partial_t^{j+1},$$
which implies 
$$I(f^{\lambda+1})\partial_t^{j+1}\delta\subseteq \cD_X\cdot I(f^{\lambda})\partial_t^j\delta\quad\text{for}\quad\lambda>0.$$
This implies that for $\lambda>0$ we have $V^{\lambda}\iota_+(\cO_X)=\cD_X\cdot I(f^{\lambda})\delta$ and, more generally, we have
$$V^{\lambda}\iota_+(\cO_X)=\sum_{j\leq\max\{1-\lambda,0\}}\cD_X\cdot I(f^{\lambda+j})\partial_t^j\quad\text{for all}\quad \lambda\in\QQ.$$
This implies that every $V^{\lambda}\iota_+(\cO_X)$ is finitely generated over $\cD_X$ (hence also over $V^0\cD_{X\times\AA^1}$). 

Note next that for every $j\geq 0$, we have
$$t\cdot I(f^{\lambda+j})\partial_t^j\delta\subseteq I(f^{\lambda+j+1})\partial_t^j\delta+I(f^{\lambda})\partial_t^{j-1}\delta\subseteq V^{\lambda+1}\iota_+(\cO_X),$$
so $t\cdot V^{\lambda}\iota_+(\cO_X)\subseteq V^{\lambda+1}\iota_+(\cO_X)$ for every $\lambda\in\QQ$. Moreover, this is an equality for $\lambda>0$. Indeed, we have seen that
in this case we have
$$V^{\lambda+1}\iota_+(\cO_X)=\cD_X\cdot I(f^{\lambda+1})\delta=f\cdot I(f^{\lambda})\delta=t\cdot I(f^{\lambda})\delta=t\cdot V^{\lambda}\iota_+(\cO_X).$$
Note also that we have 
$$\partial_t\cdot I(f^{\lambda+j})\partial_t^j=I(f^{\lambda+j})\partial_t^{j+1}\delta\subseteq V^{\lambda-1}\iota_+(\cO_X),$$
hence $\partial_t\cdot V^{\lambda}\iota_+(\cO_X)\subseteq V^{\lambda-1}\iota_+(\cO_X)$ for all $\lambda\in\QQ$. 

In order to conclude, it is enough to show that if $b\in\ZZ_{\geq 0}^r$ is such that $x^b\in I(f^{\lambda+j})$, then
$$(\partial_tt-\lambda)^rx^b\partial_t^j\delta\in V^{>\lambda}\iota_+(\cO_X).$$
Note that $x^b\partial_t^j\delta\in V^{>\lambda}\iota_+(\cO_X)$, unless there is $i$ such that $b_i=(\lambda+j) a_i-1$. 
In this case, if $e_1,\ldots,e_r$ is the standard basis of $\ZZ^r$, a simple computation gives
$$\partial_{x_i}\cdot x^{b+e_i}\partial_t^j\delta=(b_i+1)x^b\partial_t^j\delta-a_ix^{a+b}\partial_t^{j+1}\delta=-a_i(\partial_tt-\lambda)x^b\partial_t^j\delta.$$
We thus see that if $J=\big\{i\mid b_i=(\lambda+j) a_i-1\big\}$, then
$$(\partial_tt-\lambda)^{|J|}\cdot x^b\partial_t^j\delta\in\cD_X\cdot x^{b'}\partial_t^j\delta,$$
for some $b'\in\ZZ_{\geq 0}^r$ such that $x^{b'}\in I(f^{\lambda'+j})$ for some $\lambda'>\lambda$. 
We thus conclude that 
$$(\partial_tt-\lambda)^{|J|}\cdot x^b\partial_t^j\delta\in V^{>\lambda}\iota_+(\cO_X).$$
\end{eg}

\section{$V$-filtrations and $b$-functions}\label{sect_Vfil_bfcn}

The existence of $V$-filtrations is closely related to the existence of $b$-functions and the rationality of their roots. In this section we describe this relation.
As in the previous section, we assume that $X$ is a smooth, irreducible variety over the algebraically closed 
field $k$, of characteristic $0$. We consider a nonzero $f \in \cO_X(X)$ that defines the hypersurface $H$ and let $U = X \smallsetminus H$ be its complement, with $j\colon U\hookrightarrow X$ the inclusion. Let $\iota \colon X \to X\times \mathbf A^1$ be the graph embedding along $f$.

\subsection{Basics on $b$-functions}

Given a quasi-coherent left $\cD_X$-module $\cM$ on which $f$ acts bijectively
(equivalently, $\cM\simeq j_+(\cM\vert_U)$),
we consider the $\cO_X[s]$-module
$$\cM[s]f^s:=\cM\otimes_kk[s]f^s,$$
where $k[s]f^s$ is a free $k[s]$-module with basis the symbol $f^s$. 

In fact, $\cM[s]$ is not just a module over $\cO_X[s]$, but a left module over $V^0\cD_{X\times\AA^1}=\cD_X\langle s,t\rangle$.
Here $t$ acts by the translation $s\mapsto s+1$, that is
$$P(s)uf^s\mapsto P(s+1)fuf^s.$$
Moreover, a derivation $D\in {\mathcal Der}_k(\cO_X)$
commutes with the $k[s]$-action and acts on an element $uf^s$, with $u\in\cM$,
by the ``expected" rule
$$D\cdot uf^s=(Du)f^s+s\frac{D(f)}{f}uf^s.$$

It is a fundamental result, due to Bernstein and Kashiwara (see \cite[p. 25]{Bernstein}), that if $\cM$ is as above and $\cM\vert_U$ is holonomic, then 
for every local section $u$ of $\cM$, there is a nonzero $b(s)\in k[s]$ such that
\begin{equation}\label{eq_thm_b_fcn}
b(s)uf^s\in\cD_X[s]\cdot (fu)f^s.
\end{equation}
It is clear that the set of such $b(s)$ is an ideal in $k[s]$; its monic generator is the \emph{$b$-function} $b_u(s)$.

\begin{rmk}[Specialization to integral powers of $f$]\label{rmk_specialize_s}
In the above setting, we may specialize $s$ to integers. More precisely, for every 
$m\in\ZZ$, we have a morphism 
of sheaves of rings $\alpha_m\colon \cD_X[s]\to \cD_X$ that maps $P(s)$ to $P(m)$ for $P(s)\in\cD_X[s]$ and a surjective morphism of sheaves of
$\cD_X$-modules $\beta_m\colon \cM[s]f^s\to\cM$ where $\beta_m\big(Q(s)uf^s\big)=Q(m)f^mu$ for every $Q(s)\in k[s]$ and $u\in\cM$
(we here consider $\cM$ as a $\cD_X[s]$-module via $\alpha_m$). In particular, by specializing $s$ to $m\in\ZZ$ in (\ref{eq_thm_b_fcn}), we conclude 
that $b(m)f^mu\in \cD_X\cdot f^{m+1}u$. 
\end{rmk}

\begin{eg}[The Bernstein-Sato polynomial]
In the special case when $\cM=\cO_X[1/f]$ and $u = 1$, the $b$-function $b_u(s)$ is denoted by $b_f(s)$ and it is called the
\emph{Bernstein-Sato polynomial} of $f$. It is an important result of Kashiwara \cite{Kashiwara} that all roots of $b_f(s)$ are negative rational numbers.
For a more precise statement regarding the roots of $b_f(s)$ and some related $b$-functions, see Theorem~\ref{thm_DM_bfcn} below.
\end{eg}

\begin{rmk}[$b$-functions give coherence of push-forward]
As we have mentioned in Section~\ref{section_functors},
the existence of $b$-functions is a key ingredient in proving that if $\cN$ is a holonomic $\cD_U$-module, then $j_+(\cN)$ is a holonomic
$\cD_X$-module. The starting point is showing that $j_+(\cN)$ is coherent. To see this, suppose that $X$ is affine
and $\cN$ is generated over $\cD_X$ by the global sections $u_1,\ldots,u_r$. If $N\in {\mathbf Z}_{\geq 0}$ is such that 
no $b_{u_i}(s)$ has integer roots $<-N$, then by specializing  $s$ to $-j$ in the definition of $b_{u_i}$, for $j\geq N$, we see that $\tfrac{1}{f^{j+1}}u_i\in {\mathcal D_X}\cdot\tfrac{1}{f^j}u_i$.
Therefore $j_+(\cN)$ is generated over $\cD_X$ by $\tfrac{1}{f^N}u_i$, for $1\leq i\leq r$.
\end{rmk}

\begin{prop}\label{prop_V_and_b}
If $\cM$ is a coherent $\cD_X$-module on which $f$ acts bijectively, then we have
an isomorphism of $\cD_X\langle s,t\rangle$-modules
\begin{equation}\label{eq_prop_V_and_b}
\tau\colon \cM[s]f^s\overset{\sim}\to\iota_+(\cM),\quad P(s)uf^s\mapsto P(-\partial_tt)(u\otimes\delta).
\end{equation}
Moreover, the action of $\partial_t$ on the right-hand side corresponds to the action of $-st^{-1}$ on the left-hand side.
\end{prop}

\begin{proof}
Let us check the compatibility of $\tau$ with the $\cD_X\langle s,t\rangle$-action. It is clear that $\tau$ is $\cO_X[s]$-linear, hence we
only need to check the compatibility with the action of ${\mathcal Der}_k(\cO_X)$ and that of $t$.
Note that
$$\tau\big(t\cdot P(s)uf^s)=\tau\big(P(s+1)fuf^s\big)=P(-\partial_tt+1)fu\otimes\delta$$
$$=P(-\partial_tt+1)t u\otimes\delta=
tP(-\partial_tt)u\otimes\delta=t\cdot \tau\big(P(s)uf^s\big),$$
where the second to last equality follows from Lemma~\ref{lem_s}i). 

Suppose now that $D\in {\mathcal Der}_k(\cO_X)$. We have
$$\tau\big(D\cdot P(s)uf^s)=\tau\big(P(s)Duf^s+sP(s)\tfrac{D(f)}{f}uf^s\big)=P(-\partial_tt)Du\otimes\delta-P(-\partial_tt)\partial_tt\tfrac{D(f)}{f}u\otimes\delta$$
$$=P(-\partial_tt)Du\otimes\delta-P(-\partial_tt)D(f)u\otimes\partial_t\delta,\quad\text{while}$$
$$D\cdot \tau\big(P(s)uf^s\big)=D\cdot P(-\partial_tt)u\otimes \delta=P(-\partial_tt)\big(Du\otimes\delta-D(f)u\otimes\partial_t\delta\big),$$
hence $\tau\big(D\cdot P(s)uf^s)=D\cdot \tau\big(P(s)uf^s\big)$. 

Note next that
it follows from the second formula in Lemma~\ref{lem_s}ii) that 
$$(-1)^m\tau\big(s(s-1)\cdots(s-m+1)uf^s\big)=f^mu\otimes\partial_t^m\delta\quad\text{for all}\quad m.$$
Since $k[s]$ has a basis given by $\prod_{j=0}^{m-1}(s-j)$, for $m\geq 0$, and since multiplication by $f$ is invertible on $\cM$,
it follows that $\tau$ is a bijective map. This completes the proof of the proposition.
\end{proof}

We next extend slightly the notion of $b$-function, as follows. This more general notion will feature in the characterization of the $V$-filtration
in Corollary~\ref{prop_char_V_filt}.

\begin{defi}\label{gen_defi_b_fcn}
Let $\cM$ be a coherent $\cD_X$-module.
We say that a local section $w\in \iota_+(\cM)$ \emph{has a $b$-function} if there is a nonzero $b(s)\in k[s]$
such that $b(s)w\in V^1\cD_{X\times\AA^1}\cdot w$. In this case, the set of such polynomials $b(s)$ is a nonzero ideal of $k[s]$.
Its monic generator is the \emph{$b$-function} $b_w(s)$.
\end{defi}

\begin{rmk}[Compatibility of the two definitions of $b$-functions]
Note that if $\cM$ is a coherent $\cD_X$-module on which $f$ acts bijectively, $u$ is a local section of $\cM$, and $w=u\otimes\delta$ is the corresponding
section of $\iota_+(\cM)$, then
$w$ corresponds to $uf^s$ via the isomorphism $\tau$ in Proposition~\ref{prop_V_and_b} and in this case $b_w(s)$ is the monic polynomial of minimal degree
such that $b_w(s)uf^s\in\cD_X[s]\cdot fuf^s$ (hence $b_w=b_u$).
Indeed, this follows from the fact that $V^1\cD_{X\times\AA^1}=\cD_X\langle s,t\rangle\cdot t$, hence $V^1\cD_{X\times\AA^1}\cdot (u\otimes\delta)=\cD_X[s]\cdot (fu\otimes\delta)$.
 In particular, by taking $\cM=\cO_X[1/f]$, we see that $b_{\delta}(s)=b_f(s)$. 
\end{rmk}

\begin{rmk}[Interpretation of $b_w$]\label{rmk_gen_defi_b_fcn}
With the notation in Definition~\ref{gen_defi_b_fcn}, we note that in fact $b_w(s)$ satisfies $b_w(s)\cdot \tfrac{V^0\cD_{X\times\AA^1}\cdot w}{V^1\cD_{X\times\AA^1}\cdot w}=0$.
This is due to the fact that $V^0\cD_{X\times\AA^1}=\cD_X[s]+V^1\cD_{X\times\AA^1}$ and $b_w(s)$ commutes with the elements of $\cD_X[s]$. 

Moreover, if $\cN=V^0\cD_{X\times\AA^1}\cdot w$, then $b_w$ is the minimal polynomial of the action of $s$ on $\cN/t\cN$. This follows from the fact that 
$V^1\cD_{X\times\AA^1}\cdot w=\big(t\cdot \cD_X\langle s,t\rangle\big)\cdot w=t\cN$.
\end{rmk}

\begin{rmk}[Behavior of $b_w$ with respect to an open cover]\label{B_fcn_open_cover_v0}
Suppose that $X=U_1\cup\ldots\cup U_r$ is an open cover, $\cM$ is a coherent $\cD_X$-module, and $f\in\cO_X(X)$ is nonzero. If $f_j=f\vert_{U_j}$ and $\iota_f\colon X\hookrightarrow 
X\times\AA^1$ and $\iota_{f_j}\colon U_j\hookrightarrow U_j\times\AA^1$ are the corresponding graph embeddings, then we have canonical isomorphisms
$$\iota_{f_j}(\cM\vert_{U_j})\simeq\iota_f(\cM)\vert_{U_j}.$$
For $w\in\Gamma\big(X,\iota_+(\cM)\big)$, set $w_j=w\vert_{U_j}$, 
then it follows from definition that 
$$b_w={\rm lcm}\big\{b_{w_j}\mid 1\leq j\leq r\}$$
(this means that $b_w$ exists if and only if each $b_{w_j}$ exists, and if this is the case, then we have the stated equality).
\end{rmk}

\begin{rmk}[Behavior of $b_w$ when rescaling $w$]\label{bfcn_rescaling_invertible}
Note that if $w\in \iota_+(\cM)$ and $p\in\cO_X(X)$ is an invertible function, then $b_w$ exists if and only if $b_{pw}$ exists and, if this is the case, then
$b_w=b_{pw}$. Indeed, if $b(s)w=Q\cdot w$, for some $Q\in V^1\cD_{X\times\AA^1}$, then 
$$b(s)pw=(pQ)\cdot w=(pQp^{-1})pw$$
and $pQp^{-1}\in V^1\cD_{X\times\AA^1}$. This implies that if $b_w$ exists, then $b_{pw}$ exists and $b_{pw}$ divides $b_w$. 
The converse follows by symmetry since $p$ is invertible.
\end{rmk}

\begin{rmk}[Behavior of $V$-filtration and $b$-functions when rescaling $f$]\label{Vfilt_rescaling}
Given a coherent $\cD_X$-module $\cM$, it is easy to compare the $V$-filtrations of $\cM$ with respect to $f$ and $g=pf$, where $p\in\cO_X(X)$ is an invertible function. 
Note that if $\iota_f$ and $\iota_g$ are the graph embeddings corresponding to $f$ and $g$, respectively, then $\iota_g=\varphi\circ \iota_f$, where 
$\varphi\colon X\times \AA^1\to X\times\AA^1$ is the isomorphism given by $\varphi(x,t)=\big(x,p(x)t\big)$, so that $t\circ\varphi=pt$. We thus get an isomorphism of $\cD_{X\times\AA^1}$-modules
$$\varphi^*(\iota_g)_+(\cM)\simeq (\iota_f)_+(\cM),$$
which in turn induces an isomorphism of $\cO_X$-modules
$$\tau\colon (\iota_g)_+(\cM)\to (\iota_f)_+(\cM),\,\,\tau\big(\sum_{j=1}^N\tfrac{u_j}{(g-t)^j}\big)=\sum_{j=1}^N\tfrac{p^{-j}u_j}{(f-t)^j}.$$
This has the property that $\tau(Pw)=\tau_0(P)w$, where $\tau_0\colon\cD_X\langle t,\partial_t\rangle\to \cD_X\langle t,\partial_t\rangle$ is the isomorphism
of sheaves of rings which is the identity on $\cO_X$ and satisfies
\begin{equation}\label{eq_Vfilt_rescaling}
\tau_0(t)=pt,\,\,\tau_0(\partial_t)=p^{-1}\partial_t,\,\,\text{and}\quad \tau_0(Q)=Q-Q(p)p^{-1}t\partial_t
\end{equation}
for $Q\in {\rm Der}_k(\cO_X)$ (note that, in particular, $\tau_0(s)=s$).
By the last assertion in Remark~\ref{remark_Vfil_invertible_multiple}, the $V$-filtrations
on $(\iota_f)_+(\cM)$ with respect to $t$ and $pt$ coincide, hence $\cM$ has a $V$-filtration with respect to $f$ if and only if it has a $V$-filtration with respect to $g$,
and in this case we have
$$\tau\big(V^{\alpha}(\iota_g)_+(\cM)\big)=V^{\alpha}(\iota_f)_+(\cM)\quad\text{for all}\quad\alpha\in\QQ.$$
Moreover,  since $\tau$ is compatible with the action of $V^{\bullet}\cD_X\langle t,\partial_t\rangle$
(this follows from the definition of $V^{\bullet}\cD_{X\times\AA^1}$ in terms of the ideal of $X\times\{0\}$, but can be deduced also from the
explicit formulas (\ref{eq_Vfilt_rescaling})),
it follows from the definition that 
$b_w=b_{\tau(w)}$ for every $w\in (\iota_g)_+(\cM)$, in the sense that one exists if and only if the other one exists and in this case they are equal. 
 For example, we see that for every $u\in\cM$ and $m\in\ZZ_{\geq 0}$,
if $v=u\otimes \partial_t^m\delta\in (\iota_g)_+(\cM)$ and $w=u\otimes\partial_t^m\delta\in (\iota_f)_+(\cM)$, then 
$$b_w=b_{\tau(w)}=b_{p^{-m-1}v}=b_v,$$
where the last equality follows from Remark~\ref{bfcn_rescaling_invertible}.
\end{rmk}

\subsection{Existence of the $V$-filtration via $b$-functions}
The following result provides the criterion for the existence of $V$-filtrations in terms of $b$-functions. We give the proof following an approach due to Sabbah.

\begin{thm}\label{thm_existence_Vfilt}
Let $\cM$ be a coherent $\cD_X$-module.
\item[i)] If $\cM$ has a $V$-filtration with respect to $f$, then every section $w\in\iota_+(\cM)$ has a $b$-function. Moreover, if $w\in V^{\alpha}\iota_+(\cM)$, then
all roots of $b_w(s)$ are rational numbers $\gamma\leq -\alpha$ such that ${\rm Gr}_V^{-\gamma}(\iota_+\cM)\neq 0$. 
\item[ii)] Conversely, if $w_1,\ldots,w_r\in\Gamma(X,\cM)$ generate $\cM$ as a $\cD_X$-module and if each $w_i\otimes\delta\in \iota_+(\cM)$
has a $b$-function whose roots are all rational, then $\cM$ has a $V$-filtration with respect to $f$.
\end{thm}

Before giving the proof, it is convenient to introduce a version of $V$-filtration indexed by integers. We do this in the setting of a smooth hypersurface (as usual, we will use it on $X\times {\mathbf A}^1$).

\begin{defi}\label{defi_Z_filtration}
Let $\cM$ be a coherent $\cD_X$-module and $t\in\cO_X(X)$ nonzero that defines a smooth hypersurface. A \emph{pre-$V$-filtration} on $\cM$
(with respect to $t$) is a decreasing, exhaustive filtration $W^{\bullet}\cM=(W^i\cM)_{i\in {\mathbf Z}}$ on $\cM$ 
by quasi-coherent $\cO_X$-modules,  that satisfies the following properties:
\begin{enumerate}
\item[a)] We have $V^i\cD_{X}\cdot W^j\cM\subseteq W^{i+j}\cM$ for all $i,j\in\ZZ$.
\item[b)] $W^i\cM$ is locally finitely generated over $V^0\cD_X$ for every $i\in\ZZ$.
\item[c)] We have $W^{i+1}\cM=t\cdot W^i\cM$ for $i\gg 0$. 
\item[d)] There is a polynomial $p=p_W\in\QQ[x]$, with all roots in $\QQ$, such that $p(\partial_tt-i)\cdot W^i\cM\subseteq W^{i+1}\cM$ for all $i\in\ZZ$.
\end{enumerate}
If the polynomial $p_W$ in d) has all its roots in $[0,1)$, then we call $W^{\bullet}$ a \emph{$\ZZ$-indexed $V$-filtration}. 
As before, if $f\in\cO_X(X)$ is nonzero and $\iota=\iota_f\colon X\to X\times {\mathbf A}^1$ is the corresponding closed immersion, we consider 
pre-$V$-filtrations and $\ZZ$-indexed $V$-filtrations
on $\iota_+(\cM)$ with respect to $t$.
\end{defi}

The following proposition relates the various notions of $V$-filtrations. 

\begin{prop}\label{prop_Z_filtration}
Let $\cM$ be a coherent $\cD_X$-module and $t\in\cO_X(X)$ nonzero that defines a smooth hypersurface. 
\begin{enumerate}
\item[i)] If $V^{\bullet}\cM$ is a $V$-filtration on $\cM$ with respect to $t$, then $(V^i\cM)_{i\in\ZZ}$ is a $\ZZ$-indexed $V$-filtration with respect to $t$.
\item[ii)] If $\cM$ has a pre-$V$-filtration with respect to $t$, then it also has a $\ZZ$-indexed $V$-filtration.
\item[iii)] If $(W^i\cM)_{i\in\ZZ}$ is a $\ZZ$-indexed $V$-filtration on $\cM$ with respect to $t$, then $\cM$ has a $V$-filtration $V^{\bullet}\cM$
with respect to $t$ and $W^i\cM=V^i\cM$ for all $i\in\ZZ$.
\end{enumerate}
\end{prop}

\begin{proof}
For the assertion in i), we only need to prove that property d) holds for some polynomial $p$ with roots in $[0,1)$. Since $V^{\bullet}\cM$ is discrete and left-continuous, it follows from 
property ii) in Definition~\ref{defi_V_filtration} that there are polynomials $p_1$ and $p_2$ with roots in $[0,1)$ such that $p_1(\partial_tt)\cdot V^0\cM\subseteq V^1\cM$
and $p_2(\partial_tt-1)\cdot V^1\cM\subseteq V^2\cM$. Since $V^{m+1}\cM=t\cdot V^m\cM$  and $V^{-m}\cM=\sum_{i =0}^m\partial_t^i\cdot V^0\cM$ for all $m\geq 1$ 
(see assertion iii) in Corollary~\ref{cor_rmk1_Vfiltration}), using Lemma~\ref{lem_s} it follows easily that if $p=p_1p_2$, then $p(\partial_tt-m)\cdot V^m\cM\subseteq
V^{m+1}\cM$ for all $m\in\ZZ$. 

Let us prove ii). 
Note that if $W^{\bullet}\cM$ is a pre-$V$-filtration on $\cM$  with polynomial $p_W$ and for an integer $q$ we put $\widetilde{W}^m\cM=W^{m+q}\cM$, then
$\widetilde{W}^{\bullet}\cM$ is again a pre-$V$-filtration with corresponding polynomial $p_{\widetilde{W}}(x)=p_W(x-q)$. By taking a suitable $q$ (small enough)  and replacing $W^{\bullet}\cM$
by $\widetilde{W}^{\bullet}\cM$, we see that we may assume that all roots of $p_W$ are $<1$. 

Suppose now that $\lambda$ is a root of $p_W$
and let us write $p_W=(x-\lambda)^dq(x)$, where $q(\lambda)\neq 0$. We define a new filtration $U^{\bullet}\cM$ by the formula
$$U^m\cM:=W^{m+1}\cM+(\partial_tt-m-\lambda)^d\cdot W^m\cM\quad\text{for all}\quad m\in\ZZ.$$
It is clear that this is a decreasing, exhaustive filtration and it is an easy exercise to see, using Lemma~\ref{lem_s}, that since $W^{\bullet}\cM$ satisfies conditions a), b), and c)
in Definition~\ref{defi_Z_filtration}, so does
$U^{\bullet}\cM$. Let us show that we may take $p_U(x)=(x-\lambda-1)^dq(x)$. Indeed, for every $m\in\ZZ$, we have
$$(\partial_tt-\lambda-m-1)^dq(\partial_tt-m)\cdot W^{m+1}\cM\subseteq q(\partial_tt-m)\cdot U^{m+1}\cM\subseteq U^{m+1}\cM$$
by definition of $U^{\bullet}\cM$ and 
$$(\partial_tt-\lambda-m-1)^dq(\partial_tt-m)\cdot (\partial_tt-m-\lambda)^d\cdot W^m\cM$$
$$=(\partial_tt-\lambda-m-1)^dp_W(\partial_tt-m)\cdot W^m\cM\subseteq
(\partial_tt-\lambda-m-1)^d\cdot W^{m+1}\cM$$
$$\subseteq U^{m+1}\cM.$$
We thus conclude that $U^{\bullet}\cM$ is a pre-$V$-filtration and we may take $p_U(x)=(x-\lambda-1)^dq(x)$. After applying this construction finitely many times, we may replace 
$\lambda$
by a root in $[0,1)$, and after repeating the same process for the other roots of $p_W$, and replacing $W^{\bullet}\cM$ by the final pre-$V$-filtration, we see that we may assume that all
roots of $p_W$ lie in $[0,1)$, hence this is a $\ZZ$-indexed $V$-filtration. 

We now prove iii). Let $\alpha_1<\ldots<\alpha_d$ be the distinct roots of $p_W(x)$, which by assumption lie in the interval $[0,1)$. 
For every $m\in\ZZ$ and every $i$, with $1\leq i\leq d$, let $P^{(m)}_i$, with $W^{m+1}\cM\subseteq P^{(m)}_i\subseteq W^m\cM$ be such that $P^{(m)}_i/W^{m+1}\cM$ is the generalized eigenspace with eigenvalue 
$\alpha_i+m$ for the action
of $\partial_tt$ on $W^m\cM/W^{m+1}\cM$. It is a standard linear algebra result that
we have $W^m\cM/W^{m+1}\cM=\bigoplus_{i=1}^d P^{(m)}_i/W^{m+1}\cM$. Since $\partial_tt=-s$ is a $\cD_X[s]$-linear operator and since $t\cdot P^{(m)}_i\subseteq W^{m+1}\cM\subseteq P^{(m)}_i$, it follows that each 
$P_i$ is a $V^0\cD_X$-submodule of $\cM$. 
We put $V^m\cM=W^m\cM$ and for 
$1\leq i\leq d$, we define the $V^0\cD_X$-module $V^{m+\alpha_i}\cM$ such that $W^{m+1}\cM\subseteq V^{m+\alpha_i}\cM$ and 
$V^{m+\alpha_i}\cM/W^{m+1}\cM=P_i\oplus\ldots\oplus P_d$. Note that by definition we have
$$W^m\cM=V^m\cM=V^{m+\alpha_1}\cM\supseteq V^{m+\alpha_2}\cM\supseteq\ldots\supseteq V^{m+\alpha_d}\cM\supseteq W^{m+1}\cM=V^{m+1}\cM.$$
We extend this filtration to all rational numbers such that $V^{\lambda}\cM$ takes constant value for $\lambda$ in each interval $(m,m+\alpha_1],(m+\alpha_1,m+\alpha_2],\ldots,(m+\alpha_d,m+1]$. 
Checking that this is indeed a $V$-filtration corresponding to $t$ is a straightforward exercise. 
\end{proof} 

\begin{rmk}[The case when $p_W$ has roots in {$(0,1]$}]\label{rmk_var_Z_filtration}
If $W^{\bullet}\cM$ is a pre-$V$-filtration on $\cM$ with respect to $t$ such that the polynomial $p_W$ has all roots in $(0,1]$, then a similar argument to the one above shows that $\cM$ has a $V$-filtration $V^\bullet \cM$ with respect to $t$ and
\[ W^i \cM = V^{>i}\cM \quad \text{ for all }\quad i\in \ZZ.\]
\end{rmk}

We now turn to the proof of the result relating $b$-functions and $V$-filtrations.

\begin{proof}[Proof of Theorem~\ref{thm_existence_Vfilt}]
In order to prove the assertion in i), note first that 
there is $\beta$ such that 
\begin{equation}\label{eq_claim_beta}
V^{\beta}\iota_+(\cM)\cap \cD_{X\times\AA^1}\cdot w\subseteq V^1\cD_{X\times\AA^1}\cdot w.
\end{equation}
Indeed, by Proposition~\ref{properties_V_filtration1}, the $V$-filtration with respect to $t$ on 
$\cM_w:=\cD_{X\times\AA^1}\cdot w$ is the restriction of the $V$-filtration on $\iota_+(\cM)$. 
Since $V^1\cM_w$ is finitely generated over $V^0\cD_{X\times \AA^1}$, it follows that there is $\gamma$ such that
$V^1\cM_w\subseteq V^{\gamma}\cD_{X\times\AA^1}\cdot w$. We conclude that if $\beta\in {\mathbf Z}_{>1}$ is such that
$\beta+\gamma\geq 2$, then
$$V^{\beta}\iota_+(\cM)\cap \cM_w=V^{\beta}\cM_w= t^{\beta-1}\cdot V^1\cM_w\subseteq (t^{\beta-1}\cdot V^{\gamma}\cD_{X\times\AA^1})\cdot w
\subseteq V^1\cD_{X\times\AA^1}\cdot w,$$
hence (\ref{eq_claim_beta}) holds.

Since $(s+\gamma)$ is nilpotent on ${\rm Gr}_V^{\gamma}(\iota_+\cM)$
for all $\gamma\in\QQ$ and since the $V$-filtration is discrete, it follows that if $\gamma_1,\ldots,\gamma_r$ are the rational numbers $\gamma$ with 
$\alpha\leq\gamma<\beta$ and with ${\rm Gr}_V^{\gamma}(\iota_+\cM)\neq 0$, then there is $N\geq 1$ such that
$$(s+\gamma_1)^{N}\cdots(s+\gamma_r)^{N}w\in V^{\beta}\iota_+\cM\cap \cD_{X\times\AA^1}\cdot w\subseteq V^1\cD_{X\times\AA^1}\cdot w.$$
This implies that $b_w(s)$ divides $\prod_{i=1}^r(s+\gamma_i)^{N}$, proving the assertion in i). 

In order to prove the statement in ii), we see that by Proposition~\ref{prop_Z_filtration}, it is enough to show that $\iota_+(\cM)$ has a pre-$V$-filtration
with respect to $t$. For every $j\in\ZZ$, let
$$W^j\iota_+(\cM):=\sum_{\ell=1}^rV^j\cD_{X\times\AA^1}\cdot (w_{\ell}\otimes\delta)\subseteq \iota_+(\cM).$$
Since $w_1,\ldots,w_r$ generate $\cM$ over $\cD_X$, it follows that $w_1\otimes\delta,\ldots,w_r\otimes\delta$ generate $\iota_+(\cM)$
over $\cD_{X\times\AA^1}$, hence $W^{\bullet}\iota_+(\cM)$ is exhaustive. 
It is also clear that it satisfies conditions a), b), and c) in Definition~\ref{defi_Z_filtration}. 
In order to check condition d), let $b_{\ell}$ be the $b$-function of $w_{\ell}\otimes\delta$, so 
we have $b_{\ell}(-\partial_tt)\cdot V^0\cD_{X\times\AA^1}(w_{\ell}\otimes \delta)\subseteq V^1\cD_{X\times\AA^1}(w_{\ell}\otimes\delta)$ for all $\ell$
(see Remark~\ref{rmk_gen_defi_b_fcn}). This implies that if $p(x)=\prod_{\ell=1}^rb_{\ell}(-x)$, then
$p(\partial_tt)\cdot W^0\iota_+(\cM)\subseteq W^1\iota_+(\cM)$. 
Moreover, it follows from the definition of $W^{\bullet}\iota_+(\cM)$ and the formulas (\ref{new_eq1_V_D}) and (\ref{new_eq2_V_D}) that for all $m\geq 0$, we have
$W^m\iota_+(\cM)=t^m\cdot W^0\iota_+(\cM)$ and $W^{-m}\iota_+(\cM)=\sum_{j=0}^m\partial_t^j\cdot W^0\iota_+(\cM)$.
Using Lemma~\ref{lem_s}, it is then straightforward to deduce that 
$p(\partial_tt+m)\cdot W^{-m}\iota_+(\cM)\subseteq W^{-m+1}\iota_+(\cM)$ for all $m\in\ZZ$, hence $W^{\bullet}\cM$ 
is a pre-$V$-filtration on $\cM$. This completes the proof of the theorem.
\end{proof}

In fact, once we know that a $\cD$-module has a $V$-filtration, this can be characterized via $b$-functions by the condition in Theorem~\ref{thm_existence_Vfilt}i).
This was first shown by Sabbah \cite{Sabbah}.

\begin{cor}\label{prop_char_V_filt}
If a coherent $\cD_X$-module $\cM$ has a $V$-filtration with respect to $f$, then for every $\alpha\in\QQ$, $V^{\alpha}\iota_+(\cM)$ consists of those sections $w\in\iota_+(\cM)$
with the property that all roots of $b_w$ are $\leq -\alpha$.
\end{cor}

\begin{proof}
We have already seen in Theorem~\ref{thm_existence_Vfilt}i) that if $w\in V^{\alpha}\iota_+(\cM)$, then all roots of $b_w$ are rational numbers $\leq -\alpha$. Conversely,
suppose that all roots of $b_w$ are $\leq -\alpha$. Let $\beta\in\QQ$ be such that $w\in V^{\beta}\iota_+(\cM)$. If $\beta\geq\alpha$, then we are done. If $\beta<\alpha$, since the
$V$-filtration is discrete, we may assume that $w\not\in V^{>\beta}\iota_+(\cM)$ and aim for a contradiction. Since $b_w(s)w\in V^1\cD_{X\times\AA^1}\cdot w\subseteq
V^{\beta+1}\iota_+(\cM)$, we have $b_w(s)\overline{w}=0$ in ${\rm Gr}_V^{\beta}\big(\iota_+(\cM)\big)$. Since $s+\beta$ is nilpotent on ${\rm Gr}_V^{\beta}(\iota_+\cM)$,
while $b_w(s)\overline{w}=0$ and $b_w(s)=\prod_{i=1}^r(s+\alpha_i)$ with $\alpha_i\geq\alpha>\beta$ for all $i$, we conclude that 
$\overline{w}=0$ in  ${\rm Gr}_V^{\beta}(\iota_+\cM)$, a contradiction.
\end{proof}

\begin{rmk}[The roots of the $b$-functions vs. the jumps in the $V$-filtration]
With the notation in Theorem~\ref{thm_existence_Vfilt}ii), if 
$$R=\bigcup_{i=1}^r\big\{\lambda\in\QQ\mid b_{w_i}(-\lambda)=0\big\},$$
then 
$$R\subseteq \big\{\alpha\in\QQ\mid {\rm Gr}_V^{\alpha}(\iota_+\cM)\neq 0\big\}\subseteq R+\ZZ.$$
Indeed, the first inclusion follows from assertion i) in the theorem, while the second inclusion follows from the proof of assertion ii) in the theorem.
\end{rmk}

\begin{eg}[$b$-functions for invertible $f$]\label{b_fcn_invertible_f}
If $f\in\cO_X(X)$ is invertible and $\iota\colon X\hookrightarrow X\times\AA^1$ is the corresponding graph embedding, then it follows from 
Remark~\ref{Vfilt_where_f_invertible} that for every coherent $\cD_X$-module $\cM$, we have a $V$-filtration given by 
$V^{\alpha}\iota_+(\cM)=\iota_+(\cM)$ for all $\alpha\in\QQ$. In this case it follows from Corollary~\ref{prop_char_V_filt} that $b_w=1$ for every 
$w\in \iota_+(\cM)$. 
\end{eg}

\begin{eg}[Holonomic $\cD$-module without a $V$-filtration]
Let $X=\AA^1={\rm Spec}\big(k[x]\big)$ and $f=x$.
The $\cD_X$-module
 $$\cM=\cD_{\AA^1}/\cD_{\AA^1}\cdot (\partial_xx-\lambda), \quad\text{where}\quad \lambda\in k\smallsetminus\QQ,$$
  is holonomic and 
$f$ acts bijectively on $\cM$. It is easy to see that the $b$-function of $\overline{1}\otimes\delta$ is $b(s)=(s+\lambda)$,
hence $\cM$ does not have a $V$-filtration with respect to $f$ by Theorem~\ref{thm_existence_Vfilt}.
\end{eg}

\begin{eg}[Existence of $V$-filtration for $\cO_X$]\label{exist_V_filt_O_X}
Since the Bernstein-Sato polynomial $b_f(s)$ has rational roots by \cite{Kashiwara}, it follows from Theorem~\ref{thm_existence_Vfilt}
that $\cO_X$ has a $V$-filtration with respect to $f$. Since the roots of $b_f(s)$ are negative, we deduce using
Corollary~\ref{prop_char_V_filt} that $\delta\in V^{>0}\iota_+(\cO_X)$.
\end{eg}

A consequence of Theorem~\ref{thm_existence_Vfilt} is that the subcategory of $\cD_X$-modules that have a $V$-filtration with respect to $f$
is closed under extensions:

\begin{cor}\label{cor_thm_existence_Vfilt}
Given a short exact sequence of coherent $\cD_X$-modules
$$0\to \cM'\to\cM\to\cM''\to 0,$$
if $\cM'$ and $\cM''$ have a $V$-filtration with respect to $f$, then so does $\cM$.
\end{cor}

\begin{proof}
By Corollary~\ref{cor_unique_V_filt}, we may and will assume that $X$ is affine. 
After applying $\iota_+$, we get an exact sequence
$$0\to \iota_+(\cM')\hookrightarrow \iota_+(\cM)\overset{p}\longrightarrow \iota_+(\cM'')\to 0.$$
We use the criterion for the existence of $V$-filtrations with respect to $f$ in Theorem~\ref{thm_existence_Vfilt}.
Therefore it is enough to show that every element $u\in \iota_+(\cM)$ has a $b$-function with roots in $\QQ$. 
Since $\cM''$ has a $V$-filtration with respect to $f$, we have a polynomial $b_1(s)$ with roots in $\QQ$ and $Q_1\in V^1\cD_{X\times\AA^1}$ such that
$b_1(s)u -Q_1u\in \iota_+(\cM')$. Since $\cM'$ has a $V$-filtration with respect to $f$, it follows that we have a polynomial $b_2(s)$ with roots in $\QQ$ and $Q_2\in V^1\cD_{X\times\AA^1}$ such that
$$b_2(s)\big(b_1(s)u-Q_1u\big)=Q_2\big(b_1(s)u-Q_1u\big).$$ 
Therefore $b_1(s)b_2(s)u\in V^1\cD_{X\times\AA^1}u$, hence $u$ has a $b$-function that divides $b_1b_2$, and thus has rational roots. 
\end{proof}

\begin{cor}\label{cor_equiv_exist_Vfil}
If $j\colon U\hookrightarrow X$ is the inclusion of the complement of the hypersurface defined by $f$ and 
if $\cM$ is a coherent $\cD_X$-module such that $\cM\vert_U$ is holonomic, then $\cM$ has a $V$-filtration with respect to $f$ if and only if $\cM':=j_+(\cM\vert_U)$ has a $V$-filtration with respect to $f$. Moreover, in this case we have $V^{\alpha}\iota_+(\cM)=V^{\alpha}\iota_+(\cM')$ for all $\alpha>0$.
\end{cor}

\begin{proof}
Since $\cM\vert_U$ is holonomic, we know that $j_+(\cM\vert_U)$ is holonomic, hence coherent. The assertion in the statement then follows from 
Corollaries~\ref{cor_properties_V_filtration1} and 
~\ref{cor_thm_existence_Vfilt}, using the fact that the kernel and the cokernel of the canonical morphism $\gamma\colon \cM\to \cM'$ are supported on the hypersurface defined by $f$, and thus have a $V$-filtration with
 respect to $f$ by Example~\ref{eg_case_support_in_H}. Moreover, these satisfy
 $$V^{\alpha}\iota_+\big({\rm ker}(\gamma)\big)=0=V^{\alpha}\iota_+\big({\rm coker}(\gamma)\big)\quad\text{for all}\quad \alpha>0,$$
 which gives the last assertion in the corollary.
\end{proof} 

\begin{eg}\label{exist_V_filt_O_X2}
For every nonzero $f\in\cO_X(X)$, the $\cD_X$-module $\cO_X[1/f]$ has a $V$-filtration with respect to $f$ and
$$V^{\alpha}\iota_+(\cO_X)=V^{\alpha}\iota_+\big(\cO_X[1/f]\big)\quad\text{for all}\quad \alpha>0.$$ 
This follows from the above corollary and Example~\ref{exist_V_filt_O_X}.
\end{eg}

We end this section by stating without proof a result giving estimates for the roots of certain $b$-functions in terms of a log resolution of $(X,H)$,
where $H$ is the hypersurface defined by $f$. Recall that a log resolution of $(X,H)$ is a projective morphism 
$\pi\colon Y\to X$ that is an isomorphism over $X\smallsetminus H$, with $Y$ smooth, and such that $\pi^*(H)$ is a simple normal crossing divisor. We write
\begin{equation}\label{formula_pull_back_H}
\pi^*(H)=\sum_{i=1}^Na_iE_i\quad\text{and}\quad K_{Y/X}=\sum_{i=1}^Nk_iE_i,
\end{equation}
where $K_{Y/X}$ is the effective divisor on $Y$ defined by the determinant of $\pi^*(\Omega_X)\to \Omega_Y$. 
Such a log resolution exists by Hironaka's theorem.

\begin{thm}\label{thm_DM_bfcn}
Let $X$ be a smooth, irreducible algebraic variety and $f\in\cO_X(X)$ nonzero, defining the hypersurface $H$. 
We also consider a nonnegative integer $m$ and a nonzero $g\in\cO_X(X)$.
Given a log resolution $\pi\colon Y\to X$ of $(X,H)$, with the notation in (\ref{formula_pull_back_H}), the
following hold:
\begin{enumerate}
\item[i)] Every root of $b_f$ is of the form $-\tfrac{k_i+\ell}{a_i}$, for some $i$ with $1\leq i\leq N$ and some positive integer $\ell$.
\item[ii)] Every root of $b_{g\partial_t^m\delta}$ is $\leq-\min\big\{1,\tfrac{k_i+b_i+1}{a_i}-m\mid 1\leq i\leq N\big\}$, where $b_i$ the coefficient of $E_i$ in $\pi^*\big({\rm div}(g)\big)$.
\item[iii)] Every root of $b_{g\delta}$ is $\leq -\min\big\{\tfrac{k_i+b_i+1}{a_i}\mid 1\leq i\leq N\big\}$.
\item[iv)] Every root of $b_{\partial_t^m\delta}$ is either a negative integer or it is of the form $m-\tfrac{k_i+\ell}{a_i}$, for some $i$ with $1\leq i\leq N$ and some positive integer $\ell$.
Furthermore, if $H$ is reduced and the strict transforms of the components of $H$ on $X$ are disjoint\footnote{This condition can always be achieved after performing finitely many further
blow-ups.}, then we may take $i$ such that the divisor $E_i$ is exceptional. 
\end{enumerate}
\end{thm}

The assertion in i) was proved by Lichtin \cite{Lichtin}, improving on Kashiwara's method for proving the rationality of the roots of $b_f$ in \cite{Kashiwara}. The assertions in 
ii)-iv) were proved in \cite{DM}, following Lichtin's argument. We will discuss some applications to invariants of singularities in Section~\ref{section_b_invar_sing}.

\section{An example: weighted homogeneous, isolated singularities}\label{section_weighted_homogeneous}

In this section we discuss the $V$-filtration and the Bernstein-Sato polynomial for weighted homogeneous polynomials, with isolated singularities. 
We consider $X={\mathbf A}^n$, with $\cO_X(X)=R=k[x_1,\ldots,x_n]$ and we write $D_R$ for the ring of differential operators of $R$.
A polynomial $f\in R$ is \emph{weighted homogeneous} if the following condition holds:
there are $w_1,\ldots,w_n\in\QQ_{>0}$ such that if for a monomial $x^u=x_1^{u_1}\cdots x_n^{u_n}\in R$ we put $\rho(x^u)=\sum_{i=1}^nu_iw_i$, then
there is $d$ such that $f=\sum_{\rho(x^u)=d}c_ux^u$ (once the $w_i$ are fixed, we will refer to such polynomials as being
\emph{$w$-homogeneous of degree $d$} and write $\rho(f)=d$). After possibly rescaling all the $w_i$ by the same positive rational number, we may and will assume that
$\rho(f)=1$. We put $|w|:=\sum_{i=1}^nw_i$.
We consider on $R$ the filtration given by shifting by $|w|$ the filtration induced by $\rho$: for every $q\in {\mathbf Q}$, we put
$$R_{\geq q}=\bigoplus_{\rho(u)+|w|\geq q} kx^u\quad\text{and}\quad R_{>q}=\bigoplus_{\rho(u)+|w|> q} kx^u.$$
Note that these are ideals of $R$. 



A special role in our setting is played by the operator $\theta=\sum_{i=1}^nw_ix_i\partial_i$ and by $\sum_{i=1}^nw_i\partial_ix_i=\theta+|w|$. 
Note that if $h$ is $w$-homogeneous, then $\theta(h)=\rho(h)h$, hence an easy computation gives
\begin{equation}\label{eq1_eg_quasihom}
\theta\cdot hf^s=\big(s+\rho(h)\big)hf^s\quad\text{and}\quad (w_1\partial_1x_1+\ldots+w_n\partial_nx_n)\cdot hf^s=\big(s+\rho(h)+|w|\big)hf^s.
\end{equation}

We also note that since $f=\theta(f)$, it follows that $f\in J_f:=\big(\tfrac{\partial f}{\partial x_1},\ldots,\tfrac{\partial f}{\partial x_n}\big)$, hence the zero-locus of $V(J_f)$ is precisely the singular locus of the hypersurface $Z$ defined by $f$.
Furthermore, $Z$ has an isolated singularity at $0$ if and only if $V(J_f)\subseteq \{0\}$. Indeed, if $m$ is a positive integer such that
$mw_i\in\ZZ$ for all $i$, then the $k^*$-action on $X$ given by $\lambda\cdot (u_1,\ldots,u_n)=(\lambda^{mw_1}u_1,\ldots,\lambda^{mw_n}u_n)$ preserves 
$Z$; if $P$ is a singular point of $Z$ different from the origin, then $k^*\cdot P$ is a 1-dimensional subset of the singular locus of $Z$
whose closure contains $0$.

\subsection{The $V$-filtration for weighted homogeneous polynomials with isolated singularities}
From now on we assume that $f$ is $w$-homogeneous, of degree $1$, with an isolated singularity at $0$.
The following result describes the $V$-filtration on $\iota_+(\cO_X)$ in this case. We denote by $\lceil\alpha\rceil$ the smallest integer $\geq\alpha$.

\begin{thm}\label{Vfilt_qhom}
If we define $W^{\alpha}\subseteq\iota_+(R)$ by 
\begin{equation}\label{eq1_Vfilt_qhom}
W^{\alpha}= \left\{
\begin{array}{cl}
\sum_{j\geq 0}D_R\cdot R_{\geq\alpha+j}\partial_t^j\delta, & \text{if}\,\,\alpha\leq 1; \\[2mm]
t^{\lceil\alpha\rceil-1}W^{\alpha-\lceil\alpha\rceil+1}\iota_+(R), & \text{if}\,\,\alpha>1,
\end{array}\right.
\end{equation}
then $V^{\alpha}\iota_+(R)=W^{\alpha}$ for all $\alpha\in\QQ$. 
\end{thm}

 We note that a more precise statement, involving also the Hodge filtration,
is given for the microlocal $V$-filtration in \cite[(4.2.1)]{Saito_filtration}. We give an elementary proof of the above result, by showing that the filtration in the theorem satisfies
the definition of the $V$-filtration. 

\begin{proof}[Proof of Theorem~\ref{Vfilt_qhom}]
We need to show that $(W^{\alpha})_{\alpha\in\QQ}$
satisfies the definition of the $V$-filtration. Since $R_{\geq\beta}\subseteq R_{\geq\gamma}$ for $\beta\geq\gamma$, we see
that in order to check that $W^{\bullet}$ is a decreasing filtration,
it is enough to show that for every $m\in {\mathbf Z}_{>0}$ and $0<\epsilon\ll 1$, we have
$$W^{m+\epsilon}=t^m\cdot W^{\epsilon}\subseteq W^m=t^{m-1}W^1.$$
This follows using the fact that $f$ is $w$-homogeneous with $\rho(f)=1$:
$$t\cdot R_{\geq \epsilon+j}\partial_t^j\delta\subseteq f\cdot R_{\geq\epsilon+j}\partial_t^j\delta+R_{\geq \epsilon+j}\partial_t^{j-1}\delta
\subseteq R_{1+j}\partial_t^j\delta+R_{\geq j}\partial_t^{j-1}\delta\subseteq W^1.$$

The fact that $W^{\bullet}$ is exhaustive 
follows from the fact that $R_{\geq\alpha+j}=R$ if $\alpha\leq -j$. Note also that if $\ell$ is such that $\ell w_i\in {\mathbf Z}$ for all $i$,
then $R_{\geq \lambda}$ is constant for $\lambda$ in any interval of the form $\big(\tfrac{i}{\ell},\tfrac{i+1}{\ell}\big]$, with $i\in {\mathbf Z}$;
therefore the same holds for $W^{\lambda}$. 

Let us show that 
\begin{equation}\label{eq1_condW}
t\cdot W^{\alpha}\subseteq W^{\alpha+1}\quad\text{for all}\quad\alpha\in {\mathbf Q},
\end{equation}
 with equality for $\alpha>0$.
Of course, the assertion for $\alpha>0$ holds by definition, hence we may assume $\alpha\leq 0$. 
By definition of $W^{\alpha}$, it is enough to show that for every $j\in {\mathbf Z}_{\geq 0}$, we have 
$t\cdot R_{\geq\alpha+j}\partial_t^j\delta\subseteq W^{\alpha+1}$. This follows using the fact that 
$f$-is $w$-homogeneous, with $\rho(f)=1$:
$$t\cdot R_{\geq\alpha+j}\partial_t^j\delta\subseteq f\cdot R_{\geq \alpha+j}\partial_t^j\delta+R_{\geq\alpha+j}\partial_t^{j-1}\delta
\subseteq R_{\geq \alpha+j+1}\partial_t^j\delta+R_{\geq\alpha+j}\partial_t^{j-1}\delta\subseteq W^{\alpha+1}.$$

The inclusion 
\begin{equation}\label{eq2_condW}
\partial_t\cdot W^{\alpha}\subseteq W^{\alpha-1}\quad\text{for}\quad\alpha\leq 1
\end{equation}
follows directly from the definition. 
Suppose now that $\alpha\in (m,m+1]$ for some integer $m\geq 2$, in which case we have
$$\partial_t\cdot W^{\alpha}=\partial_t\cdot t^mW^{\alpha-m}\subseteq t^m\partial_t\cdot W^{\alpha-m}+t^{m-1}W^{\alpha-m}
\subseteq t^m\cdot W^{\alpha-m-1}+t^{m-1}\cdot W^{\alpha-m}\subseteq W^{\alpha-1},$$
where we use (\ref{eq1_condW}) and (\ref{eq2_condW}).

We next show that $\partial_tt-\alpha$ is nilpotent on ${\rm Gr}_W^{\alpha}=W^{\alpha}/W^{>\alpha}$ for all $\alpha\in {\mathbf Q}$. 
Suppose first that $\alpha<1$. It is enough to show that if $h\in R$ is $w$-homogeneous, with $\rho(h)+|w|=\alpha+j$,
then $(\partial_tt-\alpha)h\partial_t^j\delta\in D_R\cdot R_{>\alpha+j}\partial_t^j\delta$. To see this, note that  by Lemma~\ref{lem_s} and 
(\ref{eq1_eg_quasihom}), we have
$$(\partial_tt-\alpha)h\partial_t^j\delta=\partial_t^j(\partial_tt-\alpha-j)h\delta=-\sum_{i=1}^n w_i\partial_{x_i}x_ih\partial_t^j\delta
\in D_R\cdot R_{>\alpha+j}\partial_t^j\delta,$$
where we use the fact that $x_ih\in R_{>\alpha+j}$ for all $i$, due to the fact that $w_i>0$. 

We deduce that $(\partial_tt-1)^2\cdot {\rm Gr}_W^1=0$. Indeed, we have already seen that $\partial_tW^1\subseteq W^0$, so by what we have already proved,
we have $(\partial_tt)\partial_t\cdot W^1\subseteq W^{>0}$. We thus get 
$$(\partial_tt-1)^2\cdot W^1=t(\partial_tt\partial_t)\cdot W^1\subseteq t\cdot W^{>0}\subseteq W^{>1}$$
by (\ref{eq1_condW}). Note now that if $\alpha\in (m,m+1]$, for a positive integer $m$, we know that 
$$(\partial_tt-\alpha+m)^k\cdot W^{\alpha-m}\subseteq W^{>\alpha-m},\quad\text{with}\quad
k\in\{1,2\},$$ and
using the definition of $W^{\alpha}$ and Lemma~\ref{lem_s}, we obtain
$$(\partial_tt-\alpha)^k\cdot W^{\alpha}=(\partial_tt-\alpha)^kt^m\cdot W^{\alpha-m}=t^m(\partial_tt-\alpha+m)^k\cdot W^{\alpha-m}\subseteq t^m\cdot W^{>\alpha-m}=W^{>\alpha}.$$


In order to complete the proof of the theorem, it is enough to show that each $W^{\alpha}$ is a finitely generated $D_R[t\partial_t]$-module;
in fact, we will show that it is a finitely generated $D_R$-module. 
Since $W^{\alpha}\subseteq W^1$ for $\alpha\geq 1$ and $W^{\alpha}=t^{\lceil\alpha\rceil-1}W^{\alpha-\lceil\alpha\rceil+1}$ for $\alpha>1$, it follows that
it is enough to show that $W^1$ is a finitely generated $D_R$-module. Since $\dim_k(R/J_f)<\infty$ and $w_i>0$ for all $i$, it follows that 
there is $N$ such that $R_{\geq N+1}\subseteq J_f$. We claim that in this case 
$R_{\geq j+1}\partial_t^j\delta\subseteq D_R\cdot R_{\geq j}\partial_t^{j-1}\delta$ for all $j\geq N$. 
Indeed, given $g\in R_{\geq j+1}$, which we may assume $w$-homogeneous, we have $h\in J_f$, hence we can write $h=\sum_{i=1}^n\tfrac{\partial f}{\partial x_i}g_i$,
with $g_i\in R_{\geq j+w_i}$ (recall that $\rho(f)=1$, hence $\rho\big(\tfrac{\partial f}{\partial x_i}\big)=1-w_i$). Since we can write
$$h\partial_t^j\delta=\sum_{i=1}^n\partial_{x_i}\cdot g_i\partial_t^{j-1}\delta-\sum_{i=1}^n\tfrac{\partial g}{\partial x_i}\partial_t^{j-1}\delta$$
and $g_i,\tfrac{\partial g_i}{\partial x_i}\in R_{\geq j}$, this proves our assertion. This in turn implies that
$$W^1=\sum_{j=0}^{N}D_R\cdot R_{\geq j+1}\partial_t^j\delta,$$
hence $W^1$ is a finitely generated $D_R$-module. This completes the proof of the theorem.
\end{proof}

\subsection{The $b$-function of weighted homogeneous polynomials with isolated singularities}
We next turn to the description of the Bernstein-Sato polynomial of $f$, that will require a bit more work.
We keep the assumption that $f$ is $w$-homogeneous, of degree $1$, with an isolated singularity at $0$.
Recall that $R/J_f$ is a finite-dimensional $k$-vector space.
We put 
$$\Sigma(f)=\{\rho(g)\mid g\in R\smallsetminus J_f,\,g\,\,\text{is}\,\,w-\text{homogeneous}\}.$$
Note that this is a finite set: if $x^{u_1},\ldots,x^{u_r}$ are monomials whose classes in $R/J_f$ form a basis,
then 
$\Sigma(f)=\big\{\rho(u_1),\ldots,\rho(u_r)\big\}$. Indeed, the only thing to note is that if $h\not \in J_f$
is $w$-homogeneous and if we write $h=\sum_{i=1}^rc_ix^{u_i}+g$, where $c_i\in k$ for all $i$ and $g\in J_f$, since $J_f$ is generated
by $w$-homogeneous elements, we may assume that $g$ is $w$-homogeneous, with $\rho(h)=\rho(g)=\rho(x^{u_i})$
for all $i$ such that $c_i\neq 0$ (note also that some $c_i$ must be nonzero since $g\not\in J_f$).

\begin{thm}[\cite{BGM}]\label{thm_bfcn_qhom}
The Bernstein-Sato polynomial of $f$ is given by
\begin{equation}\label{eq1_thm_bfcn_qhom}
b_f(s)=(s+1)\cdot\prod_{\lambda\in\Sigma(f)}\big(s+\lambda+|w|).
\end{equation}
\end{thm}

Before giving the proof of the theorem, we need to make some preparations. 
We begin with the following

\begin{lem}\label{lem_form_qhomog}
If $h\in R$ is $w$-homogeneous, then
\begin{equation}\label{eq3_thm_bfcn_qhom}
\prod_{\lambda\in\Sigma(f), \lambda\geq \rho(h)}\big(s+\lambda+|w|)hf^s\in D_R\cdot J_ff^s.
\end{equation}
\end{lem}

\begin{proof}
Note first that (\ref{eq3_thm_bfcn_qhom}) clearly holds if $h\in J_f$, hence from now on we assume that $h\not\in J_f$ and argue
by descending induction on $\rho(h)$. Suppose first that $h$ is such that $\rho(h)$ is the largest element of $\Sigma(f)$, in which case we have $x_ih\in J_f$ for $1\leq i\leq n$. 
In this case it follows from (\ref{eq1_eg_quasihom}) that we have
$$\big(s+\rho(h)+|w|\big)hf^s=\sum_{i=1}^nw_i\partial_i(x_ihf^s)\in D_R\cdot J_ff^s.$$

Suppose now that $h\not\in J_f$ and we know that (\ref{eq3_thm_bfcn_qhom}) holds for all $h'$ that are $w$-homogeneous, with $\rho(h')>\rho(h)$. In particular, it holds 
for $x_ih$ for all $i$, hence 
$$\prod_{\lambda'\in\Sigma(f),\lambda'>\rho(h)}\big(s+\lambda'+|w|\big)x_ihf^s\in D_R\cdot J_ff^s\quad\text{for}\quad 1\leq i\leq n.$$
We thus deduce using (\ref{eq1_eg_quasihom}) that
$$\prod_{\lambda\in\Sigma(f),\lambda\geq \rho(h)}\big(s+\lambda+|w|\big)hf^s=\prod_{\lambda\in\Sigma(f),\lambda>\rho(h)}\big(s+\lambda+|w|\big)\cdot \big(s+\rho(h)+|w|\big)hf^s$$
$$=\prod_{\lambda\in\Sigma(f),\lambda>\rho(h)}\big(s+\lambda+|w|\big)\cdot \sum_{i=1}^nw_i\partial_i(x_ih)f^s\subseteq D_R\cdot J_ff^s.$$
This completes the proof of the induction step.
\end{proof}

Following \cite[Theorem~2.19]{Yano}, we next prove a general result concerning ${\rm Ann}_{\cD_X}(f^s)$ for isolated hypersurface singularities. 
As suggested by the functional equation (\ref{eq_thm_b_fcn}), an understanding of this annihilator could be useful for computing $b_f(s)$.

\begin{prop}\label{lem_Yano}
Let $X$ be a smooth, irreducible variety and $P\in X$ a point lying on the hypersurface $Z$ defined by $f\in\cO_X(X)$. 
If $Z$ has an isolated singularity at $P$ and $x_1,\ldots,x_n$ are algebraic coordinates in a neighborhood of $P$, then 
the left ideal 
$${\rm Ann}_{\cD_X}(f^s)=\{Q\in \cD_X\mid Q\cdot f^s=0\}$$
is generated in a neighborhood of $P$ by $\tfrac{\partial f}{\partial x_i}\partial_{x_j}-\tfrac{\partial f}{\partial x_j}\partial_{x_i}$, for $1\leq i<j\leq n$.
In particular, we have 
$$\{Q\in \cD_X\mid Q\cdot f^s\in\cD_X\cdot J_ff^s\}\subseteq \cD_X\cdot J_f\quad\text{in}\quad\cD_X,$$
where $J_f=\big(\tfrac{\partial f}{\partial x_1},\ldots,\tfrac{\partial f}{\partial x_n}\big)$.
\end{prop}

\begin{proof}
It is clear that the left ideal $I$ of $\cD_X$ generated by $\tfrac{\partial f}{\partial x_i}\partial_{x_j}-\tfrac{\partial f}{\partial x_j}\partial_{x_i}$, for $1\leq i<j\leq n$,
is contained in ${\rm Ann}_{\cD_X}(f^s)$, hence we need to prove the reverse inclusion. The assertion is easy to check when $P$ is a smooth point of $Z$.
Indeed, if $\tfrac{\partial f}{\partial x_i}(P)\neq 0$ and $Q\in {\rm Ann}_{\cD_X}(f^s)$, after writing $Q$ modulo the left ideal generated by 
$\partial_{x_j}-\big(\tfrac{\partial f}{\partial x_i}\big)^{-1}\tfrac{\partial f}{\partial x_j}\partial{x_i}$, we may assume that $Q\in\cO_X[\partial_{x_i}]$ and we need to show that $Q=0$. 
This follows from the fact that $\partial_{x_i}^mf^s=Q_mf^s$, where $Q_m\in\cO_X[1/f,s]$ is a polynomial of degree $m$ in $s$. 

From now on we assume that $P$ is a singular point of $Z$. After possibly replacing $X$ by a suitable open neighborhood of $P$, we may assume that $X$ is affine, with 
$\cO_X(X)=R$, 
that $x_1,\ldots,x_n$ are defined on $X$,
and that $\tfrac{\partial f}{\partial x_1},\ldots,\tfrac{\partial f}{\partial x_n}$ form a regular sequence in $R$. We need to show that if $Q\in \Gamma(X,\cD_X)$ is such that $Q\cdot f^s=0$, then
$Q$ lies in $I$. We argue by induction on the order $q$ of $Q$, the case $q=0$ being trivial: in this case it is clear that $Q=0$. Note that for every $\alpha\in\ZZ_{\geq 0}^n$,
we can write $\partial_x^{\alpha}f^s=g_{\alpha}f^{s-|\alpha|}$, where $g_{\alpha}\in R[s]$ has degree $|\alpha|$ in $s$, with the coefficient of the top degree term equal to 
$\prod_i\big(\tfrac{\partial f}{\partial x_i}\big)^{\alpha_i}$. 
We can write $Q=\sum_{|\alpha|\leq q}Q_{\alpha}\partial_x^{\alpha}$, with $Q_{\alpha}\in\cO_X$ for all $\alpha$, and we see that 
$$\sum_{|\alpha|=q}Q_{\alpha}\cdot\prod_{i=1}^n\big(\tfrac{\partial f}{\partial x_i}\big)^{\alpha_i}=0.$$
Since $\tfrac{\partial f}{\partial x_1},\ldots,\tfrac{\partial f}{\partial x_n}$ is a regular sequence in $R$, it follows that in the ring $R[y_1,\ldots,y_n]$ we can write
$$\sum_{|\alpha|=q}Q_{\alpha}y_1^{\alpha_1}\cdots y_n^{\alpha_n}=\sum_{i<j}h_{i,j}\big(\tfrac{\partial f}{\partial x_i}y_j-\tfrac{\partial f}{\partial x_j}y_i\big),$$
for some $h_{i,j}\in R[y_1,\ldots,y_n]$ that are homogeneous in $y_1,\ldots,y_n$, of degree $q-1$. 
If $h_{i,j}=\sum_{\beta}h_{i,j,\beta}y^{\beta}$, with $h_{i,j,\beta}\in R$, then the difference 
$$\widetilde{Q}=Q-\sum_{i,j,\beta}h_{i,j,\beta}\partial_x^{\beta}\big(\tfrac{\partial f}{\partial x_i}\partial_{x_j}-\tfrac{\partial f}{\partial x_j}\partial_{x_i}\big)\in D_R$$
has order $q-1$ and $\widetilde{Q}f^s=0$. We conclude by induction that $\widetilde{Q}\in I$ and thus $Q\in I$, too. 

The last assertion in the proposition is clear once we note that
$$\tfrac{\partial f}{\partial x_i}\partial_{x_j}-\tfrac{\partial f}{\partial x_j}\partial_{x_i}=\partial_{x_j}\tfrac{\partial f}{\partial x_i}-\tfrac{\partial^2 f}{\partial x_j\partial x_i}
-\partial_{x_i}\tfrac{\partial f}{\partial x_j}+\tfrac{\partial^2 f}{\partial x_i\partial x_j}
=\partial_{x_j}\tfrac{\partial f}{\partial x_i}-\partial_{x_i}\tfrac{\partial f}{\partial x_j}\in\cD_X\cdot J_f.$$
\end{proof}

We can now give the proof of the formula for the Bernstein-Sato polynomial for weighted homogeneous polynomials with an isolated singularity, following
\cite{BGM}.

\begin{proof}[Proof of Theorem~\ref{thm_bfcn_qhom}]
We first show that $b_f(s)$ divides $(s+1)\cdot\prod_{\lambda\in\Sigma(f)}\big(s+\lambda+|w|)$, or equivalently, that
\begin{equation}\label{eq2_thm_bfcn_qhom}
(s+1)\cdot\prod_{\lambda\in\Sigma(f)}\big(s+\lambda+|w|)f^s\in D_R[s]f^{s+1}.
\end{equation}
This follows by applying Lemma~\ref{lem_form_qhomog} for $h=1$: we can write
$$\prod_{\lambda\in\Sigma(f), \lambda\geq 0}\big(s+\lambda+|w|)f^s=\sum_{i=1}^nP_i\cdot \tfrac{\partial f}{\partial x_i}f^s$$
for some $P_1,\ldots,P_n\in D_R$, in which case, we have
$$(s+1)\cdot\prod_{\lambda\in\Sigma(f)}\big(s+\lambda+|w|)f^s=\sum_{i=1}^n(s+1)P_i\tfrac{\partial f}{\partial x_i}f^s=(P_1\partial_1+\ldots+P_n\partial_n)\cdot f^{s+1},$$
hence (\ref{eq2_thm_bfcn_qhom}) holds. 

We next need to show that $(s+1)\cdot\prod_{\lambda\in\Sigma(f)}\big(s+\lambda+|w|)$ divides $b_f(s)$, or equivalently, that for every $h\in R\smallsetminus J_f$ that is $w$-homogeneous,
if $\lambda=\rho(h)$, then $\widetilde{b}_f\big(-\lambda-|w|\big)=0$. 
By definition of the reduced Bernstein-Sato polynomial, we can write
\begin{equation}\label{eq5_thm_bfcn_qhom}
(s+1)\widetilde{b}_f(s)f^s=P\cdot f^{s+1},
\end{equation}
for some $P\in D_R[s]$. 
Let us write $P=\sum_{i=1}^nP_i\partial_i+Q$, where
$Q\in R[s]$ and $P_i\in D_R[s]$. The equality (\ref{eq5_thm_bfcn_qhom}) becomes 
$$(s+1)\widetilde{b}_f(s)f^s=Q\cdot f^{s+1}+(s+1)\cdot \sum_{i=1}^nP_i\tfrac{\partial f}{\partial x_i}\cdot f^s.$$
We can thus write $Q=(s+1)T$ for some $T\in R[s]$ and we have
$$\widetilde{b}_f(s)f^s=Q_0\cdot f^{s+1}+\sum_{i=1}^nP_i\tfrac{\partial f}{\partial x_i}\cdot f^s\in D_R[s]\cdot J_ff^s,$$
where we use the fact that $f\in J_f$.

Note also that 
\begin{equation}\label{eq6_thm_bfcn_qhom}
D_R[s]\cdot J_ff^s\subseteq D_R\cdot J_ff^s.
\end{equation} 
Indeed, we have $s f^s=\theta f^s$ and 
$\big[\theta,\tfrac{\partial f}{\partial x_i}\big]=\theta\big(\tfrac{\partial f}{\partial x_i}\big)=(1-w_i)\tfrac{\partial f}{\partial x_i}$
(due to the fact that $\tfrac{\partial f}{\partial x_i}$ is $w$-homogeneous with $\rho\big(\tfrac{\partial f}{\partial x_i}\big)=1-w_i$),
hence 
$$s\tfrac{\partial f}{\partial x_i}f^s=(\theta-1+w_i)\tfrac{\partial f}{\partial x_i}f^s.$$
We thus conclude that 
$\widetilde{b}_f(s)f^s\in D_R\cdot J_ff^s$, hence also
\begin{equation}\label{eq7_thm_bfcn_qhom}
\widetilde{b}_f(s)hf^s\in D_R\cdot J_ff^s.
\end{equation}

Note now that by (\ref{eq1_eg_quasihom}), we have
$$\left(\sum_{i=1}^nw_i\partial_ix_i\right)\cdot hf^s=\big(s+\lambda+|w|\big)hf^s.$$
Let us write $\widetilde{b}_f(s)=\big(s+\lambda+|w|\big)q(s)+a$, for some $q\in k[s]$ and some $a\in k$. 
Our goal is to show that $a=0$. Since 
$$\widetilde{b}_f(s)hf^s=ahf^s+\left(\sum_{i=1}^nw_i\partial_ix_i\right)\cdot q(s)hf^s\in D_R\cdot J_ff^s$$
and since 
$q(s)f^s\in D_R\cdot f^s$ (this follows using again $sf^s=\theta f^s$), we conclude that there is
$P\in D_R$ such that 
$$ahf^s+\left(\sum_{i=1}^nw_i\partial_ix_i\right)\cdot Pf^s\in D_R\cdot J_ff^s,$$
in which case, we conclude from Proposition~\ref{lem_Yano} that
$$ah+\left(\sum_{i=1}^nw_i\partial_ix_i\right)\cdot P\in D_R\cdot J_f.$$
Since we have a direct sum decomposition $D_R=\bigoplus_{\alpha\in\ZZ_{\geq 0}^n}\partial^{\alpha}R$
that induces the decomposition $D_R\cdot J_f=\bigoplus_{\alpha\in\ZZ_{\geq 0}^n}\partial^{\alpha}J_f$, 
we conclude that $ah\in J_f$. Since $h\not\in J_f$, it follows that $a=0$, hence $\widetilde{b}_f\big(-\lambda-|w|\big)=0$.
This completes the proof of the theorem.
\end{proof}

\section{Duality and $V$-filtrations}\label{section_behavior_push_forward}

Our goal in this section is to discuss the compatibility of $V$-filtrations with duality.
Let $X$ be a smooth, irreducible, $(n+1)$-dimensional algebraic variety. Recall that on $X$ we have
 the duality functor $(-)^*$, see (\ref{eq_def_duality_functor}), such that
 for a left $\cD_X$-module $\cM$, 
the right $\cD_X$-module corresponding to $\cH^j(\cM^*)$ is ${\mathcal Ext}_{\cD_X}^{j+n+1}(\cM,\cD_X)$.
The following theorem allows us to relate the graded pieces of the $V$-filtration to the duality functor.




\begin{thm} \label{thm_duality} $($\cite[Prop. 4.6-2]{MaisonobeMebkhout}, \cite[Theorem 1.6]{SaitoDuality}$)$ Let $X$ be a smooth, irreducible algebraic variety and $H\subseteq X$ a smooth hypersurface, defined by $t \in \cO_X(X)$. If $\cM$ is a $\cD_X$-module that has a $V$-filtration $V^\bullet \cM$ along $H$, then 
all $\cH^j(\cM^*)$ have such a $V$-filtration and
we have the following functorial isomorphisms for all $j\in \mathbf Z$: 
\begin{enumerate} 
\item[i)] $\delta_1\colon  {\rm Gr}_V^1\big(\cH^j(\cM^*)\big) \overset{\sim}\longrightarrow \cH^j\big(({\rm Gr}_V^1 \cM)^*\big)$,
\item[ii)] $\delta_0\colon {\rm Gr}_V^0\big(\cH^j(\cM^*)\big) \overset{\sim}\longrightarrow \cH^j\big(({\rm Gr}_V^0\cM)^*\big)$,
\item[iii)] $\delta_\alpha\colon {\rm Gr}_V^\lambda\big((\cH^j(\cM^*)\big) \overset{\sim}\longrightarrow \cH^j\big(({\rm Gr}_V^{1-\lambda}\cM)^*\big)$ for all $\lambda \in (0,1)$
\end{enumerate}
such that if we denote by $N$ the nilpotent endomorphism $(s+\lambda)$ on ${\rm Gr}_V^{\lambda}(-)$ and ${\rm can}, {\rm Var}$ are the maps from formulas (\ref{defn-can})
and (\ref{defn-Var}), then
\begin{enumerate} \item[iv)] $N^* \circ \delta_\lambda = \delta_{\lambda} \circ (-N)$,
\item[v)] ${\rm can}^* \circ \delta_1 = \delta_0 \circ {\rm Var}$,
\item[vi)] $(- {\rm Var})^* \circ \delta_0 = \delta_1 \circ {\rm can}$.
\end{enumerate}
\end{thm}

We explain the key ingredients in the proof of the above theorem, following the argument in \cite{MebkhoutSabbah,MaisonobeMebkhout}.
Note that a different proof of the same result for Hodge modules, keeping track also of the Hodge filtration, can be found in \cite{SaitoDuality}.
The main idea is to show that locally, every $\cD$-module which admits a $V$-filtration has a resolution by ``elementary" $\cD$-modules, and then prove the theorem for such modules. 

\begin{defi}\label{defi_elementary}
Let $b(s) \in \mathbf Q[s]$ be a nonzero polynomial such that $b(-\lambda) = 0$ implies $\lambda \in [0,1)\cap\QQ$. We denote by $m$ the multiplicity of the root $0$ in $b(s)$. Let $Q$ be a $p\times q$ matrix with entries in $\Gamma(X, V^0\cD_X)$ and $P$ a $q\times q$ matrix with entries in $\Gamma(X, V^1 \cD_X)$. We denote by $\Lambda = \Lambda(b,P,Q)$ the 
$(p+q)\times (p+q)$-matrix
\[ \begin{pmatrix} s^m \text{Id}_p & -Q \\ 0 & (s+1)b(s+1)\text{Id}_q -P\end{pmatrix},\]
and we let $\cE = \cE(b,P,Q)$ be the cokernel of the $\cD_X$-linear morphism
\[ \cD_X^{p} \oplus \cD_X^{q} \xrightarrow[]{\cdot \Lambda} \cD_X^{p} \oplus \cD_X^{q},\]
which we call an \emph{elementary $\cD_X$-module}.
\end{defi}

\begin{lem} \label{lem_elementaryresolution} $($\cite[Lem. 4.5-2]{MaisonobeMebkhout}$)$ 
Suppose that $X$ is an affine variety.
If $\cM$ is a $\cD_X$-module that admits a $V$-filtration along $H$, then $\cM$ is a quotient of an elementary $\cD_X$-module.
\end{lem}

\begin{proof} 
Let $b_0(s)$ be the minimal polynomial of $s = -\de_t t$ on $V^0\cM/V^1\cM$.
Note that by condition ii) in the definition of the $V$-filtration, all roots of $b_0(-s)$ lie in $[0,1)\cap \QQ$.
Let $m$ be the multiplicity of the factor $s$ in $b_0(s)$, let $v_1,\dots, v_p$ be global sections of $V^0\cM$ which generate $\ker(s^m) \subseteq V^0\cM/V^1\cM$ over $V^0\cD_X$,
and let $u_1,\dots, u_q$ be global sections of $V^1\cM$ generating it over 
$V^0\cD_X$. The $\cD_X$-submodule generated by $u_1,\dots, u_q,v_1,\dots, v_p$ is $\cM$: indeed, it contains $V^0\cM$ since
$V^0\cM\subseteq V^1\cM+\sum_{i=1}^pV^0\cD_X\cdot v_i+\partial_t\cdot V^1\cM$, and then we conclude that it is equal to $\cM$ using
Corollary~\ref{cor_rmk1_Vfiltration}iii).

To understand the relations among these generators, note that by construction we have
\begin{equation}\label{eq_lem1_duality}
s^m v_i = \sum_{j=1}^q Q_{i,j}u_j,\quad\text{with}\quad Q_{i,j}\in\Gamma(X, V^0\cD_X).
\end{equation} 
Moreover, we can also write
\begin{equation}\label{eq_lem2_duality}
(s+1)b_0(s+1) u_i = \sum_{j=1}^q P_{i,j}u_j,\quad\text{where} \quad P_{i,j} \in \Gamma(X, V^1\cD_X).
\end{equation}
Indeed, since $s^m {\rm Gr}_V^0\cM = 0$, using Lemma~\ref{lem_s}, we see that
$$(s+1)^{m+1}\cdot V^1\cM=t\partial_t(s+1)^{m} \cdot V^1\cM 
=ts^{m}\partial_t\cdot V^1\cM\subseteq ts^{m}\cdot V^0\cM\subseteq t\cdot V^{>0}\cM\subseteq V^{>1}\cM.$$
Using also the fact that $V^{\alpha+1}\cM=t\cdot V^{\alpha}\cM$ for $\alpha>0$, Lemma~\ref{lem_s}, and the
definition of $b_0(s)$, we conclude that
$$(s+1)\cdot b_0(s+1)\cdot V^1\cM\subseteq V^2\cM=t\cdot V^1\cM\subseteq\sum_{j=1}^q V^1\cD_X\cdot u_j,$$
giving (\ref{eq_lem2_duality}). We then deduce from (\ref{eq_lem1_duality}) and (\ref{eq_lem2_duality})  that $\cM$ is a quotient of $\cE(b_0,P,Q)$
\end{proof}

The following lemma gives more details on the structure of elementary modules. 

\begin{lem} \label{lem_elementary} 
With the notation in Definition~\ref{defi_elementary}, if $\cE=\cE(b,P,Q)$, then the following hold:
\begin{enumerate}
\item[i)] The
sequence
\begin{equation}\label{eq1_lem_elementary}
0 \to \cD_X^p \oplus \cD_X^q \xrightarrow[]{\cdot \Lambda} \cD_X^p \oplus \cD_X^q \xrightarrow[]{\pi} \cE \to 0
\end{equation}
is exact. 
\item[ii)] If we put $U^k\cE:=\pi\big((V^k \cD_X)^p \oplus (V^{k-1}\cD_X)^q\big)$ for $k\in\ZZ$, then we have an induced exact sequence
\begin{equation}\label{eq2_lem_elementary}
0\to {\rm Gr}_V^k(\cD_X)^p\oplus {\rm Gr}_V^{k-1}(\cD_X)^{q}\to {\rm Gr}_V^k(\cD_X)^p\oplus {\rm Gr}_V^{k-1}(\cD_X)^{q}\to U^k\cE/U^{k+1}\cE\to 0.
\end{equation}
\item[iii)] The filtration $U^{\bullet}\cE$ is a $\ZZ$-indexed $V$-filtration on $\cE$ with respect to $t$ in the sense of Definition~\ref{defi_Z_filtration}.
Therefore $\cE$ has a $V$-filtration with respect to $t$, we have $V^k\cE=U^k\cE$
for every $k\in\ZZ$, and we thus have
 an isomorphism of ${\rm Gr}_V^0\cD_X = \cD_H[s]$-modules
\begin{equation} \label{VFiltElementary} 
V^k \cE/V^{k+1}\cE \simeq \left( \frac{\cD_H[s]}{(s+k)^m}\right)^p \oplus \left( \frac{\cD_H[s]}{(s+k)b(s+k)}\right)^q.\end{equation}
\end{enumerate}
\end{lem}

\begin{proof} 
Let $\Lambda = \Lambda(b,P,Q)$ be a matrix as in the definition of the elementary module $\cE(b,P,Q)$. Assume $(A,B) \in \cD_X^p \oplus \cD_X^q$ lies in the kernel of $\cdot \Lambda$, that is, we have
\[ A s^m=0\quad\text{and}\quad  -AQ + B\big((s+1)b(s+1) -P\big)=0.\]

Since every nonzero element of $k[s]$ acts injectively on $\cD_X$,
it is clear that $A=0$, hence
$B(s+1)b(s+1) = BP$. Note also that the nonzero polynomials in $k[s]$ act injectively on ${\rm Gr}_V^{\bullet}(\cD_X)$
(this follows easily from the formulas (\ref{formula_description_V0}), (\ref{eq1_V_D}), and (\ref{eq2_V_D})). Since the entries of $P$ 
lie in $V^1\cD_X$, we have $B=0$. Therefore the map $\cdot \Lambda$ is injective and (\ref{eq1_lem_elementary}) is exact.

In order to prove ii), 
note first that since $Q$ has entries in $V^0\cD_X$ and $P$ has entries in $V^1\cD_X$, the map $\cdot \Lambda$
induces a map 
$${\rm Gr}^k_V(\cdot\Lambda)\colon {\rm Gr}_V^k(\cD_X)^p\oplus {\rm Gr}_V^{k-1}(\cD_X)^{q}\to {\rm Gr}_V^k(\cD_X)^p\oplus {\rm Gr}_V^{k-1}(\cD_X)^{q}$$
represented by the matrix
\[ \begin{pmatrix} s^m \text{Id}_p & 0 \\ 0 & (s+1) b(s+1) \text{Id}_q \end{pmatrix}.\]
We see that this is injective, using again the fact that nonzero polynomials in $k[s]$ act injectively on ${\rm Gr}_V^{\bullet}(\cD_X)$. 
The exactness of (\ref{eq2_lem_elementary}) is now clear.

Moreover, we see that 
\[ U^k \cE/U^{k+1}\cE \simeq \left( \frac{{\rm Gr}_V^k \cD_X}{{\rm Gr}_V^k \cD_X s^m}\right)^p \oplus \left( \frac{{\rm Gr}_V^{k-1}\cD_X}{{\rm Gr}_V^{k-1}\cD_X (s+1)b(s+1)}\right)^q\] \[\simeq \left( \frac{{\rm Gr}_V^k \cD_X}{(s+k)^m {\rm Gr}_V^k \cD_X}\right)^p \oplus \left( \frac{{\rm Gr}_V^{k-1}\cD_X}{(s+k)b(s+k) {\rm Gr}_V^{k-1}\cD_X}\right)^q,\]
where the second isomorphism follows easily from Lemma~\ref{lem_s} and the formulas (\ref{formula_description_V0}), (\ref{eq1_V_D}), and (\ref{eq2_V_D}).
This implies that the polynomial $p(x)=xb(x)$, whose roots are all in $[0,1)$, has the property that $p(\partial_tt-k)\cdot U^k\cE\subseteq U^{k+1}\cE$ for all $k\in\ZZ$. Since it is straightforward to see that
$U^{\bullet}\cE$ also satisfies properties a)-c) in Definition~\ref{defi_Z_filtration}, we conclude that $U^{\bullet}\cE$ is a $\ZZ$-indexed $V$-filtration on $\cE$
with respect to $t$. The fact that $\cE$ has a $V$-filtration and $V^k\cE=U^k\cE$ for all $k\in\ZZ$ now follows from Proposition~\ref{prop_Z_filtration}.
\end{proof}

The free resolution given in Lemma~\ref{lem_elementary}  allows us to compute the dual of an elementary $\cD_X$-module. As we will see shortly, the dual has its own free resolution and we can use it to compute the $V$-filtration. 

Working locally, in what follows we assume that we have a system of coordinates $x_1,\ldots,x_n,t$ on $X$.
Recall that we have the classical adjoint given by (\ref{eq_classical_adjoint}).
For example, we have ${}^t s = t\partial_t =-s-1$.
For a matrix of differential operators $\Lambda = (\Lambda_{ij})$, we denote by 
${}^t \Lambda$ the matrix $({}^t \Lambda_{ji})$,
that is, the transpose of the matrix of classical adjoints of the entires of $\Lambda$.

\begin{lem}  \label{lem_dualelementary} 
With the notation in Definition~\ref{defi_elementary}, if $\cE= \cE(b,P,Q)$ is an elementary $\cD_X$-module, with $\Lambda$ the matrix of differential operators defining $\cE$,
then the following hold:
\begin{enumerate}
\item[i)]  We have $\cH^i(\cE^*) = 0$ for all $i\neq -n$. 
\item[ii)] If $\widetilde{\cE} = \widetilde{\cE}(b,P,Q):= \cH^{-n}(\cE^*)$, then we have an exact sequence
\begin{equation}\label{eq1_duallem_elementary}
0 \to \cD_X^p \oplus \cD_X^q \xrightarrow[]{\cdot {}^t \Lambda} \cD_X^p \oplus \cD_X^q \xrightarrow[]{\pi} \widetilde{\cE} \to 0.
\end{equation}
\item[iii)] If we put $W^k\widetilde{\cE}=\pi\big((V^k\cD_X)^p\oplus (V^{k+1}\cD_X)^q\big)$ for all $k\in\ZZ$, then we have an induced exact sequence 
\[ 0 \to  {\rm Gr}_V^k(\cD_X)^p \oplus  {\rm Gr}_V^{k+1}(\cD_X)^q \xrightarrow[]{\cdot {}^t \Lambda} {\rm Gr}_V^k(\cD_X)^p \oplus {\rm Gr}_V^{k+1}(\cD_X)^q \xrightarrow[]{\pi} W^k \widetilde{\cE}/W^{k+1}\widetilde{\cE} \to 0.\]
\item[iv)] The polynomial $p(s)=(s+1)b(-s-1)$ satisfies $p(s+k)\cdot W^k\widetilde{\cE}\subseteq W^{k+1}\widetilde{\cE}$
for all $k\in\ZZ$. Therefore $\widetilde{\cE}$ has a $V$-filtration with respect to $t$, with
$W^k \widetilde{\cE} = V^{>k} \widetilde{\cE}$ for all $k\in \mathbf Z$, and we have an isomorphism of $\cD_H[s]$-modules
\begin{equation} \label{VFiltDualElementary} 
V^{>k} \widetilde{\cE}/V^{>k+1} \widetilde{\cE} 
\simeq\left( \frac{\cD_H[s]}{(s+k+1)^m \cD_H[s]} \right)^{p} \oplus \left( \frac{\cD_H[s]}{(s+k+1)b(-s-k-1) \cD_H[s]}\right)^{q}.
\end{equation}
\end{enumerate}
\end{lem}

\begin{proof}
We compute $\cE^*$ using the free resolution in Lemma \ref{lem_elementary}. Applying 
$\cH om_{\cD_X}(-,\cD_X)$ to this resolution and passing to left $\cD_X$-modules, we see that $\cE^*$ is given by the complex 
\[ \cD_X^{\oplus p} \oplus \cD_X^{\oplus q} \xrightarrow[]{\cdot {}^t \Lambda} \cD_X^{\oplus p} \oplus \cD_X^{\oplus q},\]
placed in degrees $-n-1$ and $-n$, where 
\[ {}^t \Lambda = \begin{pmatrix} (-s-1)^m \text{Id}_p & 0 \\ -{}^t Q & -sb(-s)\text{Id}_q - {}^t P\end{pmatrix}.\]
Arguing as in the proof of Lemma~\ref{lem_elementary}, we see that the map $\cdot {}^t \Lambda$ is injective, giving the assertions in i) and ii),
and that it induces the exact sequence in iii). The assertion about $p(s)$ follows easily and we deduce the existence of the $V$-filtration and the relation 
with $W^{\bullet}\widetilde\cE$ from Remark~\ref{rmk_var_Z_filtration}.
\end{proof}

We next explain how to relate the $V$-filtrations on $\cE=\cE(b,P,Q)$ and $\widetilde{\cE}=\widetilde{\cE}(b,P,Q)$.
For every $\lambda\in\QQ$ and every positive integer $d$, we consider the $k$-bilinear map 
$$k[s]/(s-\lambda)^d\times k[s]/(s-\lambda)^d\to k$$
that maps $\big((s-\lambda)^i,(s-\lambda)^j\big)$ to $1$ if $i+j=d-1$ and to $0$, otherwise. We get in this way a $k$-linear isomorphism
\begin{equation}\label{isom_nat1}
k[s]/(s-\lambda)^d\to {\rm Hom}_k\big(k[s]/(s-\lambda)^d,k\big).
\end{equation} 
Note now that we have an isomorphism
\begin{equation}\label{isom_nat2}
\cH^{-n}\big(\cD_H[s]/(s-\lambda)^d)^*\simeq \cD_H^{\ell}\otimes_k{\rm Hom}_k\big(k[s]/(s-\lambda)^d,k\big),
\end{equation}
where $\cD_H^{\ell}$ is the left $\cD_X$-module corresponding to the standard right $\cD_H$-module structure on $\cD_H$, and we thus get an isomorphism
\begin{equation}\label{isom_nat3}
{\rm nat}_{\lambda,d}\colon \cD_H[s]/(s-\lambda)^d\overset{\sim}\longrightarrow \cD_H^{\ell}[s]/(s-\lambda)^d\overset{\sim}\longrightarrow \cH^{-n}\big(\cD_H[s]/(s-\lambda)^d)^*,
\end{equation}
where the first isomorphism is induced by applying the classical adjoint in $\cD_H$ to the coefficients 
and the second one is obtained by tensoring (\ref{isom_nat1}) with $\cD_X^{\ell}$
and using (\ref{isom_nat2}).

Let $\cE = \cE(b,P,Q)$ be an elementary module and $\widetilde{\cE}\simeq \cH^{-n}(\cE^*)$ the corresponding
$\cD_X$-module described in Lemma~\ref{lem_dualelementary}.
We first relate ${\rm Gr}_V^{1-\lambda}(\cE)$ and ${\rm Gr}_V^\lambda(\widetilde{\cE})$ for $\lambda \in (0,1)$. 
Let $m_{\lambda}$ be the multiplicity of $\lambda-1$ as a root of $b(s)$. 
Since
$${\rm Gr}_V^{1-\lambda}(\cE)\simeq\big\{u\in V^0\cE/V^1\cE\mid (s+1-\lambda)^{m_{\lambda}}u=0\big\},$$
it follows from (\ref{VFiltElementary}) that we have an isomorphism 
\begin{equation}\label{isom_delta1}
{\rm Gr}_V^{1-\lambda}(\cE)\simeq \big\{u\in (\cD_H[s]/s^m)^p\oplus (\cD_H[s]/sb(s))^q\mid (s+1-\lambda)^{m_{\lambda}}u=0\}\simeq 
\big(\cD_H[s]/(s+1-\lambda)^{m_{\lambda}}\big)^q.
\end{equation}
Similarly, using Lemma~\ref{lem_dualelementary},  we get
$${\rm Gr}^{\lambda}_V(\widetilde{\cE})\simeq \big\{u\in V^{>0}\widetilde{\cE}/V^{>1}\widetilde{\cE}=W^0\widetilde{\cE}/W^1\widetilde{\cE}\mid (s+\lambda)^{m_{\lambda}}u=0\big\}$$
hence (\ref{VFiltDualElementary}) gives an isomorphism
\begin{equation}\label{isom_delta2}
{\rm Gr}^{\lambda}_V(\widetilde{\cE})\simeq \big\{u\in (\cD_H[s]/(s+1)^m)^p\oplus (\cD_H[s]/(s+1)b(-s-1))^q\mid (s+\lambda)^{m_{\lambda}}=0\big\}
\end{equation} 
$$\simeq\big(\cD_H[s]/(s+\lambda)^{m_{\lambda}}\big).
$$
It follows from (\ref{isom_delta1}) that ${\rm Gr}_V^{1-\lambda}(\cE)$ is a free $\cD_H$-module, hence $\cH^i\big({\rm Gr}_V^{1-\lambda}(\cE)^*\big)=0$ for all $i\neq -n$ and
we have an isomorphism
$$\cH^{-n}\big({\rm Gr}_V^{1-\lambda}(\cE)^*\big)\simeq \cH^{-n}\left(\big(\cD_H[s]/(s+1-\lambda)^{m_{\lambda}}\big)^*\right)^q\simeq \big(\cD_H[s]/(s+1-\lambda)^{m_{\lambda}}\big)^q,$$
where the second isomorphism is provided by ${\rm nat}_{\lambda-1,m_{\lambda}}$. 
Using (\ref{isom_delta2}) and mapping $s$ to $-s-1$, we get an isomorphism
$$\delta_{\lambda}\colon {\rm Gr}^{\lambda}_V(\widetilde{\cE})\overset{\sim}\longrightarrow \cH^{-n}\big({\rm Gr}_V^{1-\lambda}(\cE)^*\big),$$
that satisfies $\delta_{\lambda}\circ (s+\lambda)=(-s-1+\lambda)\circ\delta_{\lambda}$.

We proceed similarly to relate ${\rm Gr}_V^{0}(\cE)$ and ${\rm Gr}_V^0(\widetilde{\cE})$. Recall that $m$ is the multiplicity of $0$ as a root of $b(s)$, hence
we have an isomorphism
$${\rm Gr}_V^0(\cE)\simeq \big\{u\in V^0\cE/V^1\cE\mid s^{m+1}u=0\big\}$$
and using (\ref{VFiltElementary}), we obtain an isomorphism
\begin{equation}\label{isom_delta3}
{\rm Gr}_V^0(\cE)\simeq \big\{u\in (\cD_H[s]/s^m)^p\oplus (\cD_H[s]/sb(s))^q\mid s^{m+1}u=0\big\}
\simeq \big(\cD_H[s]/s^m\big)^p\oplus \big(\cD_H[s]/s^{m+1}\big)^q.
\end{equation}
Similarly, we have an isomorphism
$${\rm Gr}_V^0(\widetilde{\cE})\simeq \big\{u\in V^{>-1}\widetilde{\cE}/V^{>0}\widetilde{\cE}=W^{-1}\widetilde{\cE}/W^0\widetilde{\cE}\mid s^{m+1}u=0\big\}$$
and using the isomorphism (\ref{VFiltDualElementary}), we obtain an isomorphism
\begin{equation}\label{isom_delta4}
{\rm Gr}_V^0(\widetilde{\cE})\simeq \big\{u\in (\cD_H[s]/s^m)^p\oplus (\cD_H[s]/sb(-s))^q\mid s^{m+1}u=0\big\}
\end{equation}
$$\simeq \big(\cD_H[s]/s^m\big)^p\oplus \big(\cD_H[s]/s^{m+1}\big)^q.
$$
It follows from (\ref{isom_delta3}) that ${\rm Gr}_V^{0}(\cE)$ is a free $\cD_H$-module, hence $\cH^i\big({\rm Gr}_V^{0}(\cE)\big)=0$ for all $i\neq -n$ and
we have an isomorphism
$$\cH^{-n}\big({\rm Gr}_V^{0}(\cE)^*\big)\simeq \cH^{-n}\left(\big((\cD_H[s]/s^m)^*\big)^p\oplus \big((\cD_H[s]/s^{m+1})^*\big)\right)^q
\simeq \big(\cD_H[s]/s^m\big)^p\oplus \big(\cD_H[s]/s^{m+1}\big)^q,$$
where the second isomorphism is provided by ${\rm nat}_{0,m}$ and ${\rm nat}_{0,m+1}$. Using (\ref{isom_delta4}) and mapping $s$ to $-s$, we get an isomorphism
$$\delta_0\colon {\rm Gr}^0_V(\widetilde{\cE})\overset{\sim}\longrightarrow \cH^{-n}\big({\rm Gr}_V^{0}(\cE)^*\big),$$
such that $\delta_0\circ s=(-s)\circ\delta_0$.

Finally, we relate ${\rm Gr}_V^{1}(\cE)$ and ${\rm Gr}_V^1(\widetilde{\cE})$. We have an isomorphism
$${\rm Gr}_V^1(\cE)\simeq \big\{u\in V^1\cE/V^2\cE\mid (s+1)^{m+1}u=0\}$$
and using (\ref{VFiltElementary}), we obtain an isomorphism
\begin{equation}\label{isom_delta5}
{\rm Gr}_V^1(\cE)\simeq \big\{u\in (\cD_H[s]/(s+1)^m)^p\oplus (\cD_H[s]/(s+1)b(s+1))^q\mid (s+1)^{m+1}u=0\big\}
\end{equation}
$$\simeq
\big(\cD_H[s]/(s+1)^m\big)^p\oplus \big(\cD_H[s]/(s+1)^{m+1}\big)^q.
$$
Similarly, we have an isomorphism 
$${\rm Gr}_V^1(\widetilde{\cE})\simeq \big\{u\in V^{>0}\widetilde{\cE}/V^{>1}\widetilde{\cE}=W^0\widetilde{\cE}/W^1\widetilde{\cE}\mid (s+1)^{m+1}u=0\big\},$$
hence using  the isomorphism (\ref{VFiltDualElementary}), we obtain an isomorphism
\begin{equation}\label{isom_delta6}
{\rm Gr}_V^1(\widetilde{\cE})\simeq \big\{u\in (\cD_H[s]/(s+1)^m)^p\oplus (\cD_H[s]/(s+1)b(-s-1))^q\mid (s+1)^{m+1}u=0\big\}
\end{equation}
$$\simeq \big(\cD_H[s]/(s+1)^m\big)^p\oplus\big(\cD_H[s]/(s+1)^{m+1}\big)^q.$$
It follows from the isomorphism (\ref{isom_delta5}) that ${\rm Gr}_V^{1}(\cE)$ is a free $\cD_H$-module, hence $\cH^i\big({\rm Gr}_V^{1}(\cE)\big)=0$ for all $i\neq -n$ and
we have an isomorphism
$$\cH^{-n}\big({\rm Gr}_V^{1}(\cE)^*\big)\simeq \cH^{-n}\left(\big((\cD_H[s]/(s+1)^m)^*\big)^p\oplus \big((\cD_H[s]/(s+1)^{m+1})^*\big)^q\right)$$
$$\simeq \big(\cD_H[s]/(s+1)^m\big)^p\oplus \big(\cD_H[s]/(s+1)^{m+1}\big)^q,$$
where the second isomorphism is provided by ${\rm nat}_{-1,m}$ and ${\rm nat}_{-1,m+1}$. 
Using (\ref{isom_delta6}) and mapping $s$ to $-s-2$, we get an isomorphism
$$\delta_1\colon {\rm Gr}^1_V(\widetilde{\cE})\overset{\sim}\longrightarrow \cH^{-n}\big({\rm Gr}_V^{1}(\cE)^*\big),$$
such that $\delta_1\circ s=(-s-2)\circ\delta_0$.

We can now address the compatibility of $V$-filtrations with duality:

\begin{proof}[Sketch of proof of Theorem~\ref{thm_duality}]
We have constructed the isomorphisms $\delta_{\alpha}$ in Theorem~\ref{thm_duality}, for $\alpha\in [0,1]$, when $\cM=\cE(b,P,Q)$
and we have a system of coordinates $x_1,\ldots,x_n,t$ on $X$. It is an important fact that these isomorphisms are functorial with
respect to morphisms between such modules, but we do not discuss this here and refer instead to \cite[pp. 356, 359, 361]{MaisonobeMebkhout}.
Using this, we can now extend the constructions of the isomorphisms $\delta_{\alpha}$ to an arbitrary coherent $\cD_X$-module $\cM$ on $X$
that admits a $V$-filtration,
under the further assumption that $X$ is affine. Indeed, by a repeated application of Lemma \ref{lem_elementaryresolution}, we have a resolution
by elementary $\cD_X$-modules  
$\cE^\bullet \to \cM \to 0$, such that $\cE^i=\cE(b_i,P_i,Q_i)$ for $i\leq 0$. It follows that via the left-right equivalence of categories,
$\cM^*$ corresponds to ${\mathbf R}\cH om_{\cD_X}(\cE^\bullet,\cD_X)[n+1]$. 

By Lemma~\ref{lem_dualelementary}, we know that for all $i\in \ZZ$, we have $\cE xt^j_{\cD_X}(\cE^i,\cD_X) = 0$ for $j \neq 1$ 
and the left $\cD_X$-module corresponding to $\cE xt^1_{\cD_X}(\cE^i,\cD_X)$ is $\widetilde{\cE}^i$. We deduce that
$\cM^*\simeq \widetilde{\cE}^\bullet[n]$ and using the exactness of ${\rm Gr}_V^{\alpha}(-)$, we conclude that
$${\rm Gr}_V^{\alpha}\big(\cH^j(\cM^*)\big)\simeq\cH^{j+n}\big({\rm Gr}_V^{\alpha}(\widetilde{\cE}^{\bullet})\big).$$

On the other hand, for every $\alpha$, we have seen that each ${\rm Gr}^{\alpha}_V(\cE^i)$ is a free $\cD_H$-module,
hence by exactness of ${\rm Gr}_V^{\alpha}(-)$, we get a resolution by free $\cD_H$-modules
\[ {\rm Gr}_V^\alpha(\cE^\bullet) \to {\rm Gr}_V^{\alpha} \cM \to 0.\]
It follows that via the left-right equivalence of categories, the dual
$\big({\rm Gr}_V^{\alpha} \cM\big)^*$ corresponds to $\cH om_{\cD_H}\big({\rm Gr}_V^\alpha(\cE^\bullet),\cD_H\big)[n]$, and thus
$$\cH^j\big({\rm Gr}^{\alpha}_V(\cM)^*\big)\simeq \cH^{j+n}\left(\cH^{-n}\big({\rm Gr}^{\alpha}_V(\cE^{\bullet})^*\big)\right).$$
The functoriality of the isomorphisms $\delta_{\lambda}$, $\delta_0$, and $\delta_1$ for the $\cE^i$ implies that we get 
corresponding isomorphisms for $\cM$. 

Given any two resolutions $\cE_1^\bullet \to \cM$ and $\cE_2^\bullet \to \cM$ by elementary $\cD_X$-modules, there is a third one
$\cE_3^\bullet \to \cM$ which dominates those two. Using again the functoriality of the isomorphisms $\delta_{\alpha}$ for elementary modules,
one can check that all three resolutions induce the same isomorphisms for $\cM$. This implies that for arbitrary $X$, we can construct the isomorphisms
$\delta_{\alpha}$ by gluing the corresponding isomorphisms on a suitable affine open cover. Furthermore, one can prove the functoriality of the $\delta_{\alpha}$
using again the special case of functoriality for morphisms between elementary $\cD_X$-modules. Finally, the assertions iv), v), and vi) in the theorem can be easily checked by checking them for elementary $\cD_X$-modules. This concludes our sketch of the proof of Theorem~\ref{thm_duality}.
\end{proof}

We have the following useful corollary of Theorem \ref{thm_duality}:

\begin{cor} 
Let $\cM$ be a holonomic $\cD_X$-module which admits a $V$-filtration along $H$. Then for all $\alpha \in \mathbf Q$, the $\cD_H$-module ${\rm Gr}_V^\alpha \cM$ is holonomic.
\end{cor}
\begin{proof} The assertion follows, using the theorem, from the fact that a coherent $\cD_X$-module $\cM$ is holonomic if and only if $\cH^j(\cM^*) = 0$ for all $j\neq 0$, see
 \cite[Cor. 2.6.8]{HTT}.
\end{proof}

\begin{rmk}[The case of regular holonomic $\cD_X$-modules]
In fact, if $\cM$ is a regular holonomic $\cD_X$-module, then ${\rm Gr}_V^\alpha \cM$ is also regular holonomic for all $\alpha \in \mathbf Q$, see
\cite[Cor. 4.7-5]{MaisonobeMebkhout}, as well as Corollary \ref{cor-regholo} below.
\end{rmk}

\begin{rmk}[Compatibility with duality for $V$-filtrations with respect to arbitrary hypersurfaces]
We have only stated Theorem~\ref{thm_duality} in the case of a smooth hypersurface. However, given any nonzero $f\in\cO_X(X)$, if $\iota\colon X
\hookrightarrow X\times {\mathbf A}^1$ is the corresponding graph embedding, then we have $\iota_+(\cM)^*\simeq \iota_+(\cM^*)$ by \cite[Section~2.7.2]{HTT}. 
It follows that by applying Theorem~\ref{thm_duality} for $\iota_+(\cM)$, we obtain a corresponding statement for the $V$-filtrations with respect to $f$. 
\end{rmk}

\section{Push-forward and $V$-filtrations}\label{section_behavior_push_forward2}

We turn to the behavior of $V$-filtrations under proper push-forward. Let $\pi \colon X\to Y$ be a morphism between smooth algebraic varieties and let $\cM$ be a coherent
$\cD_X$-module such that $\pi$ is proper on the support of $\cM$. Note that in this case
$\cH^j(\pi_+\cM)$ is a coherent $\cD_Y$-module for all $j$ (see \cite[Theorem~3.4.1]{Sabbah2}). 
Suppose that $f$ is a nonzero regular function on $Y$. We are interested in the relation between the $V$-filtration on $\cM$ with respect to $f\circ\pi$ and the $V$-filtration on
 $\cH^j(\pi_+\cM)$ with respect to $f$.

\subsection{The case of smooth projections}
We first consider the situation of a smooth projection $p\colon X = Y \times Z \to Y$, proper on the support of $\cM$, and of a nonzero $t\in\cO_Y(Y)$ that defines a smooth hypersurface. The general situation will be reduced to this case. 
With a slight abuse of notation, we also denote by $t$ the pull-back $t\circ p$. We put $d=\dim(Z)$.

Recall that by formula (\ref{pushforwardSmooth}), we have
$$\cH^j(p_+\cM)= R^jp_* \big(\Omega^\bullet_{X/ Y} \otimes_{\cO_X} \cM [d]\big).$$
Note that since $X=Y\times Z$, we have $V^0\cD_X=V^0\cD_Y\boxtimes\cD_Z$, hence for every $V^0\cD_X$-module $\cT$, we can form
the relative de Rham complex 
$\Omega^\bullet_{X/ Y} \otimes_{\cO_X}\cT$.
In particular,
if $\cM$ has a $V$-filtration $V^{\bullet}\cM$ with respect to $t$,
then for every $j\in\ZZ$ and every $\alpha\in\QQ$, we may consider
$$V^{\alpha}\cH^j(p_+\cM):={\rm Im}\left(R^jp_* \big(\Omega^\bullet_{X/ Y} \otimes_{\cO_X} V^{\alpha}\cM[d]\big)\to
R^jp_* \big(\Omega^\bullet_{X/ Y} \otimes_{\cO_X} \cM [d]\big)\right).$$
It is clear that this is a $V^0\cD_Y$-submodule of $\cH^j(p_+\cM)$.

\begin{thm}\label{thm:strict}
With the above notation, if $\cM$ is a coherent $\cD_X$-module such that $p$ is proper on $\mathrm{Supp\,}(\cM)$
and $\cM$ has a $V$-filtration $V^{\bullet}\cM$ with respect to $t$, then for every $j\in {\mathbf Z}$ and $\alpha\in {\mathbf Q}$, the canonical morphism 
 \begin{equation}\label{eq:v}
R^jp_* \big(\Omega^\bullet_{X/ Y} \otimes_{\cO_X} V^{\alpha}\cM[d]\big)\to
R^jp_* \big(\Omega^\bullet_{X/ Y} \otimes_{\cO_X} \cM [d]\big) 	
 \end{equation}
is injective. Moreover, $V^{\bullet}\cH^j(p_+\cM)$
is the $V$-filtration of $\cH^j(p_+\cM)$ with respect to $t$.
\end{thm}

\begin{proof}
We give the proof in $4$ steps, adapting an approach due to Sabbah, see \cite[Section III.4.8]{MebkhoutSabbah}.

\emph{Step 1}. We prove that for every locally finitely generated $V^0\cD_X$-module ${\mathcal T}$ such that $p$ is proper on ${\rm Supp}(\cT)$
and every $j\in {\mathbf Z}$, the $V^0\cD_Y$-module $R^jp_*\big(\Omega^\bullet_{X/ Y} \otimes_{\cO_X}\cT[d]\big)$ is locally finitely generated
(in particular this will be the case for $\cT=V^{\alpha}\cM$).
Since $\cT$ is locally finitely generated over $V^0\cD_X$, we can find a coherent $\cO_X$-submodule $\cF\subseteq \cT$ 
such that the induced map
\[
V^0\cD_X \otimes_{\cO_X}\cF \to \cT
\]
in surjective. Note that the kernel is again locally finitely generated over $V^0\cD_X$.
In order to show, by descending induction on $j$, that $R^jp_*\big(\Omega^\bullet_{X/ Y} \otimes_{\cO_X}\cT[d]\big)$ is locally finitely generated over $V^0\cD_Y$ for all
$\cT$ as above, it is thus enough, by a standard cohomological argument, to see that this is the case when $\cT=V^0\cD_X\otimes_{\cO_X}\cF$,
for a coherent $\cO_X$-module $\cF$ such that $p$ is proper on ${\rm Supp}(\cF)$. 
Since
$V^0\cD_X=V^0\cD_Y\boxtimes\cD_Z$ and the de Rham complex $\Omega_Z^{\bullet}\otimes_{\cO_Z}\cD_Z$ of $\cD_Z$ is quasi-isomorphic
to $\omega_Z[-d]$, it follows from the projection formula that
\begin{equation}\label{eq1_thm:strict}
R^jp_*\big(\Omega^\bullet_{X/ Y} \otimes_{\cO_X}V^0\cD_X\otimes_{\cO_X}\cF[d]\big)
\simeq V^0\cD_Y\otimes_{\cO_Y}R^{j}p_*\big(q^*(\omega_Z)\otimes_{\cO_X}\cF\big),
\end{equation}
where $q\colon X\to Z$ is the projection onto the second component.
Since $p$ is proper on the support of the coherent $\cO_X$-module $q^*(\omega_Z)\otimes_{\cO_X}\cF$, it follows that
$R^{j}p_*\big(q^*(\omega_Z)\otimes_{\cO_X}\cF\big)$ is a coherent $\cO_Y$-module, and thus 
(\ref{eq1_thm:strict}) is a locally finitely generated $V^0\cD_Y$-module. 
This completes the proof of Step 1. 

We note that $p_+(\cM)$ 
can be computed by the Godement resolution of the relative de Rham complex
\[
\cN^{\bullet}:=  p_* \mathrm{God}^\bullet \big(\Omega^\bullet_{X/ Y} \otimes_{\cO_X} \cM [d]\big).
\]
For every $\alpha\in\QQ$, we put 
$$V^{\alpha}\cN^{\bullet}:= p_* \mathrm{God}^\bullet \big(\Omega^\bullet_{X/ Y} \otimes_{\cO_X} V^{\alpha}\cM [d]\big)\subseteq\cN^{\bullet}.$$
It is clear that $V^{\bullet}\cN^j$ gives a decreasing, exhaustive, discrete and left-continuous filtration of $\cN^j$ for all $j\in \ZZ$.
Our goal is to show that $\cH^j(V^{\alpha}\cN^{\bullet})\to \cH^j(\cN^{\bullet})$ is injective for all $j$ and all $\alpha$ and the images of these maps,
when we vary $\alpha$, give a $V$-filtration on $\cH^j(\cN^{\bullet})$ with respect to $t$.

Note that the following $4$ conditions hold for $(\cN^\bullet,V^{\bullet})$:
\begin{enumerate}
	\item For every $\alpha \in \mathbf Q$ and $i\in \mathbf Z$, we have 
	$V^i\cD_Y\cdot V^\alpha\cN^\bullet \subseteq V^{\alpha+i}\cN^\bullet$, hence also
	$V^i\cD_Y\cdot \cH^j(V^\alpha\cN^\bullet) \subseteq \cH^j(V^{\alpha+i}\cN^\bullet)$ for all $j\in\ZZ$;
	\item  For every $\alpha\in \mathbf Q$, the operator $s+\alpha$ is nilpotent on $\gr_V^\alpha\cN^\bullet$; 
	\item For every $\alpha\in \mathbf{Q}_{>0}$, the left multiplication by $t\colon V^\alpha \cN^\bullet \to V^{\alpha+1}\cN^\bullet$ is bijective,
	hence so is the left multiplication by $t\colon \cH^j(V^\alpha \cN^\bullet)\to \cH^j(V^{\alpha+1}\cN^\bullet)$ for all $j\in\ZZ$;
	\item The $V^0\cD_Y$-module $\cH^j(V^{\alpha}\cN^\bullet)$ is locally finitely generated for all $\alpha\in \mathbf Q$ and $j\in\ZZ$.
\end{enumerate}
Indeed, conditions $(a)$, $(b)$ and $(c)$ follow by functoriality of the Godement resolution from the definition of the $V$-filtration on $\cM$ 
and the condition $(d)$ was proved in Step 1. If we show that for every $\alpha\in\QQ$ and $j\in\ZZ$, the canonical morphism
$$\cH^j(V^{\alpha}\cN^{\bullet})\to \cH^j(\cN^{\bullet})$$
is injective, then 
$$\cH^j(V^{\alpha}\cN^{\bullet})/\cH^j(V^{>\alpha}\cN^{\bullet})\simeq \cH^j(\gr_V^\alpha\cN^\bullet)$$
and $(s+\alpha)$ acts nilpotently on this by condition (b) above. Using also conditions
(a), (c), and (d) above, it is then clear that $\big(\cH^j(V^{\alpha}\cN^{\bullet})\big)_{\alpha\in\QQ}$ gives a $V$-filtration on $\cH^j(\cN^{\bullet})$
with respect to $t$.


\emph{Step 2}. We prove a formal analogue of the injectivity of the maps $\cH^j(V^{\beta}\cN^{\bullet})\to \cH^j(V^\alpha\cN^{\bullet})$ for $\alpha\geq\beta$.
For every $\alpha\in\QQ$, we put
\[
\widehat{V^\alpha \cN^\bullet}:= \varprojlim_\gamma V^\alpha \cN^\bullet/ V^\gamma \cN^\bullet,
\]
where the inverse limit is over $\{\gamma\in\QQ\mid\gamma\geq\alpha\}$, with the standard order
(which is a directed set). 
For every $\alpha\leq\beta\leq \gamma$, we have an exact sequence of complexes 
\begin{equation}\label{eq:sh}
	0\to V^\beta \cN^\bullet / V^\gamma \cN^\bullet \to V^\alpha\cN^\bullet/V^\gamma \cN^\bullet \to V^\alpha \cN^\bullet/V^\beta\cN^\bullet \to 0.
\end{equation}
Since the maps in the inverse system $(V^\beta \cN^\bullet/V^\gamma \cN^\bullet)_{\gamma}$  
are clearly surjective, the system satisfies the Mittag-Leffler condition and by taking the inverse limit,
we get an exact sequence
\[
0 \to \widehat{V^\beta \cN^\bullet} \to \widehat{V^\alpha \cN^\bullet} \to V^\alpha\cN^\bullet/ V^\beta\cN^\bullet \to 0.
\]
We conclude that for $\alpha \leq \beta$, the canonical morphism
\begin{equation}\label{eq:forminj}
	\widehat{V^\beta \cN^\bullet} \rightarrow \widehat{V^\alpha\cN^\bullet} 
\end{equation}
is injective and we have a canonical isomorphism
\begin{equation}\label{eq:formquo}
	V^\alpha\cN^\bullet/ V^\beta\cN^\bullet\simeq \widehat{V^\alpha \cN^\bullet}/\widehat{V^\beta \cN^\bullet}. 
\end{equation}

The exact sequence~\eqref{eq:sh} gives a long exact sequence
\[
 \cdots \to \cH^{j-1}(V^\alpha\cN^\bullet/ V^\beta\cN^\bullet) \xrightarrow{\delta} \cH^j(V^\beta \cN^\bullet / V^\gamma \cN^\bullet) \to \cH^j(V^\alpha\cN^\bullet/V^\gamma \cN^\bullet) \to \cdots.
\]
A key observation is that $\delta$ vanishes: note that $\delta$ is compatible with the $k[s]$-action and we have polynomials $p$ and $q$ with roots in $(-\beta,-\alpha]$ 
and $(-\gamma,-\beta]$, respectively, such that $p(s)$ acts nilpotently on the domain of $\delta$ and $q(s)$ acts nilpotently on the codomain of $\delta$.
Since $p$ and $q$ are relatively prime, it follows that 
$\delta = 0$. We conclude that for every $j\in\ZZ$, we have an exact sequence
\begin{equation}\label{eq:coml}
	0\to \cH^j(V^\beta \cN^\bullet / V^\gamma \cN^\bullet) \to \cH^j(V^\alpha\cN^\bullet/V^\gamma \cN^\bullet) \to \cH^j(V^\alpha\cN^\bullet/V^\beta \cN^\bullet) \to 0.
\end{equation}
In particular, the maps in the inverse system 
$\big(\cH^j(V^\alpha\cN^\bullet/V^\gamma \cN^\bullet)\big)_{\gamma}$ are surjective. We can thus apply Remark~\ref{rmk:commute} below to deduce that 
we have a canonical isomorphism 
\[
\cH^j({\widehat{V^\alpha\cN^\bullet}}) \simeq  \varprojlim_{\gamma} \cH^j (V^\alpha \cN^\bullet/ V^\gamma  \cN^\bullet).
\]
Moreover, since the inverse system $\big( \cH^j(V^\beta \cN^\bullet / V^\gamma \cN^\bullet)\big)_{\gamma}$ satisfies the Mittag-Leffler condition, 
 by taking the inverse limit of~\eqref{eq:coml}, we obtain the exact sequence
\begin{equation}\label{eq:formcinj}
	0\to \cH^i( \widehat{V^\beta\cN^\bullet}) \to \cH^i( \widehat{V^\alpha \cN^\bullet}) \to \cH^i(V^\alpha \cN^\bullet/V^\beta \cN^\bullet) \to 0.
\end{equation}


\emph{Step 3}. We now show
 that the $t$-torsion subsheaf $\cT^j_\alpha \subset \cH^j(V^\alpha \cN^\bullet)$ vanishes for $\alpha>0$. We prove this by descending induction on $j$. 
 The base case is trivial since $\cN^j=0$ for $j\gg 0$. Suppose now that $\cT^{j+1}_\alpha=0$ for all $\alpha>0$. The sequence 
\[
\cH^j(V^\alpha \cN^\bullet) \xrightarrow{t^k} \cH^j(V^\alpha \cN^\bullet) \to \cH^j(V^\alpha \cN^\bullet/ V^{\alpha+k}\cN^\bullet) \to 0,
\]
induced by the short exact sequence
\[
0\to V^\alpha \cN^\bullet \xrightarrow{t^k} V^\alpha \cN^\bullet \to V^\alpha \cN^\bullet/ V^{\alpha+k}\cN^\bullet \to 0
\] 
is exact, since the image of the connecting map 
\[
\delta\colon \cH^{j}(V^\alpha\cN^\bullet/ V^{\alpha+k} \cN^\bullet) \to \cH^{j+1}(V^\alpha\cN^\bullet),
\] 
is $t$-torsion and thus vanishes by the assumption that $\cT^{j+1}_\alpha=0$. We conclude that 
\begin{equation}\label{eq_100_quot}
\cH^j(V^\alpha\cN^\bullet/V^{\alpha+k}\cN^\bullet)\simeq \cH^j(V^\alpha\cN^\bullet)/t^k\cdot \cH^j(V^\alpha\cN^\bullet).
\end{equation}

We next show that for $k\gg 0$, the canonical map
\[
\cT^j_\alpha \to \cH^j(V^\alpha \cN^\bullet)/t^k\cdot \cH^j(V^\alpha \cN^\bullet)
\]
is injective, that is, we have  $\big(t^k\cdot \cH^j(V^\alpha \cN^\bullet)\big) \cap \cT^j_\alpha=0$. Indeed, since 
$\cH^j(V^\alpha\cN^\bullet)$ is a locally finitely generated $V^0 \cD_Y$-module, working locally we can find
a surjection $(V^0\cD_Y)^N\to \cH^j(V^\alpha \cN^\bullet)$ for some positive integer $N$.
For every $k\geq 0$, since $t^k\cdot V^0\cD_Y=V^k\cD_Y$, we deduce a surjection
 $(V^k\cD_Y)^N\to t^k\cdot \cH^j(V^\alpha  \cN^\bullet)$. 
 We obtain a short exact sequence of graded modules over ${\mathcal R}=\bigoplus_{k\geq 0}V^k\cD_X$:
 $$0\to \bigoplus_{k\geq 0}Q_k\to {\mathcal R}^N\to \bigoplus_{k\geq 0}t^k\cH^j(V^\alpha\cN^\bullet)\to 0.$$
 Since $\cR$ is generated over $\cR_0=V^0\cD_X$ by $t\in \cR_1$, it follows that $\cR(U)$ is both left and right Noetherian
 (see Remark~\ref{V0_Noeth}); 
 therefore $\bigoplus_{k\geq 0}Q_k$ is finitely generated over $\cR$ and, moreover, there is $k_0$ such that $Q_{k+1}=t\cdot Q_k$ for all $k\geq k_0$. 
 Since multiplication by $t$ gives an isomorphism $V^k\cD_X\to V^{k+1}\cD_X$, we conclude by the Snake Lemma that multiplication by $t$
 gives an isomorphism $t^k\cH^i(V^\alpha\cN^\bullet)\to t^{k+1}\cH^i(V^\alpha\cN^\bullet)$ for all $k\geq k_0$. Therefore we have
 $\big(t^k\cdot \cH^j(V^\alpha \cN^\bullet)\big) \cap \cT^j_\alpha=0$ for $k\geq k_0$, giving our assertion.

The composition
$$\cT^j_{\alpha}\to \cH^j({V^\alpha \cN^\bullet})\to \cH^j(\widehat{V^\alpha \cN^\bullet})$$
is injective as well for $\alpha>0$ since (\ref{eq:coml}) and (\ref{eq_100_quot}) give canonical isomorphisms
\[
\cH^j(\widehat{V^\alpha\cN^\bullet})/\cH^j(\widehat{V^{\alpha+k}\cN^\bullet})\simeq\cH^j(V^\alpha\cN^\bullet/V^{\alpha+k}\cN^\bullet)\simeq 
\cH^j(V^\alpha\cN^\bullet)/t^k\cdot \cH^j(V^\alpha\cN^\bullet).
\]
On the other hand, since $\alpha>0$, multiplication by $t$ on $\cH^j(\widehat{V^\alpha\cN^\bullet)}$ factors as
$$\cH^j(\widehat{V^\alpha\cN^\bullet)}\to \cH^j(\widehat{V^{\alpha+1}\cN^{\bullet}})\hookrightarrow \cH^j(\widehat{V^{\alpha}\cN^{\bullet}}),$$
where the first map is an isomorphism and the second one is injective by (\ref{eq:formcinj}). Therefore multiplication by $t$ is injective on 
$\cH^j(\widehat{V^\alpha \cN^\bullet})$ and we conclude that $\cT^j_{\alpha}=0$.

\emph{Step 4}. Given $\alpha\leq \gamma$, if $\ell\gg 0$, then $\alpha+\ell\geq\gamma$, so that $t^{\ell}\cdot V^{\alpha}\cN^{\bullet}\subseteq
V^{\gamma}\cN^{\bullet}$. This implies that the kernel of the map
\[
\delta^j\colon \cH^j(V^\gamma\cN^\bullet) \to \cH^j(V^\alpha\cN^\bullet),
\] 
is killed by $t^{\ell}$, hence it is contained in $\cT^j_{\gamma}$. By \emph{Step 3}, it follows that if $\gamma>0$, then all $\delta^j$ are injective,
hence the exact sequence
\[
0 \to V^\gamma\cN^\bullet \to V^\alpha\cN^\bullet \to V^\alpha\cN^\bullet/ V^\gamma \cN^\bullet \to 0
\]
induces exact sequences 
\begin{equation}\label{eq_101_exact_seq}
0\to \cH^j(V^\gamma\cN^\bullet) \to \cH^j(V^\alpha \cN^\bullet) \to \cH^j(V^\alpha\cN^\bullet/V^\gamma \cN^\bullet) \to 0.
\end{equation}
For every $\alpha\geq\beta$, by taking $\gamma>\max\{\alpha,0\}$, and combining (\ref{eq_101_exact_seq}) and 
(\ref{eq:coml}), we conclude that the exact sequence 
\[
0\to \cH^j(V^\alpha\cN^\bullet) \to \cH^j(V^\beta \cN^\bullet) \to \cH^j(V^\beta\cN^\bullet/V^\alpha \cN^\bullet) \to 0
\]
is exact as well. This implies that 
\[
\cH^j(V^\alpha\cN^\bullet) \to \varinjlim_{\beta}\cH^j(V^\beta \cN^\bullet)=\cH^j(\cN^\bullet)
\]
is injective for every $\alpha\in\QQ$ and every $j\in\ZZ$. As we have seen, this also implies the last assertion in the theorem.
\end{proof}

\begin{rmk}[A criterion for commuting cohomology and inverse limits]\label{rmk:commute}
	Let $(A_i^\bullet)_{i\in I}$ be an inverse system of complexes, where $(I,\leq)$ is a directed set. Suppose that for every $j\leq k$ in $I$ and every $q\in\ZZ$, the canonical maps
	$A_k^{\bullet}\to A_j^{\bullet}$ and $\cH^q(A_k^{\bullet})\to \cH^q(A_j^{\bullet})$ are surjective. 
	Let $B_i^m\subseteq A_i^m$ and $Z^m_i\subseteq A_i^m$ be the coboundaries and, respectively, the cocycles.
	Note that for every $j\leq k$ in $I$, the induced map $B_k^m\to B_j^m$ is clearly surjective. Moreover,  we see that $Z^m_k\to Z^m_j$
	is surjective using the Snake Lemma and the exact sequences
	\begin{equation}\label{eq1_rmk:commute}
	0\to B_{i}^m\to Z_{i}^m\to \cH^m(A^{\bullet}_{i})\to 0
	\end{equation}
	for all $i\in I$. In particular, for every $m\in\ZZ$, the inverse systems $(B^m_i)_{i\in I}$ and $(Z^m_i)_{i\in I}$ satisfy the  Mittag-Leffler condition.
	Therefore the exact sequences (\ref{eq1_rmk:commute}) and
	$$0\to Z^m_i\to A^m_i\to B^{m+1}_i\to 0$$
	induce after taking inverse limits the exact sequences
	$$0\to \varprojlim B_i^m\to \varprojlim_iZ^m_i\to \varprojlim_i\cH^m(A_i^{\bullet})\to 0$$
	and
	$$0\to \varprojlim_iZ^m_i\to \varprojlim_iA^m_i\to\varprojlim_iB^{m+1}_i\to 0.$$
	We conclude from here that the canonical morphism 
	\[
	\cH^m(\varprojlim_{i} A_i^\bullet) \to \varprojlim_{i} \cH^m (A_i^{\bullet}).
	\]
is an isomorphism for all $m\in\ZZ$. 
	\end{rmk}

\subsection{The general case}
Suppose now that $\pi\colon X\to Y$ is any morphism between smooth, irreducible algebraic varieties
and $\cM$ is a coherent $\cD_X$-module such that $\pi$ is proper on 
the support of $\cM$. Let $f\in\cO_Y(Y)$ be such that $g=f\circ \pi$ is nonzero.
We assume that $\cM$ has a $V$-filtration with respect to $g$ and we want to show that in this case
each $\cH^j(\pi_+\cM)$ has a $V$-filtration with respect to $f$ and to describe this $V$-filtration.

Let $\iota_f \colon Y\to Y\times \mathbf A^1$ and $\iota_g \colon X\to X\times \mathbf A^1$ be the corresponding graph embeddings, so we have the 
Cartesian diagram
\[
\begin{tikzcd} 
	X \arrow{r}{\iota_g} \arrow{d}{\pi} & X\times \mathbf A^1 \arrow{d}{\pi'} \\
	Y \arrow{r}{\iota_f} 	& Y\times \mathbf A^1,
\end{tikzcd}
\]
where $\pi'=\pi\times {\rm Id}$.
Note that $\pi'_+({\iota_g}_+\cM)\simeq {\iota_f}_+(\pi_+\cM)$ and thus
$${\iota_f}_+\cH^j(\pi_+\cM)\simeq \cH^j(\pi'_+\cN\big),$$
where $\cN={\iota_g}_+\cM$. Let $t$ denote the coordinate on ${\mathbf A}^1$.
We thus assume that $\cN$ has a $V$-filtration with respect to $t$ and we need to show that each $\cH^j(\pi'_+\cN)$ has a $V$-filtration
with respect to $t$ and describe this $V$-filtration. As usual, we can factor $\pi'$ as
$$X\times{\mathbf A}^1\overset{i}\hookrightarrow X\times Y\times {\mathbf A}^1\overset{p}\longrightarrow Y\times {\mathbf A}^1,$$
where $i$ is the closed immersion given by $i(x,t)=\big(x,\pi(x),t\big)$ and $p(x,y,t)=(y,t)$. 
In this case, we can compute the $V$-filtration on $i_+\cN$ using Proposition~\ref{prop_Vfil_immersion} and we then compute the
filtration on $\cH^j\big(p_+(i_+\cN)\big)\simeq \cH^j(\pi'_+\cN)$ using Theorem~\ref{thm:strict} for the smooth projection $p$. 
Using this approach, we obtain the following

\begin{thm}
Let $\pi\colon X\to Y$ be a  morphism between smooth, irreducible algebraic varieties
and let $\cM$ be a coherent $\cD_X$-module such that $\pi$ is proper on 
the support of $\cM$. If $f\in\cO_Y(Y)$ is such that $g=f\circ \pi$ is nonzero and $\cM$ has a 
$V$-filtration with respect to $g$, then $\cH^j(\pi_+\cM)$ has a $V$-filtration with respect to $f$
for every $j\in\ZZ$ and we have isomorphisms
$${\rm Gr}^{\alpha}_V\big(\cH^j(\pi_+\cM)\big)\simeq \cH^j\big(\pi_+{\rm Gr}^{\alpha}_V(\cM)\big)\quad\text{for all}\quad \alpha\in\QQ.$$
\end{thm}

\begin{proof}
Using the previously described approach, we see that it is enough to prove the assertion separately in the following two cases:
\begin{enumerate}
\item[i)] When $\pi$ is a closed immersion and both $f$ and $g$ define smooth hypersurfaces. In this case the assertion follows from
the description of the $V$-filtration on $\pi_+(\cM)$ in Proposition~\ref{prop_Vfil_immersion}.
\item[ii)] When $\pi\colon Y\times Z\to Y$ is the projection onto the first component. In this case, the assertion follows from Theorem~\ref{thm:strict}.
\end{enumerate}
\end{proof}

\section{The comparison with the topological vanishing and nearby cycles}\label{section_comparison}

In this section we work over the ground field $k = \C$. For every complex algebraic variety $Z$, we denote by $Z^{\rm an}$ the corresponding analytic space.
Let $X$ be a smooth, irreducible complex algebraic variety. Recall that in this setting we have the Riemann-Hilbert
correspondence provided by the analytic de Rham functor ${\rm DR}_X^{\rm an}(-)$. 

\subsection{Topological vanishing and nearby cycles}
For an introduction to the topological point of view on vanishing and nearby cycles, we refer to
\cite[Chapter~4.2]{Dimca} and \cite[Chapter~10.3]{Maxim} (see also \cite[Section 1]{Bry} and \cite[Section 2]{MaisonobeMebkhout}). In what follows we simply review the basic definitions.

Let $f\colon X \to \mathbf {\mathbf A}^1$ be a regular function, which we view as a holomorphic map between complex manifolds $f\colon X^{\rm an} \to \mathbf C$. 
Let $H = f^{-1}(0)$ and $U=X\smallsetminus H$, with corresponding embeddings $i\colon H^{\rm an} \hookrightarrow X^{\rm an}$
and $j\colon U^{\rm an}\hookrightarrow X^{\rm an}$.
Let $\C^* =\C\smallsetminus \{0\}$ be the punctured plane, so $U^{\rm an} = f^{-1}(\C^*)$. We have the following commutative diagram, with Cartesian squares
\[ \begin{tikzcd} H^{\rm an} \ar[r,"i"] \ar[d] & X^{\rm an} \ar[d,"f"] & U^{\rm an} \ar[swap,l,"j"] \ar[d] & \widetilde{U}\ar[swap,l,"\chi"] \ar[d] \\ \{0\} \ar[r] & \C & \C^* \ar[l] & \mathbf C \ar[l,"\exp"], \end{tikzcd}\]
where $\exp\colon \mathbf C \to \C^*$ is the universal cover of the punctured plane.

Following Grothendieck and Deligne \cite{Deligne}, we define for $\cK \in \cD^b_c(X^{\rm an})$ the \emph{nearby cycles} complex $\psi_f \cK$ by
\[ \psi_f \cK : = i^{-1} \mathbf R(j \circ \chi)_* \chi^{-1} j^{-1} \cK\in \cD^b_c(H^{\rm an}).\]
The map $\mathbf C \to \mathbf C$, sending $z$ to $z+1$, induces a \emph{monodromy operator} $T$ on $\psi_f\cK$. Moreover, adjunction for direct images gives a natural morphism $i^{-1} \cK \to \psi_f \cK$ in $\cD^b_c(H^{\rm an})$. By definition, the \emph{vanishing cycles} $\varphi_f\cK$ of $\cK$ is the cone of this morphism, hence we have an exact triangle
\begin{equation}\label{trg_nearby}
i^{-1} \cK \to \psi_f \cK \xrightarrow[]{{\rm can}} \varphi_f\cK \xrightarrow[]{+1}.
\end{equation}

If $\cK$ is a perverse sheaf on $X$, then by a result of Gabber \cite[p. 14]{Bry}, both $\psi_f \cK [-1]$ and $\varphi_f \cK [-1]$ are perverse sheaves on $H$. We denote the shifted functor with a superscript $p$, 
that is, we put
\[ {}^p \psi_f : = \psi_f [-1] \quad\text{and}\quad {}^p \varphi_f: = \varphi_f [-1].\]

If we let $T$ act on $i^{-1}\cK$ by the identity, then the first morphism in the triangle (\ref{trg_nearby}) commutes with $T$, hence we get a monodromy operator $T$ on $\varphi_f \cK$, too. As the category of $\C$-perverse sheaves is abelian and $\C$-linear, we can decompose these objects into generalized eigenspaces for the operator $T$, so we have
\[ {}^p \psi_f \cK = \bigoplus_{\alpha \in \mathbf C^*} {}^p \psi_{f,\alpha} \cK,  \quad {}^p \varphi_f \cK = \bigoplus_{\alpha \in \mathbf C^*} {}^p \varphi_{f,\alpha} \cK.\]

By triviality of the monodromy on $i^{-1} \cK[-1]$, we see that the morphism ${\rm can}$ gives an isomorphism ${}^p \psi_{f,\alpha}\cK \to {}^p \varphi_{f,\alpha}\cK$ for all $\alpha \neq 1$.

\begin{rmk}[The map ${\rm Var}$] It is also possible to define a morphism ${\rm Var}\colon \varphi_{f,1} \cK \to \psi_{f,1}\cK$ so that $T-1 = {\rm can} \circ {\rm Var}\colon \varphi_{f,1}\cK \to \varphi_{f,1}\cK$ and $T-1 = {\rm Var} \circ {\rm can}\colon \psi_{f,1} \cK \to \psi_{f,1}\cK$.
\end{rmk}

\subsection{The vanishing and nearby cycles via the $V$-filtration}
The following comparison theorem relates the vanishing and nearby cycles to the $V$-filtration on regular holonomic $\cD_X$-modules.

\begin{thm} \label{thm-comparison} Let $X$ be a smooth, irreducible complex algebraic variety and $H\subseteq X$ a smooth hypersurface defined by $t \in \cO_X(X)$. If $\cM$ is a regular holonomic $\cD_X$-module that has a $V$-filtration $V^{\bullet}\cM$ with respect to $t$ and 
$\cK = {\rm DR}^{\rm an}_X(\cM)$ is the corresponding perverse sheaf, then there are natural isomorphisms
\[ {\rm DR}^{\rm an}_X\big({\rm Gr}_V^{\lambda}(\cM)\big) \cong {}^p \psi_{t,\alpha} \cK, \text{ where } \alpha = \exp(-2\pi i \lambda),\,\,\text{and}\]
\[ {\rm DR}^{\rm an}_X\big({\rm Gr}_V^0 (\cM)\big) \cong {}^p \varphi_{t,1}\cK.\]
\end{thm}

\begin{rmk}[Compatibility with ${\rm can}$ and ${\rm Var}$]\label{remark_nearby_cycles} 
The above isomorphisms satisfy the additional property that the morphisms\[ {\rm can}\colon {}^p \psi_{t,1} \cK \to {}^p \varphi_{t,1} \cK\]
\[ {\rm Var}\colon {}^p \varphi_{t,1}\cK \to {}^p \psi_{t,1}\cK\]
correspond under the functor ${\rm DR}_H^{\rm an}(-)$ to the morphisms
\[ -\de_t\colon {\rm Gr}_V^1(\cM) \to {\rm Gr}_V^0(\cM)\]
\[ \tfrac{\exp\big(2\pi i(s+1)\big) -1}{s+1} t\colon {\rm Gr}_V^0(\cM) \to {\rm Gr}_V^1( \cM).\]

In particular, the factorizations $T-1 = {\rm can}\circ {\rm Var}$ and $T-1 = {\rm Var} \circ {\rm can}$ show that $T$ corresponds to $\exp(2\pi i s)$ on ${\rm Gr}_V^0(\cM)$ 
and to $\exp\big(2\pi i (s+1)\big)$ on ${\rm Gr}_V^1(\cM)$. Similarly, for $\lambda\in (0,1)$ and $\alpha={\rm exp}(- 2\pi i\lambda)$, the action of $T$ on ${}^p \psi_{t,\alpha} \cK$ corresponds to 
${\rm exp}(2\pi i s)$ on ${\rm Gr}_V^{\lambda}(\cM)$. 
\end{rmk}

\begin{rmk}[The case of arbitrary hypersurfaces]
Suppose now that $X$ is a smooth, irreducible complex algebraic variety and $f\in\cO_X(X)$ is an arbitrary nonzero regular function defining the hypersurface
$H\overset{i}\hookrightarrow X$. 
If $\iota\colon X\hookrightarrow X\times {\mathbf A}^1$ 
is the graph embedding corresponding to $f$ and $t$ denotes the coordinate function on ${\mathbf A}^1$,
 then it is easy to see that for every $\cK \in \cD^b_c(X^{\rm an})$, we have
$$i_*(\psi_f\cK)\simeq \psi_t\big(\iota_*(\cK)\big)\quad\text{and}\quad i_*(\varphi_f\cK)\simeq \varphi_t\big(\iota_*(\cK)\big).$$
Suppose now that $\cK={\rm DR}_X(\cM)$, for a regular holonomic $\cD_X$-module $\cM$ that admits a $V$-filtration with respect to $f$.
Since the functor ${\rm DR}_X^{\rm an}(-)$ commutes with push-forward, we obtain from Theorem~\ref{thm-comparison} a similar description of
${}^p \psi_{f,\alpha}\cK$ and ${}^p \varphi_{f,1}\cK$ in terms of the graded pieces of the $V$-filtration on $\iota_+(\cM)$. 
\end{rmk}

We give a proof of Theorem~\ref{thm-comparison}, following \cite{MaisonobeMebkhout} (see also Wu's survey article \cite{LeiRH}). We refer also to \cite[Section 3.4]{Saito-MHP} for a proof of the comparison theorem for filtered $\cD_X$-modules underlying mixed Hodge modules. 

If $\cM$ is a holonomic $\cD_X$-module, then we put $\cM\vert_H: = \big(\mathbf Li^*(\cM^*)[-1]\big)^* \in \cD^b_{h}(\cD_H)$.
Note that if $\cM$ has a $V$-filtration $V^{\bullet}\cM$ with respect to $t$, 
then Theorem~\ref{thm_duality} implies that $\cM^*$, too, has such a $V$-filtration. By Proposition~\ref{prop_descr_Var}, we see that
${\mathbf L}i^*(\cM^*)[-1]$ is represented by the complex 
\[{\rm Gr}_V^0(\cM^*) \overset{t\cdot}\longrightarrow {\rm Gr}_V^1(\cM^*),\]
placed in degrees $0$ and $1$, and we conclude using Theorem~\ref{thm_duality} that $\cM\vert_H$ 
 is represented by the complex
\[{\rm Gr}_V^1(\cM) \overset{\partial_t\cdot}\longrightarrow {\rm Gr}_V^0(\cM),\]
placed in degrees $-1$ and $0$. 

The general outline of the proof of Theorem~\ref{thm-comparison} is the following: we wish to replace ${\rm Gr}_V^\alpha(\cM)$ with some objects whose behavior under the Riemann-Hilbert correspondence is better understood. This is done by the \emph{moderate nearby cycles} construction (in Lemmas \ref{lem-compare} and  \ref{lem-compare0}), which essentially gives an isomorphism between ${\rm Gr}_V^\alpha(\cM)$ and the restriction of a certain $\cD_X$-module. In this step we also describe how $s$ acts under this isomorphism.

This restriction is known, by \cite{HTT}*{Equation 7.1.5}, to commute with the de Rham functor, yielding the usual restriction functor on the topological side. So the problem becomes understanding the moderate nearby cycles in a way which relates to the nearby/vanishing cycles of the de Rham complex of the module $\cM$. This is done using various theorems from \cite{MaisonobeMebkhout}, which we state below but do not prove. One technical point (seen in Theorem \ref{thm-MM213}) is that it is better at that point to work with the \emph{solution complex} (defined below), which is dual to the de Rham complex. 

The key fact underlying the theorems we use from \cite{MaisonobeMebkhout} is that, for regular holonomic $\cD$-modules, one can replace the solution complexes we are interested in with those of finite determination (see Theorem \ref{thm-MM314}). From there, after decomposing into the eigenspaces for the monodromy operator, one can even replace those solution complexes with those of \emph{moderate growth} (as in Theorem \ref{thm-MM331}). These moderate growth solution complexes are precisely the de Rham complexes of the moderate nearby cycle construction, which completes the argument.

\bigskip

We begin by introducing some $\cD$-modules that play an important role in the proof.
On $\mathbf A^1 \smallsetminus \{0\} = \mathbf G_m$, for $\alpha \in (0,1]\cap\QQ$ and $k\in \mathbf Z_{>0}$, we consider the regular holonomic $\cD$-module $\cN^{\alpha,k} : = \cD_{\mathbf G_m}/ \cD_{\mathbf G_m}(z \de_z + \alpha -1 )^{k+1}$, where we denote by $z$ the coordinate on ${\mathbf A}^1$. 
This module is free over $\cO_{{\mathbf A}^1}[1/z]$ with generators $e_{\alpha,0},\dots, e_{\alpha,k}$,
where 
\[ e_{\alpha,\ell} = (z\de_z + \alpha-1)^{k-\ell}\cdot \overline{1}\quad\text{for}\quad 0\leq\ell\leq k,\]
so that $(z\de_z + \alpha-1)^{\ell+1}e_{\alpha,\ell}=0$. 
If $j\colon \mathbf G_m \to \mathbf A^1$ is the open embedding, then the $\cD_{{\mathbf A}^1}$-module
$\cN_{\alpha,k} : = j_+ \cN^{\alpha,k}$ is again regular holonomic and has the same $\cO_{{\mathbf A}^1}[1/z]$-basis given by $e_{\alpha,0},\dots, e_{\alpha,k}$.

\begin{rmk}[The monodromy operator on $\cN_{\alpha,k}$]\label{rmk-moderateMonodromy} 
One can embed $\cN^{\alpha,k}$ in the $\cD$-module consisting of the multivalued complex functions on ${\mathbf C}^*$, such that 
$e_{\alpha,\ell}$  maps to $\frac{1}{\ell!} z^{1-\alpha} \log(z)^\ell$ for $0\leq\ell\leq k$. Note that the monodromy around a counter-clockwise loop around $0$ in $\C$ maps $\log(z)$ to $\log(z) + 2\pi i$. As $z^{1-\alpha} = \exp\big((1-\alpha) \log(z)\big)$, the monodromy acts on this multivalued function by multiplication with $\exp(-2\pi i \alpha) = : \xi_\alpha$.
Moreover, it acts on $\log(z)^\ell$ by the formula
\[ \log(z)^\ell \mapsto \big(\log(z) + 2\pi i\big)^\ell = \sum_{j=0}^\ell \binom{\ell}{j} (2\pi i)^j \log(z)^{\ell - j}.\]
Thus we can define a monodromy operator $T\colon \cN_{\alpha,k} \to \cN_{\alpha,k}$ by the formula
\[ e_{\alpha,\ell} \mapsto \xi_\alpha \cdot \sum_{j=0}^\ell \tfrac{(2\pi i)^j}{j!} e_{\alpha,\ell-j}\quad\text{for}\quad 0\leq\ell\leq k.\]
Note that this is a $\cD_{\mathbf A^1}$-linear map.
\end{rmk}

We consider the $\cD_X$-module ${\mathbf L}t^*(\cN_{\alpha,k}) = t^*(\cN_{\alpha,k})$, where we view $t\colon X \to \mathbf A^1$ as a (flat) morphism of algebraic varieties.
For a  regular holonomic $\cD_X$-module $\cM$, we put
\[ \cM_{\alpha,k} : = \cM \otimes_{\cO_X} t^*(\cN_{\alpha,k}) = \bigoplus_{0\leq \ell \leq k} \cM[1/t] e_{\alpha,\ell},\]
on which derivations act via the Leibniz rule (see \cite[Proposition~1.2.9]{HTT}). In particular, we have
\begin{equation}\label{eq_action_on_e}
(t\partial_t-1+\alpha)\cdot me_{\alpha,\ell}=(t\partial_t\cdot m)e_{\alpha,\ell}+me_{\alpha,\ell-1}\quad\text{for}\quad0\leq\ell\leq k
\end{equation}
(with the convention that $e_{\alpha,-1}=0$).
We have an induced automorphism $T\colon \cM_{\alpha,k} \to \cM_{\alpha,k}$ given by
$T(m\otimes u) = m \otimes T(u)$.
This is $\cD_X$-linear and thus induces automorphisms $\cH^i(\cM_{\alpha,k}\vert_H)\to \cH^i(\cM_{\alpha,k}\vert_H)$, for $i=-1$, $0$, that we also denote by $T$. 

\begin{rmk}[Compatibity when we vary $k$]
Note that right multiplication by $(z\partial_z+\alpha-1)$ induces, for every $\alpha$ and $k$, a morphism of $\cD_{{\mathbf A}^1}$-modules
$\cN_{\alpha,k}\hookrightarrow \cN_{\alpha,k+1}$ that maps $e_{\alpha,\ell}$ to $e_{\alpha,\ell}$ for all $0\leq\ell\leq k$. For every $\cM$ as above,
we thus get a corresponding morphism of $\cD_X$-modules $\cM_{\alpha,k}\hookrightarrow \cM_{\alpha,k+1}$.
\end{rmk}

We note that if $\cM$ has a $V$-filtration with respect to $t$, since $\coker\big(\cM\to \cM[1/t]\big)$ also has one by Example~\ref{eg_case_support_in_H},
it follows that $\cM[1/t]$ has such a $V$-filtration by Proposition~\ref{properties_V_filtration1} and Corollary~\ref{cor_thm_existence_Vfilt}.

\begin{lem} \label{lem_V_filt_M_alpha}
With the above notation, if $\cM$ has a $V$-filtration with respect to $t$, then so does $\cM_{\alpha,k}$, and
\begin{equation}\label{eq_lem_V_M_alpha}
V^\beta \cM_{\alpha,k} = \bigoplus_{0\leq \ell \leq k} V^{\beta+ \alpha-1}\big(\cM[1/t]\big) e_{\alpha,\ell}\quad\text{for all}\quad\beta\in\QQ.
\end{equation}
\end{lem}

\begin{proof} 
Using the formula (\ref{eq_action_on_e}) and the properties of the $V$-filtration on $\cM[1/t]$, it is easy to check 
 that the filtration defined by the right-hand side of equation (\ref{eq_lem_V_M_alpha})
 satisfies the properties of the 
  $V$-filtration on $\cM_{\alpha,k}$.
\end{proof}

Since $t$ acts bijectively on $\cM_{\alpha,k}$, we know that ${\rm Gr}_V^0(\cM_{\alpha,k}) \overset{t\cdot}\longrightarrow {\rm Gr}_V^1(\cM_{\alpha,k})$ is an isomorphism by Corollary \ref{cor_prop_descr_Var}. Thus, we can identify \[\cH^{-1} (\cM_{\alpha,k}\vert_H) =
 \ker\big({\rm Gr}_V^1(\cM_{\alpha,k}) \overset{t\partial_t\cdot}\longrightarrow {\rm Gr}_V^1(\cM_{\alpha,k})\big).\]

\begin{lem} \label{lem-compare} 
For every $\alpha \in (0,1]$ and every $k$ larger than the index of nilpotency of $(s+\alpha)$ on ${\rm Gr}_V^{\alpha}(\cM)$, 
we have a functorial isomorphism of $\cD_H$-modules
\begin{equation}\label{eq_Gr_1}
\varphi\colon {\rm Gr}_V^\alpha(\cM) \overset{\sim}\longrightarrow \cH^{-1} (\cM_{\alpha,k}\vert_H).
\end{equation}
Moreover, we have $T \circ \varphi = \xi_{\alpha}\varphi \circ \exp\big(2\pi i (s+\alpha)\big)$, where $\xi_{\alpha}={\rm exp}(-2\pi i\alpha)$, and
we have $\varinjlim_k \cH^0 (\cM_{\alpha,k}\vert_H) = 0$.
\end{lem}

\begin{proof} 
Since $V^\alpha \cM \to V^{\alpha}\big(\cM[1/t]\big)$ is an isomorphism for $\alpha > 0$ (see Corollary \ref{cor_equiv_exist_Vfil}), 
it follows from Lemma~\ref{lem_V_filt_M_alpha} that
\[ V^1 \cM_{\alpha,k} = \sum_{0\leq \ell \leq k} V^\alpha\cM e_{\alpha,\ell}.\]
Let $m_0,\dots, m_k \in V^\alpha\cM$.
Using (\ref{eq_action_on_e}), we obtain
\[ t\de_t\cdot\sum_{\ell=0}^km_\ell e_{\alpha,\ell} = \big((t\de_t -\alpha + 1)m_k\big)e_{\alpha,k} + \sum_{\ell=0}^{k-1}\big(m_{\ell +1} + (t\de_t - \alpha + 1)m_{\ell}\big)e_{\alpha,\ell}.\]
Since $t\partial_t-\alpha+1=-(s+\alpha)$, we conclude that
$\sum_{\ell=0}^k m_\ell e_{\alpha,\ell} \in \ker(t\de_t)$
if and only if $m_{\ell}=(s+\alpha)^{\ell}m_0$ for $0\leq\ell\leq k$ and $(s+\alpha)^{k+1}m_0=0$. Note that by our assumption on $k$, the latter condition is automatically satisfied.
We thus have an isomorphism
\[ \varphi\colon {\rm Gr}_V^\alpha (\cM)\overset{\sim}\longrightarrow \cH^{-1}(\cM_{\alpha,k}\vert_H),\quad
\varphi(m_0)= \sum_{\ell =0}^k \big((s+\alpha)^\ell m_0\big) e_{\alpha,\ell}.\]
In particular, we see that an element of $\cH^{-1}(\cM_{\alpha,k}\vert_H)$ is determined by the coefficient of $e_{\alpha,0}$.

Now, we wish to describe
$T\colon \cH^{-1}(\cM_{\alpha,k}\vert_H) \to \cH^{-1} (\cM_{\alpha,k}\vert_H)$ via the above isomorphism. 
Let $m_0 \in {\rm Gr}_V^\alpha(\cM)$, and consider 
\[ \varphi(m_0) = \sum_{\ell =0}^k(s+\alpha)^\ell m_0 e_{\alpha,\ell}\in\cH^{-1}(\cM_{\alpha,k}\vert_H).\]
By definition of $T$, we have
\[ T\big(\varphi(m_0)\big) = \sum_{\ell = 0}^k (s+\alpha)^\ell m_0 T(e_{\alpha,\ell}),\]
and, using the formula in Remark \ref{rmk-moderateMonodromy}, the coefficient of $e_{\alpha,0}$ in this expression is
\[\xi_\alpha \cdot\sum_{\ell = 0}^k \tfrac{(s+\alpha)^\ell(2\pi i)^\ell}{\ell!} m_0 = \xi_\alpha \exp\big(2\pi i(s+\alpha)\big)\cdot m_0.\]

In order to prove the vanishing $\varinjlim\cH^0 (\cM_{\alpha,k}\vert_H)$, we use again the fact that the map ${\rm Gr}_V^0(\cM_{\alpha,k})\overset{t\cdot}\longrightarrow {\rm Gr}_V^1(\cM_{\alpha,k})$ is an isomorphism; in particular, it is surjective. We thus have
\[ \cH^0 ( \cM_{\alpha,k}\vert_H) = \tfrac{{\rm Gr}_V^0(\cM_{\alpha,k})}{\de_t \cdot{\rm Gr}_V^1(\cM_{\alpha,k})}={\rm coker}\big({\rm Gr}_V^0(\cM_{\alpha,k}) 
\overset{\partial_tt}\longrightarrow {\rm Gr}_V^0(\cM_{\alpha,k})\big).\]
We use the description of $V^0\cM_{\alpha,k}$ in Lemma~\ref{lem_V_filt_M_alpha}. Note that
for every $k$, every $0\leq\ell\leq k$, and every
$m \in V^{\alpha-1}\big(\cM[1/t]\big)$, it follows from (\ref{eq_action_on_e}) that 
$$(\de_t t)\cdot \sum_{j =0}^d (-1)^{j}\big((t\de_t - \alpha +2)^j m\big) e_{\alpha,\ell+j} = m e_{\alpha,\ell} + \big((s+\alpha-1)^dm\big)e_{\alpha,\ell+d}.$$
We thus conclude that 
if $d$ is larger than the nilpotency index of $(s+\alpha-1)$ on ${\rm Gr}_V^{\alpha-1}\big(\cM[1/t]\big)$, then 
the image of $me_{\alpha,\ell}$ in 
$\varinjlim_{k} \cH^0(\cM_{\alpha,k}\vert_H)$ is $0$. This completes the proof of the lemma.
\end{proof}

In order to similarly describe ${\rm Gr}_V^0(\cM)$, we first note that $\cN_{1,0}\simeq\cO_{{\mathbf A}^1}[1/z]$, and thus for every $k\geq 0$,
we have an injective morphism of $\cD_X$-modules $\cM\hookrightarrow \cM_{1,k}$ given by $m\mapsto me_{1,0}$. We get induced morphisms
$\eta_1\colon
 {\rm Gr}_V^1(\cM) \to {\rm Gr}_V^1(\cM_{1,k})$ and $\eta_0\colon {\rm Gr}_V^0(\cM) \to {\rm Gr}_V^0(\cM_{1,k})$, which are injective by Proposition~\ref{properties_V_filtration1}.
  We consider the complex $(\cM \to \cM_{1,k})\vert_H$, which we
 define as the cone of the morphism $\cM\vert_H \to \cM_{1,k}\vert_H$. Therefore this is the total complex of the double complex
\[ \begin{tikzcd} {\rm Gr}_V^1(\cM) \ar[r,"\de_t"] \ar[d,"\eta_1"] & {\rm Gr}_V^0 \cM \ar[d,"\eta_0"] \\ {\rm Gr}_V^1(\cM_{1,k}) \ar[r,"\de_t"] & {\rm Gr}_V^0(\cM_{1,k}),
 \end{tikzcd} \]
that is, it is the following complex, placed in degrees $-2$, $-1$, $0$:
\[  {\rm Gr}_V^1(\cM) \xrightarrow[]{\de_t \oplus \eta_1} {\rm Gr}_V^0(\cM) \oplus {\rm Gr}_V^1(\cM_{1,k}) \xrightarrow[]{-\eta_0+\de_t}  {\rm Gr}_V^0(\cM_{1,k}). \]

The automorphism $T$ on $\cM_{1,k}$ induces the identity on $\cM$, hence the same holds on 
${\rm Gr}_V^1(\cM)$ and ${\rm Gr}_V^0(\cM)$. Therefore we can factor the map $T- {\rm Id}\colon (\cM \to \cM_{1,k})\vert_H \to (\cM \to \cM_{1,k})\vert_H$ as follows
\[ \begin{tikzcd} {\rm Gr}_V^1(\cM) \ar[d] \ar[r, "\de_t \oplus \eta_1"] & {\rm Gr}_V^0(\cM) \oplus {\rm Gr}_V^1(\cM_{1,k}) \ar[r,"-\eta_0+\de_t"] \ar[d,"T-{\rm Id}"]&{\rm Gr}_V^0(\cM_{1,k}) \ar[d,"T-{\rm Id}"] \\ 
0  \ar[r] \ar[d]&{\rm Gr}_V^1(\cM_{1,k}) \ar[r,"\de_t"] \ar[d,"0\oplus {\rm Id}"] &{\rm Gr}_V^0(\cM_{1,k})\ar[d,"{\rm Id}"] \\ 
{\rm Gr}_V^1(\cM) \ar[r, "\de_t \oplus \eta_1"] &{\rm Gr}_V^0(\cM)\oplus {\rm Gr}_V^1(\cM) \ar[r,"- \eta_0+\de_t"] & {\rm Gr}_V^0(\cM_{1,k}).\end{tikzcd}\]
We denote the top morphism of complexes in the diagram by
\begin{equation} \label{eq-formulaVar} 
\widetilde{\rm var}\colon (\cM \to \cM_{1,k})\vert_H \to \cM_{1,k}\vert_H ,
\end{equation}
and the lower one by
\begin{equation} \label{eq-formulaCan} 
\widetilde{\rm can}\colon \cM_{1,k}\vert_H \to (\cM\to \cM_{1,k})\vert_H,
\end{equation}
so that $T - {\rm Id} = \widetilde{\rm can} \circ \widetilde{\rm var}$.

\begin{lem} \label{lem-compare0} 
With the above notation, the following hold: 
\begin{enumerate}
\item[i)] 
We have $\cH^j \big((\cM \to \cM_{1,k})\vert_H\big)=0$ for $j\not\in\{-1,0\}$ and $\varinjlim_k \cH^0\big((\cM \to \cM_{1,k})\vert_H\big) = 0$.
\item[ii)] If $k$ is larger than the index of nilpotency of $(s+1)$ on ${\rm Gr}_V^1(\cM)$, then we have a functorial isomorphism
\begin{equation}\label{eq_Gr_0}
{\rm Gr}_V^0(\cM) \overset{\sim}\longrightarrow \cH^{-1}\big((\cM \to \cM_{1,k})\vert_H\big).
\end{equation}
\item[iii)] Via the isomorphisms (\ref{eq_Gr_1}) and (\ref{eq_Gr_0}), the map
${\rm Gr}_V^1(\cM)\overset{-\partial_t}\longrightarrow {\rm Gr}_V^0(\cM)$ corresponds to the natural morphism $\cH^{-1}(\widetilde{\rm can})$ 
and the map $a\mapsto \tfrac{\exp(2\pi i(s+1))-1}{s+1} ta$ corresponds to the
morphism $\cH^{-1}(\widetilde{\rm var})$.
\end{enumerate}
\end{lem}

\begin{proof}
We have already observed that $\eta_1$ is injective, hence $\cH^{-2}\big((\cM \to \cM_{1,k})\vert_H\big) = 0$, which gives the first assertion in i).
The second assertion in i) follows from the exact sequence
$$\cH^0(\cM_{1,k}\vert_H)\to \cH^0\big((\cM \to \cM_{1,k})\vert_H\big) \to \cH^1(\cM\vert_H)=0$$
and the fact that $\varinjlim_k \cH^0(\cM_{1,k}\vert_H) = 0$ by Lemma~\ref{lem-compare}.

We next prove the assertion in ii), using the description of the $V$-filtration on $\cM_{1,k}$ in Lemma~\ref{lem_V_filt_M_alpha}.
Note that by (\ref{eq_action_on_e}), we have
\[ \de_t \cdot \sum_{\ell =0}^k m_\ell e_{1,\ell} = \sum_{\ell = 0}^k (\de_t m_\ell)e_{1,\ell} +\sum_{\ell=0}^{k-1}\big(\tfrac{1}{t} m_{\ell+1}\big) e_{1,\ell}.\]
Given $m_0,\dots, m_k \in V^1\big(\cM[1/t]\big)$ and $a\in V^0\cM$, the image of 
$\big(a,\sum_{\ell} m_\ell e_{1,\ell}\big)$ in ${\rm Gr}_V^0(\cM) \oplus {\rm Gr}_V^1(\cM_{1,k})$
lies in the kernel of $-\eta_0+\de_t$ if and only if the following conditions hold:
\[ -a+ \de_t m_0 + \tfrac{1}{t}m_1= 0 \in V^{>0}\big(\cM[1/t]\big), \]
\[ \de_t m_1 + \tfrac{1}{t}m_2= 0 \in V^{>0}\big(\cM[1/t]\big),\ldots,
\de_t m_k = 0 \in V^{>0}\big(\cM[1/t]\big).\]

Since ${\rm Gr}_V^0\big(\cM[1/t]\big) \overset{t\cdot}\longrightarrow {\rm Gr}_V^1\big(\cM[1/t]\big)$ is injective, the above conditions are equivalent to
\[ m_1 = ta - t\de_t m_0 \mod V^{>1}\big(\cM[1/t]\big) = V^{>1}\cM, \]
\[ m_2 = - t\de_t m_1 \mod V^{>1}\cM,\ldots,t\de_t m_k = 0 \mod V^{>1}\cM.\]
Suppose now that $(s+1)^k\cdot {\rm Gr}_V^1(\cM) = (t\de_t)^k {\cdot \rm Gr}_V^1(\cM) = 0$. We thus have an isomorphism of $\cD_H$-modules
\[ {\rm Gr}_V^0(\cM) \oplus {\rm Gr}_V^1(\cM) \xrightarrow[]{\overline{\Psi}} \ker(-\eta_0+\de_t)\]
\[ (a,m_0) \mapsto \left(a,\sum_{\ell=1}^k \big((-t\de_t)^{\ell-1}ta\big) e_{1,\ell} + \sum_{\ell=0}^k \big((-t\de_t)^\ell m_0\big)e_{1,\ell}\right).\]

Moreover, we see that
\[ \overline{\Psi}(a-\de_t m_0,0) - \overline{\Psi}(a,m_0) = \big(-\de_t m_0,-\eta_1(m_0)\big) \in {\rm Im}(\de_t,\eta_1),\]
hence the map
\[ \Psi\colon {\rm Gr}_V^0(\cM) \to \cH^{-1}\big((\cM\to \cM_{1,k})\vert_H\big), \quad a \mapsto \overline{\Psi}(a,0)\]
is surjective. It is easily seen to be also injective, and thus it is an isomorphism.

We next turn to the assertions in iii).
Let $\varphi\colon {\rm Gr}_V^1(\cM) \to \cH^{-1}(\cM_{1,k}\vert_H)$ be the isomorphism in Lemma \ref{lem-compare}. Recall from the proof of that lemma that if
 $m_0 \in {\rm Gr}_V^1(\cM)$, then
\[ \varphi(m_0) = \sum_{\ell =0}^k (-t\de_t)^\ell m_0 e_{1,\ell}.\]
Since we have
\[ \Psi(-\de_t m_0) = \left(-\de_t m_0, \sum_{\ell =1}^k \big((-t\de_t)^{\ell}m_0\big)e_{1,\ell}\right) =\left(0, \sum_{\ell =0}^k \big((-t\de_t)^{\ell} m_0\big)e_{1,\ell}\right) \mod {\rm Im}(\de_t,\eta_1),\]
we see that $\widetilde{\rm can}\big(\varphi(m_0)\big) = \Psi(-\de_t m_0)$.

Finally, let $a\in {\rm Gr}_V^0 \cM$. We have $\Psi(a) = \big(a,\sum_{\ell =1}^k ((-t\de_t)^{\ell-1} ta) e_{1,\ell}\big)$ and $\widetilde{\rm var}$ maps this to the class of
\[ (T-{\rm Id})\left(\sum_{\ell =1}^k (-t\de_t)^{\ell-1}ta e_{1,\ell}\right) \in \cH^{-1} (\cM_{1,k}\vert_H).\]
We need to show that the coefficient of 
$e_{1,0}$ is $\sum_{\ell=1}^k\tfrac{(2\pi i)^{\ell}}{\ell !}(-t\partial_t)^{\ell-1}ta$. This follows from the fact that by the formula in 
Definition~\ref{rmk-moderateMonodromy} and the fact that $\xi_1 = 1$, we have
\[ (T-{\rm Id})e_{1,\ell} = \tfrac{(2\pi i)^\ell}{\ell !} e_{1,0} + \varepsilon,\]
where $\varepsilon$ is a linear combination of $e_{1,j}$ for $1\leq j \leq \ell$.
This completes the proof of the lemma.
\end{proof}

\begin{cor} \label{cor-regholo} 
If $\cM$ is a regular holonomic $\cD_X$-module that has a $V$-filtration with respect to $t$, then the $\cD_H$-module ${\rm Gr}_V^\alpha(\cM)$ is regular holonomic
for every $\alpha\in\QQ$. 
\end{cor}

\begin{proof} 
By Proposition~\ref{rmk1_Vfiltration}, it is enough to prove the assertion for 
$\alpha \in [0,1]$. It is a basic fact that regular holonomicity is preserved under duality and pull-back (see \cite[Chapter~6]{HTT}). First, this implies that the tensor product 
of regular holonomic $\cD$-modules is regular holonomic, and thus all $\cM_{\alpha,k}$ are regular holonomic. Second, it implies the fact that $\cH^{-1}(\cM_{\alpha,k}\vert_H)$
and $\cH^{-1}\big((\cM\to \cM_{1,k})\vert_H\big)$ are regular holonomic and we conclude that ${\rm Gr}_V^\alpha(\cM)$ is regular holonomic using
 Lemmas~ \ref{lem-compare} and \ref{lem-compare0}.
\end{proof}

\bigskip

For proof of Theorem \ref{thm-comparison}, we need to understand ${\rm DR}^{\rm an}_H\big({\rm Gr}_V^\alpha(\cM)\big)$ for a regular holonomic 
$\cD_X$-module $\cM$ that admits a $V$-filtration with respect to $t$. Note that
by exactness of direct limits, the vanishings and the isomorphisms in Lemma \ref{lem-compare} and Lemma \ref{lem-compare0}
give functorial quasi-isomorphisms 
\[ {\rm Gr}_V^\alpha(\cM) \to \varinjlim_{k} (\cM_{\alpha,k}\vert_H)\simeq \big(\cM \otimes_{\cO} \varinjlim_k t^*(\cN_{\alpha,k})\big)\vert_H 
\quad\text{ for } \quad \alpha \in (0,1],\]
and
\[ {\rm Gr}_V^0(\cM) \to \varinjlim_k \big((\cM \to \cM_{1,k})\vert_H\big) =\big(\cM \to \cM\otimes_{\cO_X} \varinjlim_k t^*(\cN_{\alpha,k})\big)\vert_H,\]
hence we only need to understand the objects on the right hand side.

We note that the set-up for topological nearby cycles can also be used for the sheaf $\cO^{\rm an}_X$. Define
\[ \Psi_t(\cO^{\rm an}_X) : = i^{-1} j_* \chi_* \cO_{\widetilde{U}},\]
where $j\colon U^{\rm an} = \{t\neq 0\} \subseteq X^{\rm an}$ is the open embedding and $\chi\colon \widetilde{U} \to U^{\rm an}$ is induced 
by the universal cover $\exp\colon \C \to \C^*$. By adjunction,  we have an injective morphism $i^{-1}\cO^{\rm an}_X\to \Psi_t(\cO^{\rm an}_X)$, whose cokernel we denote $\Phi_t(\cO^{\rm an}_X)$.
We thus have an exact triangle in the derived category of $i^{-1}\cD^{\rm an}_X$-modules
\[ 0 \to i^{-1}\cO^{\rm an}_X \to \Psi_t(\cO^{\rm an}_X) \xrightarrow[]{\rm can} \Phi_t(\cO^{\rm an}_X) \to 0.\]

We also recall the \emph{solution complex} ${\rm Sol}^{\rm an}_X(\cM):=\mathbf R\cH om_{\cD_{X^{\rm an}}}(\cM^{\rm an}, \cO_{X^{\rm an}})$ associated to $\cM$. This is related to the 
analytic de Rham complex by duality: if  $\cM^*$ is the dual of the regular holonomic $\cD_{X}$-module $\cM$, then we have by \cite[Sect. 4.6]{HTT} a canonical isomorphism
 \begin{equation} \label{eq-DRDuality} 
 {\rm Sol}^{\rm an}_X(\cM) \simeq {\rm DR}^{\rm an}_X(\cM^*).
\end{equation}

This solution functor is more convenient thanks to the following theorem:

\begin{thm} \label{thm-MM213} $\big($\cite[Thm 2.1-3]{MaisonobeMebkhout}$\big)$ 
For every coherent $\cD_X$-module $\cM$, we have a functorial isomorphism between the exact triangle
\begingroup\makeatletter\def\f@size{10}\check@mathfonts
\[ i^{-1}{\rm Sol}^{\rm an}_X(\cM) \to \Psi_t {\rm Sol}_X^{\rm an}(\cM) \xrightarrow[]{\rm can} \Phi_t {\rm Sol}_X^{\rm an}(\cM) \xrightarrow[]{+1}\]
\endgroup
and the exact triangle
\begingroup\makeatletter\def\f@size{10}\check@mathfonts
\[ {\mathbf R}\cH om_{i^{-1}\cD^{\rm an}_X}(i^{-1}\cM^{\rm an},i^{-1}\cO^{\rm an}_X) \to {\mathbf R}\cH om_{i^{-1}\cD^{\rm an}_X}\big(i^{-1}\cM^{\rm an}, \Psi_t(\cO^{\rm an}_X)\big) \to {\mathbf R}\cH om_{i^{-1}\cD^{\rm an}_X}\big(i^{-1}\cM^{\rm an},\Phi_t(\cO^{\rm an}_X)\big) \xrightarrow[]{+1}.\]
\endgroup
\end{thm}

\begin{proof} 
By a standard argument, it suffices to construct a functorial morphism of triangles. 
Indeed, then checking that this is an isomorphism can be done locally, and since $\cM$ is coherent, it is enough to check that
we have an isomorphism when $\cM=\cD_X$, in which case the assertion is clear. Let us put $q=j\circ \chi$. Note that $q$ is \'{e}tale, hence the $\cD$-module push-forward agrees with the $\cO$-module push-forward.

To construct the morphism of triangles, we take a resolution of $\cO^{\rm an}_X$ by injective $\cD_X^{\rm an}$-modules $\cO^{\rm an}_X \to I^\bullet$. Then $q^{-1}I^\bullet$ is a $q^{-1}\cD^{\rm an}_X$-injective resolution of $q^{-1}\cO^{\rm an}_X$ and $q_* q^{-1} I^\bullet$ is a $\cD^{\rm an}_X$-module resolution of $q_* q^{-1}\cO^{\rm an}_X$. We then have functorial morphisms of triangles
\[\begin{tikzcd} \cH om_{\cD^{\rm an}_X}(\cM^{\rm an}, I^\bullet) \ar[r,"\alpha"] \ar[d]&q_*q^{-1}\cH om_{\cD^{\rm an}_X}(\cM^{\rm an}, I^\bullet)  \ar[r] \ar[d]& {\rm Cone}(\alpha) \ar[d] \ar[r,"+1"] &{ }\\
\cH om_{\cD^{\rm an}_X}(\cM^{\rm an}, I^\bullet)  \ar[r,"\alpha' "] \ar[d]&\cH om_{\cD^{\rm an}_X}(q_*q^{-1}\cM^{\rm an},q_*q^{-1} I^\bullet)  \ar[r] \ar[d]& {\rm Cone}(\alpha') \ar[d] \ar[r,"+1"] &{ }\\
\cH om_{\cD^{\rm an}_X}(\cM^{\rm an}, I^\bullet)  \ar[r,"\beta"]&\cH om_{\cD^{\rm an}_X}(\cM^{\rm an},q_*q^{-1} I^\bullet)  \ar[r]& {\rm Cone}(\beta)  \ar[r,"+1"] &{ }, \end{tikzcd}\]
and thus a morphism
\[\begin{tikzcd} 
{\mathbf R}\cH om_{\cD^{\rm an}_X}(\cM^{\rm an}, \cO^{\rm an}_X) \ar[r,"\alpha"] \ar[d]& {\mathbf R}q_*q^{-1} {\mathbf R}\cH om_{\cD^{\rm an}_X}(\cM^{\rm an}, \cO^{\rm an}_X)  \ar[r] \ar[d]& {\rm Cone}(\alpha) \ar[d] \ar[r,"+1"] &{ }\\
{\mathbf R}\cH om_{\cD^{\rm an}_X}(\cM^{\rm an}, \cO^{\rm an}_X)  \ar[r,"\beta"]& {\mathbf R}\cH om_{\cD^{\rm an}_X}(\cM^{\rm an},q_*q^{-1}\cO^{\rm an}_X)  \ar[r]& {\rm Cone}(\beta)  \ar[r,"+1"] &{ }. \end{tikzcd}\]
By applying $i^{-1}$ and using the functorial morphism
\[ i^{-1}{\mathbf R}\cH om_{\cD^{\rm an}_X}(-,-) \to {\mathbf R}\cH om_{i^{-1}\cD^{\rm an}_X}\big(i^{-1}(-),i^{-1}(-)\big),\]
we obtain the desired morphism of triangles.
\end{proof}

We can view the sections of $\Psi_t(\cO^{\rm an}_X)$ as multi-valued holomorphic functions on $U^{\rm an}$. Consider the subsheaf $\Psi_t^{\rm fd}(\cO^{\rm an}_X)$ 
of $\Psi_t(\cO_X^{\rm an})$ consisting of those  functions that are annihilated by some nonzero complex polynomial in the monodromy operator:
these are the functions of \emph{finite determination}. This sheaf contains $i^{-1}\cO^{\rm an}_X$, as monodromy acts trivially on regular functions. 
We may thus consider the quotient denoted by $\Phi_t^{\rm fd}(\cO^{\rm an}_X)$.

From the point of view of complexes of solutions, we may pass to functions of finite determination:
\begin{thm} \label{thm-MM314} $\big($\cite[Thm. 3.1-4]{MaisonobeMebkhout}$\big)$ 
For every coherent $\cD_X$-module $\cM$, the canonical morphisms
\[ {\mathbf R}\cH om_{\cD^{\rm an}_X}\big(\cM^{\rm an}, \Psi_t^{\rm fd}(\cO^{\rm an}_X)\big) \to {\mathbf R}\cH om_{\cD^{\rm an}_X}\big(\cM^{\rm an},\Psi_t(\cO^{\rm an}_X)\big)\quad\text{and}\]
\[ {\mathbf R}\cH om_{\cD^{\rm an}_X}\big(\cM^{\rm an}, \Phi_t^{\rm fd}(\cO^{\rm an}_X)\big) \to {\mathbf R}\cH om_{\cD^{\rm an}_X}\big(\cM^{\rm an},\Phi_t(\cO^{\rm an}_X)\big)\]
are isomorphisms
\end{thm}

By \cite[Thm. 3.1-2]{MaisonobeMebkhout}, a multi-valued holomorphic function is of finite determination if and only if it
can be written as a finite sum
\[\sum_{\alpha,j} c_{\alpha,j} t^\alpha \log(t)^j,\] where $\alpha \in \C$, $j \in {\mathbf Z}_{\geq 0}$, and each $c_{\alpha,j}$ is a uni-valued function;
moreover, this expression is unique if we require the classes in ${\mathbf C}/{\mathbf Z}$ of the $\alpha$ with $c_{\alpha,j}\neq 0$ to be pairwise distinct.
For any $\alpha\in {\mathbf C}$, we consider the subsheaf $\Psi_{t,\alpha}(\cO^{\rm an}_X)$ of $\Psi_{t}^{\rm fd}(\cO_X^{\rm an})$ consisting of those sections of the form 
$\sum_{j\geq 0} c_j t^{1-\alpha} \log(t)^j$ for some fixed $\alpha \in \mathbf C$, with the $c_j$ uni-valued.
In what follows we fix a section $\sigma\colon \mathbf C/\mathbf Z \to \mathbf C$ of the canonical projection such that $\sigma(0)=1$, so
we get a decomposition
\begin{equation} \label{eq-fdDecomp} 
\Psi_t^{\rm fd}(\cO^{\rm an}_X) = \bigoplus_{\alpha \in {\rm Im}(\sigma)} \Psi_{t,\alpha}(\cO^{\rm an}_X).
\end{equation}

A section of $\Psi_{t,\alpha}(\cO^{\rm an}_X)$, as above, has \emph{moderate growth} if all coefficients $c_j$ are meromorphic, with poles at most along $(t=0)$.
We denote the subsheaves of such functions by $\Psi_{t,\alpha}^m(\cO^{\rm an}_X) \subseteq \Psi_{t,\alpha}(\cO^{\rm an}_X)$ and 
$\Psi_{t}^m(\cO^{\rm an}_X) \subseteq \Psi_{t}(\cO^{\rm an}_X)$. We also put $\Phi_{t}^m(\cO^{\rm an}_X)=\Psi_{t}^m(\cO^{\rm an}_X)/i^{-1}\cO_X^{\rm an}$.
Note that $\Psi_{t,\alpha}^m(\cO^{\rm an}_X)$ is filtered by 
the subsheaves $\Psi_{t,\alpha,k}(\cO^{\rm an}_X)$ which only allow nonzero coefficients $c_j$ for $j \leq k$. The natural morphism 
$i^{-1} \cO^{\rm an}_X \to \Psi_t(\cO^{\rm an}_X)$ factors through $\Psi_{t,1,k}(\cO^{\rm an}_X)$ for any $k\geq 1$.

\begin{thm} \label{thm-MM331} $\big($\cite[Thm. 3.3-1]{MaisonobeMebkhout}$\big)$ For every regular holonomic $\cD_X$-module $\cM$, the natural morphisms
\[{\mathbf R}\cH om_{i^{-1}\cD^{\rm an}_X}\big(i^{-1}\cM^{\rm an},\Psi_{t,\alpha}^m(\cO^{\rm an}_X)\big) \to {\mathbf R}\cH om_{i^{-1}\cD^{\rm an}_X}\big(i^{-1}\cM^{\rm an},\Psi_{t,\alpha}(\cO^{\rm an}_X)\big)\]
\[{\mathbf R}\cH om_{i^{-1}\cD^{\rm an}_X}\big(i^{-1}\cM^{\rm an},\Phi_{t}^m(\cO^{\rm an}_X)\big) \to {\mathbf R}\cH om_{i^{-1}\cD^{\rm an}_X}\big(i^{-1}\cM^{\rm an},\Phi_{t}^{\rm fd}(\cO^{\rm an}_X)\big)\]
are isomorphisms. 
\end{thm}

We can now prove the main result of this section. 

\begin{proof}[Proof of Theorem~\ref{thm-comparison}]
Applying Theorem \ref{thm-MM213} to $\cM^*$ and using the isomorphism (\ref{eq-DRDuality}), we 
see that the following triangles are naturally isomorphic, where $\cM'=(\cM^*)^{\rm an}$:
\begingroup\makeatletter\def\f@size{10}\check@mathfonts
\[ i^{-1} {\rm DR}_X^{\rm an}(\cM) \to \Psi_t {\rm DR}_X^{\rm an}(\cM) \xrightarrow[]{\rm can} \Phi_t {\rm DR}_X^{\rm an}(\cM) \xrightarrow[]{+1}\quad\text{and}\quad\]
\endgroup
\begingroup\makeatletter\def\f@size{10}\check@mathfonts
\[ {\mathbf R}\cH om_{i^{-1}\cD^{\rm an}_X}(i^{-1}\cM',i^{-1}\cO^{\rm an}_X) \to {\mathbf R}\cH om_{i^{-1}\cD^{\rm an}_X}\big(i^{-1}\cM',\Psi_t(\cO^{\rm an}_X)\big) \to 
{\mathbf R}\cH om_{i^{-1}\cD^{\rm an}_X}\big(i^{-1}\cM', \Phi_t(\cO^{\rm an}_X)\big) \xrightarrow[]{+1}.\]
\endgroup

Since the dual of a holonomic module is holonomic, hence coherent, we can apply Theorem \ref{thm-MM314} and we see that the latter triangle is isomorphic to
\begingroup\makeatletter\def\f@size{10}\check@mathfonts
\[ {\mathbf R}\cH om_{i^{-1}\cD^{\rm an}_X}(i^{-1}\cM',i^{-1}\cO^{\rm an}_X) \to {\mathbf R}\cH om_{i^{-1}\cD^{\rm an}_X}\big(i^{-1}\cM',\Psi^{\rm fd}_t(\cO^{\rm an}_X)\big) \to {\mathbf R}\cH om_{i^{-1}\cD^{\rm an}_X}\big(i^{-1}\cM', \Phi^{\rm fd}_t(\cO^{\rm an}_X)\big) \xrightarrow[]{+1}.\]
\endgroup

Using the decomposition (\ref{eq-fdDecomp}) and Theorem \ref{thm-MM331}, this triangle is isomorphic to
\[ {\mathbf R}\cH om_{i^{-1}\cD^{\rm an}_X}(i^{-1}\cM',i^{-1}\cO^{\rm an}_X) \to 
\bigoplus_{\alpha \in {\rm Im}(\sigma)} {\mathbf R}\cH om_{i^{-1}\cD^{\rm an}_X}\big(i^{-1}\cM',\Psi^{m}_{t,\alpha}(\cO^{\rm an}_X)\big)\] \[ \to 
\bigoplus_{\alpha \in {\rm Im}(\sigma)} {\mathbf R}\cH om_{i^{-1}\cD^{\rm an}_X}\big(i^{-1}\cM', \Phi^{m}_{t,\alpha}(\cO^{\rm an}_X)\big) \xrightarrow[]{+1}.\]

For $\alpha \notin \mathbf Z$, the second map in the above triangle is an isomorphism. In what follows, we mainly focus on the case $\alpha=1$,
which is the most interesting.

By definition of the sheaves $\cN_{\alpha,k}$, we have an isomorphism $i^{-1}t^*(\cN^{\rm an}_{\alpha,k}) \cong \Psi_{t,\alpha,k}(\cO^{\rm an}_X)$, hence
 $i^{-1}\varinjlim_k t^*(\cN^{\rm an}_{\alpha,k}) \cong \Psi^m_{t,\alpha}(\cO^{\rm an}_X)$. Let $C_k^\bullet$ be the complex $\cO^{\rm an}_X \to t^*(\cN^{\rm an}_{1,k})$ in the derived category of analytic $\cD$-modules on $X$, with regular holonomic cohomology.
Then $i^{-1} \varinjlim_k C^\bullet_k = \Phi^m_{t,1}(\cO^{\rm an}_X)$. The above triangle, for $\alpha=1$, is thus isomorphic to 
\[ i^{-1} {\mathbf R}\cH om_{\cD^{\rm an}_X}(\cM',\cO^{\rm an}_X) \to  i^{-1} {\mathbf R}\cH om_{\cD^{\rm an}_X}\big(\cM', \varinjlim_k t^*(\cN_{1,k})\big) \to 
 i^{-1}{\mathbf R}\cH om_{\cD^{\rm an}_X}(\cM', \varinjlim_k C^\bullet_k) \xrightarrow[]{+1}.\]

By (the analytic version of) \cite[Lemma 2.6.13]{HTT}, which relates $- \otimes_{\cO} -$ and $\cH om_{\cD}(-,-)$, this triangle is isomorphic to
\[ i^{-1} {\rm DR}^{\rm an}_X(\cM) \to i^{-1} {\rm DR}^{\rm an}_X\big(\cM \otimes_{\cO_X} \varinjlim_k t^*(\cN_{1,k} )\big) 
\to i^{-1} {\rm DR}^{\rm an}_X(\cM\otimes_{\cO_X}  \varinjlim_k C^\bullet_k) \xrightarrow[]{+1}.\]

By the compatibility of the Riemann-Hilbert correspondence with pull-back 
(see \cite[Equation 7.1.5]{HTT}), for every regular holonomic $\cD_X$-module $\cN$, we have
\begin{equation} \label{eq-CKK} i^{-1} {\rm DR}_X(\cN) \cong {\rm DR}_H(\cN\vert_H).
\end{equation}
Hence, the above exact triangle is isomorphic to

\[ {\rm DR}^{\rm an}_H(\cM\vert_H) \to  {\rm DR}^{\rm an}_H\big((\cM \otimes_{\cO_X} \varinjlim_k t^*(\cN_{1,k}) )\vert_H\big)\to {\rm DR}^{\rm an}_H\big((\cM\otimes_{\cO_X}  \varinjlim_k C^\bullet_k)\vert_H\big) \xrightarrow[]{+1}\]
which, in the notation we have been using, is
\[ {\rm DR}^{\rm an}_H(\cM\vert_H) \to {\rm DR}^{\rm an}_H\big( \varinjlim_k(\cM_{1,k})\vert_H\big) \to {\rm DR}^{\rm an}_H\big(\varinjlim_k(\cM \to \cM_{1,k})\vert_H\big) \xrightarrow[]{+1}.\]

Finally, we have already shown that this is isomorphic to
\[ {\rm DR}^{\rm an}_H(\cM\vert_H) \to {\rm DR}^{\rm an}_H\big({\rm Gr}_V^1(\cM)\big) \to {\rm DR}^{\rm an}_H\big({\rm Gr}_V^0(\cM)\big) \xrightarrow[]{+1},\]
proving the claim in the case $\alpha=1$. The case $\alpha\not\in\ZZ$ is proved similarly and we leave it as an exercise for the reader.
\end{proof}

\section{The $V$-filtration on $B_f$ and invariants of singularities}\label{section_b_invar_sing}

In this final section we discuss some connections between the invariants of singularities that we have seen so far (the roots of the Bernstein-Sato polynomial of $f$ and the $V$-filtration
of $\cO_X$ with respect to $f$) and other invariants of singularities. 

\subsection{The reduced Bernstein-Sato polynomial}

We first note that the Bernstein-Sato polynomial can be defined for any hypersurface. Let $X$ be a smooth, irreducible algebraic variety over an algebraically closed field $k$ of characteristic $0$ and let $H$ be a hypersurface in $X$. We can choose a finite open cover $X=\bigcup_{i=1}^rU_i$ and $f_i\in\cO_X(U_i)$ nonzero such that
$H\cap U_i$ is defined in $U_i$ by $f_i$ for all $i$. The \emph{Bernstein-Sato} polynomial $b_H(s)$ is defined as ${\rm lcm}\big\{b_{f_i}(s)\mid 1\leq i\leq r\big\}$.
It follows from Remarks~\ref{B_fcn_open_cover_v0} and \ref{Vfilt_rescaling} that the definition is independent of the choice of open cover and the $f_i$.

\begin{prop}\label{trivial_root_b_fcn_general}
If $H$ is a nonempty hypersurface in $X$, then $(s+1)$ divides $b_H(s)$.
\end{prop}

\begin{proof}
We may and will assume that $X$ is affine and $H$ is defined by $f\in\cO_X(X)$ noninvertible.
By definition, we have
$$b_f(s)f^s \in \cD_X[s] ff^s.$$
After specializing to $s=-1$ (see Remark~\ref{rmk_specialize_s}), we obtain
$$b_f(-1)\tfrac{1}{f}\in\cD_X\cdot 1=\cO_X,$$
and since $f$ is not invertible, we conclude that $b_f(-1)=0$.
\end{proof}

In light of this proposition, if $H$ is nonempty, we write $b_H(s)=(s+1)\cdot\widetilde{b}_H(s)$; the polynomial $\widetilde{b}_H(s)$ is the 
\emph{reduced Bernstein-Sato polynomial} of $H$. If $H$ is defined by $f\in\cO_X(X)$, then we also write $\widetilde{b}_f$ for $\widetilde{b}_H$.

\subsection{The $V$-filtration and multiplier ideals}\label{section_multiplier}

In this section we describe the connection between the $V$-filtration and the Bernstein-Sato polynomial, on one side, and multiplier ideals, on the other side. We begin with a quick introduction to multiplier ideals. For a more detailed discussion and for the proofs of some of the results we state, we refer to \cite[Chapter~9]{Lazarsfeld}.

Let $X$ be a smooth, irreducible algebraic variety over an algebraically closed field $k$ of characteristic $0$. Suppose that $H$ is a (nonempty) hypersurface on $X$. 
We consider a log resolution $\pi\colon Y\to X$ of the pair $(X,H)$.

\begin{defi}
For every $\lambda\in\QQ_{\geq 0}$, the \emph{multiplier ideal} $\cJ(X,\lambda H)$ is given by
$$\cJ(X,\lambda H)=\pi_*\cO_Y\big(K_{Y/X}-\lfloor \lambda \pi^*(H)\rfloor\big).$$
\end{defi}

We note that, by definition, if $\pi^*(H)=\sum_{i=1}^Na_iE_i$, then $\lfloor \lambda \pi^*(H)\rfloor=\sum_{i=1}^N\lfloor\lambda a_i\rfloor E_i$. We also note that since $\lfloor \lambda \pi^*(H)\rfloor$ is an effective
divisor and since $K_{Y/X}$ is an effective exceptional divisor, we have
$$\cJ(X,\lambda H)\subseteq\pi_*\cO_Y(K_{Y/X})=\cO_X,$$
hence $\cJ(X,\lambda H)$ is indeed a coherent ideal of $\cO_X$. 
It is a basic fact that the definition of multiplier ideals is independent of the choice of log resolution (see \cite[Theorem~9.2.18]{Lazarsfeld}). 

In what follows we list a few properties of multiplier ideals. Most of these follow in a straightforward way from the definition:
\begin{enumerate}
\item[1)] If $\lambda\geq\mu$, then 
$$\cJ(X,\lambda H)\subseteq\cJ(X,\mu H).$$
This is a consequence of the fact that the divisor $\lfloor\lambda \pi^*(H)\rfloor-\lfloor\mu \pi^*(H)\rfloor$ is effective.
\item[2)] It is an immediate consequence of the properties of the round-down function that for every $\lambda\in\QQ_{\geq 0}$, there is $\epsilon>0$
such that
$$\cJ(X,\lambda H)=\cJ(X,\mu H)\quad\text{for}\quad \lambda\leq\mu\leq\lambda+\epsilon.$$
\item[3)] In particular, we have $\cJ(X,\mu H)=\cO_X$ for $0\leq\mu\ll 1$.
\item[4)] We say that $\lambda\in\QQ_{>0}$ is a \emph{jumping number} of $(X,H)$ if 
$$\cJ(X,\lambda H)\subsetneq \cJ\big(X,(\lambda-\epsilon)H\big)\quad\text{for all}\quad \epsilon>0.$$
Note that in this case we have $a_i\lambda\in\ZZ$ for some $i$ with $1\leq i\leq N$. Therefore the set of jumping numbers of $(X,H)$ is contained in $\tfrac{1}{\ell}\ZZ_{>0}$ 
for some positive integer $\ell$.
\item[5)] The smallest jumping number of $(X,H)$ is the \emph{log canonical threshold} $\lct(X,H)$
(also written $\lct(H)$ when $X$ is understood):
$$\lct(X,H)=\min\big\{\lambda>0\mid\cJ(X,\lambda H\big)\neq\cO_X\big\}.$$
Note that, using the notation in (\ref{formula_pull_back_H}),
 we have $1\in\cJ(X,\lambda H)$ if and only if $k_i\geq \lfloor \lambda a_i\rfloor$ for all $i$; equivalently,
$k_i>\lambda a_i-1$ for all $i$. We thus deduce that
\begin{equation}\label{eq_formula_lct}
\lct(X,H)=\min_i\tfrac{k_i+1}{a_i}.
\end{equation}
\item[6)] $1$ is always a jumping number of $(X,H)$. In order to see this, we may replace $X$ by any open subset $U$ such that  $U\cap H\neq\emptyset$.
We may thus assume that $H=mZ$ for some smooth hypersurface $Z$ and some positive integer $m$. In this case we may take $\pi$ to be the identity and we see that
$$\cJ(X,H)=\cO_X(-mZ)\quad\text{and}\quad \cJ\big(X,(1-\epsilon)H\big)=\cO_X\big(-(m-1)Z\big)\,\,\text{for}\,\,0<\epsilon\ll 1.$$
\item[7)] For every $\lambda\geq 1$, it follows from the definition and the projection formula that
$$\cJ(X,\lambda H)=\cO_X(-H)\cdot\cJ\big(X,(\lambda-1)H\big).$$
In particular, we see that $\lambda>1$ is a jumping number of $(X,H)$ if and only if $\lambda-1$ has this property. 
This means that as invariants of singularities, it is enough to consider the multiplier ideals $\cJ(X,\lambda H)$ for $\lambda<1$.
\item[8)] If $k=\CC$, then there is an analytic description of $\cJ(X,\lambda H)$ that is more intuitive than the algebraic one that we gave. Suppose, for simplicity,
that $H$ is defined by $f\in\cO_X(X)$. In this case we have
$$\cJ(X,\lambda H)=\big\{g\in\cO_X(X)\mid \tfrac{|g|^2}{|f|^{2\lambda}}\,\,\text{is locally integrable}\big\}.$$
The local integrability condition means that for every $P\in X$, if $z_1,\ldots,z_n$ are local coordinates around $P$, then there is an open neighborhood $U$
of $P$ such that $\int_U\tfrac{|g|^2}{|f|^{2\lambda}}dzd\overline{z}<\infty$.
The equivalence with the formula in the algebraic definition is shown using the Change of Variable formula and the fact that, in one variable,
$\tfrac{1}{|z|^{2\lambda}}dzd\overline{z}$ is locally integrable if and only if $\lambda<1$ (see \cite[Chapter~9.3.D]{Lazarsfeld} for details). 
\end{enumerate}

The following 3 results relate the $\cD$-module theoretic invariants of $f$ to the multiplier ideals of $H$.
The description of the log canonical threshold in the following theorem was proved by Koll\'{a}r \cite{Kollar},  by making use of Lichtin's upper bound \cite{Lichtin} 
 for the roots of the Bernstein-Sato
polynomial that appears in Theorem~\ref{thm_DM_bfcn}i).

\begin{thm}\label{thm_LichtinKollar}
The largest root of $b_H(s)$ is $-\lct(X,H)$.
\end{thm}

Partially generalizing this to higher jumping numbers, we have the following result due to Ein, Lazarsfeld, Smith, and Varolin \cite{ELSV}.

\begin{thm}\label{thm_ELSV}
If $\lambda\leq 1$ is a jumping number of $(X,H)$, then $b_H(-\lambda)=0$.
\end{thm}

As we will see, both the above results follow from the following theorem of Budur and Saito \cite{BudurSaito}, that describes the multiplier ideals of $(X,H)$
via the $V$-filtration of $f$, when $H$ is defined by $f\in\cO_X(X)$. Recall that $\iota\colon X\hookrightarrow X\times\AA^1$ is the graph embedding associated to $f$.

\begin{thm}\label{thm_BudurSaito}
For every $\lambda\in\QQ_{\geq 0}$, we have
$$\cJ(X,\lambda H)=\big\{g\in\cO_X\mid g\delta\in V^{>\lambda}\iota_+(\cO_X)\big\}.$$
\end{thm}

The proof of this result in \cite{BudurSaito} makes use of results in Saito's theory of mixed Hodge modules \cite{Saito-MHM}. Here we give a more elementary proof
following \cite{DM}.

\begin{proof}[Proof of Theorem~\ref{thm_BudurSaito}]
We may and will assume that $X$ is affine, with $R=\cO_X(X)$ and let $\pi\colon Y\to X$ be a log resolution of $(X,H)$. Let $g\in R$ be nonzero. We use the notation in (\ref{formula_pull_back_H}) and denote by
$b_i$ the coefficient of $E_i$ in $\pi^*\big({\rm div}(g)\big)$. By definition, we have $g\in\cJ(X,\lambda H)$ if and only if 
$b_i+k_i\geq\lfloor\lambda a_i\rfloor$ for all $i$, which is the case if and only if $\lambda<\min_i\tfrac{b_i+k_i+1}{a_i}=:{\rm lct}_g(X,H)$. 
On the other hand, it follows from Corollary~\ref{prop_char_V_filt}  that $g\delta\in V^{>\lambda}\iota_+(\cO_X)$ if and only if all roots of $b_{g\delta}$ are $<-\lambda$.
We thus conclude that the assertion in the theorem is equivalent to the fact that for every nonzero $g\in\cO_X(X)$, the largest root of $b_{g\delta}$ is $-\lct_g(X,H)$.

The fact that every root of $b_{g\delta}$ is $\leq -\lct_g(X,H)$ follows from Theorem~\ref{thm_DM_bfcn}iii). In order to complete the proof, 
it is enough to show that if $g\in \cJ\big(X,(\lambda-\epsilon) H\big)$ for every $\epsilon>0$, but
$g\not\in\cJ(X,\lambda H)$, then $b_{g\delta}(-\lambda)=0$. 
For this, 
we use the Lefschetz Principle to reduce to the
case when the ground field is $\CC$, when we can use the analytic description of multiplier ideals.
The argument follows closely the argument for the proof of Theorem~\ref{thm_ELSV} in \cite{ELSV}. 

We may assume that we have coordinates $z_1,\ldots,z_n$ on $X$. 
Since $g\not\in\cJ(X,\lambda H)$, it follows that there is a point $x_0\in X$ such that $\tfrac{|g|^2}{|f|^{2\lambda}}$ is not integrable in any neighborhood of $x_0$.
On the other hand, since $\tfrac{|g|^2}{|f|^{2\mu}}$ is locally integrable at $x_0$ for all $\mu<\lambda$, we can choose an open ball $B$ around $x_0$ (with respect to our coordinates)
such that $\int_B\tfrac{|g|^2}{|f|^{2\mu}}dzd\overline{z}<\infty$ for all $\mu<\lambda$ (the fact that we can choose $B$ independently of $\mu$ follows from the proof of the analytic characterization 
of the multiplier ideal, see \cite[Chapter~9.3.D]{Lazarsfeld}). 

By definition of the $b$-function, if we put $b=b_{g\delta}$, then there is $P\in D_R[s]$ such that
$$b(s)gf^s=P\cdot gf^{s+1}.$$
Suppose that we are in an open subset $U$ of $X^{\rm an}$ where a branch of ${\rm log}(f)$ is defined, hence $f^{\mu}={\rm exp}\big(\mu\cdot {\rm log}(f)\big)$ is defined
for every $\mu\in\RR$. We choose $\mu$ such that $0<\lambda-\mu\ll 1$
and specialize to  $s=-\mu$ (see Remark~\ref{rmk_specialize_s}), to get
$$b(-\mu)gf^{-\mu}=P(-\mu)\cdot gf^{1-\mu}.$$
Applying complex conjugation, we obtain
$$b(-\mu)\overline{g}\overline{f}^{-\mu}=\overline{P}(-\mu) \overline{g}\overline{f}^{1-\mu}.$$
Using the fact that that 
$$P\cdot (\overline{h_1}h_2)=\overline{h_1}P\cdot h_2\quad\text{and}\quad \overline{P}\cdot (\overline{h_1}h_2)=h_2\overline{P}\cdot \overline{h_1}$$
for every holomorphic functions $h_1$ and $h_2$, we conclude that if $R=P\overline{P}$, then
\begin{equation}\label{eq_integral_formula}
b(-\mu)^2\cdot \tfrac{|g|^2}{|f|^{2\mu}}=R(-\mu)\cdot\tfrac{|g|^2}{|f|^{2(\mu-1)}}.
\end{equation}
Note that this formula does not depend on $U$ and it makes sense on $X\smallsetminus H$. 

Since both sides of (\ref{eq_integral_formula}) are integrable on $B$, 
if $\varphi$ is a smooth function with compact support on $B$, then the following integrals are finite
$$\int_B b(-\mu)^2\cdot \tfrac{|g|}{|f|^{2\mu}}\varphi dzd\overline{z}=\int_B\big(R(-\mu)\cdot\tfrac{|g|^2}{|f|^{2(\mu-1)}}\big)\varphi dzd\overline{z}
=\int_B\tfrac{|g|^2}{|f|^{2(\mu-1)}}\psi dzd\overline{z},$$
where $\psi=\widetilde{R}(-\mu)$ (here $\widetilde{R}$ is the classical adjoint of $R$) and the last equality in the displayed formula is a consequence of the
Stokes Theorem. In particular, by choosing $\varphi$ to be nonnegative and with $\varphi=1$ on a smaller ball $B'\subsetneq B$, we get
$$b(-\mu)^2\cdot  \int_{B'}\tfrac{|g|^2}{|f|^{2\mu}}dzd\overline{z}\leq \int_B\tfrac{|g|^2}{|f|^{2(\mu-1)}}\psi dzd\overline{z}\leq M,$$
for some constant $M$ that is independent of $\mu$. If $b(-\lambda)\neq 0$, then we conclude that$ \int_{B'}\tfrac{|g|^2}{|f|^{2\mu}}dzd\overline{z}$
is bounded for $\mu\to \lambda$, hence $\tfrac{|g|^2}{|f|^{2\lambda}}$ is integrable on $B'$ by the monotone convergence theorem.
This contradiction completes the proof of the theorem.
\end{proof}

\begin{rmk}[The result of Koll\'{a}r and Lichtin]
Suppose that $X$ and $H$ are as in Theorem~\ref{thm_BudurSaito}. We have seen in the proof of the theorem that its statement is equivalent to the fact that for every $g\in\cO_X(X)$ nonzero,
the largest root of $b_{g\delta}$ is $-\lct_g(X,H)$. The special case $g=1$ corresponds to the assertion in Theorem~\ref{thm_LichtinKollar}
\end{rmk}

\begin{proof}[Proof of Theorem~\ref{thm_ELSV}]
We may and will assume that $X$ is an affine variety, with $R=\cO_X(X)$, and $H$ is defined by $f\in R$. Since $\lambda$ is a jumping number, it follows that there is $g\in R$ such that $g\in\cJ\big(X,(\lambda-\epsilon)H\big)$
for $0<\epsilon\ll 1$, but $g\not\in \cJ(X,\lambda H)$. For simplicity, we write $V^{\alpha}$ for $V^{\alpha}\iota_+(\cO_X)$.
By Theorem~\ref{thm_BudurSaito}, we have $g\delta\in V^{\lambda}\smallsetminus V^{>\lambda}$. 
On the other hand, by definition of $b_f$, there is $P\in D_R[s]$ such that
$$b_f(s)\delta=P\cdot t\delta.$$
Note that $\delta\in V^{>0}$: this follows from Theorem~\ref{thm_BudurSaito} since $\cJ(X,\lambda H)=\cO_X$ for $\lambda=0$; alternatively, it follows from the description
of the $V$-filtration in Proposition~\ref{prop_char_V_filt} and the fact that the roots of $b_f$ are negative by Theorem~\ref{thm_DM_bfcn}i). Therefore $t\delta\in V^{>1}\subseteq
V^{>\lambda}$. Since $b_f(s)g\delta=gP\cdot t\delta\in V^{>\lambda}$, it follows that $b_f(s)$ annihilates a nonzero element in ${\rm Gr}_V^{\lambda}$. 
Since $(s+\lambda)^N$ annihilates ${\rm Gr}_V^{\lambda}$ for some $N$, it follows that ${\rm gcd}\big(b_f,(s+\lambda)^N\big)$ annihilates a nonzero element, hence
$(s+\lambda)$ divides $b_f$.
\end{proof}

\subsection{The $V$-filtration and the minimal exponent}\label{section_minimal_exponent}

Our goal in this section is to discuss a refinement of the notions of log canonical threshold and multiplier ideals due to Saito. 
The idea is to use the $V$-filtration in order to define a version of multiplier ideals that give new information also for $\lambda\geq 1$.

We fix a smooth, irreducible algebraic variety 
$X$ and a (nonempty) hypersurface $H$ in $X$. To begin with, we assume that $H$ is defined by $f\in\cO_X(X)$. For simplicity, we write
$V^{\alpha}$ for $V^{\alpha}\iota_+(\cO_X)$, where $\iota\colon X\hookrightarrow X\times\AA^1$ is the graph embedding corresponding to $f$. 

\begin{defi}
For every $\lambda\in\QQ_{\geq 0}$, we write $\lambda=\alpha+q$, where $q\in\ZZ_{\geq 0}$ and $\alpha\in [0,1)$, and we denote by
$\widetilde{\cJ}(X, \lambda H)$ the coherent ideal of $\cO_X$ consisting of those $h\in\cO_X$ with the property that there are $h_0,\ldots,h_{q-1}\in\cO_X$
such that $h_0\delta+\ldots+h_{q-1}\partial_t^{q-1}\delta+h\partial_t^q\delta\in V^{>\alpha}$. 
\end{defi}

These ideals have been introduced by Saito in \cite{Saito_Hodge} as \emph{microlocal multiplier ideals}, due to the fact that the definition
was expressed in terms of the so-called \emph{microlocal $V$-filtration} (however, we will not use this terminology). Under the name of \emph{higher multiplier ideals}
and with a different indexing, the ideals have been studied systematically by Schnell and Yang in their recent preprint \cite{SchnellYang} using twisted Hodge modules.

\begin{rmk}[The connection with the classical version]
It is a consequence of Theorem~\ref{thm_BudurSaito} that for $\lambda<1$, we have
$$\widetilde{\cJ}(X,\lambda H)=\cJ(X,\lambda H).$$
\end{rmk}

\begin{rmk}[Independence of the equation]
If $g\in\cO_X(X)$ defines the same hypersurface, then we can write $g=pf$, for some invertible $p\in\cO_X(X)$. If $\iota_f$ and $\iota_g$ are the graph embeddings
corresponding to $f$ and $g$, respectively, then it follows from Remark~\ref{Vfilt_rescaling} that 
$$V^{>\alpha}(\iota_g)_+(\cO_X)=\big\{\sum_{j=0}^qp^{j+1}u_j\partial_t^j\delta\mid q\in\ZZ_{\geq 0},\,\, \sum_{j=0}^qu_j\partial_t^j\delta\in V^{>\alpha}(\iota_f)_+(\cO_X)\big\}.$$
This immediately implies that $\widetilde{\cJ}(X, \lambda H)$ does not depend on the choice of $f$. We can thus define $\widetilde{\cJ}(X, \lambda H)$ for every hypersurface $H$,
by glueing the corresponding ideals on a suitable open cover such that $H$ is a principal divisor in each of these open subsets. 
\end{rmk}

\begin{rmk}[Discreteness and right continuity of the ideals $\widetilde{\cJ}(X,\lambda H)$]\label{discreteness_microlocal_multiplier}
For every $H$, there is a positive integer $\ell$ such that $\widetilde{\cJ}(X, \lambda H)$ is constant
for $\lambda\in \big[\tfrac{i}{\ell},\tfrac{i+1}{\ell}\big)$ for every $i\in\ZZ_{\geq 0}$. Indeed, it is enough to check this when $H$ is defined by $f\in\cO_X(X)$, in which case
the assertion follows from the fact that
$V^{\bullet}\iota_+(\cO_X)$ is discrete and left continuous.
 In particular, we see for every $\lambda\in\QQ_{\geq 0}$, there is $\lambda'>\lambda$ such that $\widetilde{\cJ}(X,\lambda H)=\widetilde{\cJ}(X,\mu H)$
 for every $\lambda\leq\mu\leq \lambda'$.
\end{rmk}

\begin{prop}\label{inclusion_microlocal}
For every hypersurface $H$ on the smooth, irreducible variety $X$ and every $\lambda\geq\mu$, we have
$$\widetilde{\cJ}(X,\lambda H)\subseteq \widetilde{\cJ}(X,\mu H).$$
\end{prop}

\begin{proof}
We may assume that $X$ is affine and $H$ is defined by $f\in\cO_X(X)$. 
It follows directly from the definition that for every $q\in\ZZ_{\geq 0}$, we have 
$$\widetilde{\cJ}(X, \lambda H)\subseteq\widetilde{\cJ}(X,\mu H)$$
for $q\leq\mu\leq\lambda<q+1$. Therefore, in order to prove the proposition, it is enough to show that
for every $q\in\ZZ_{>0}$ and every $\mu$, with $q-1\leq\mu<q$, we have
$$\widetilde{\cJ}(X,qH)\subseteq\widetilde{\cJ}\big(X,\mu H).$$
In order to prove this, let $h$ be a global section of $\widetilde{\cJ}(X, qH)$, hence there are $h_0,\ldots,h_{q-1}\in\cO_X(X)$
such that $u=h_0\delta+\ldots+h_{q-1}\partial_t^{q-1}\delta+h\partial_t^q\delta\in V^{>0}$.
Note that $\delta\in V^{>0}$: this follows, using Proposition~\ref{prop_char_V_filt}, from the fact that $b_f$ has negative roots, see
Theorem~\ref{thm_DM_bfcn}i). Therefore we have
$$\partial_t\cdot (h_1\delta+\ldots+h\partial_t^{q-1}\delta)\in V^{>0}\subseteq V^0.$$
We deduce from Corollary~\ref{cor_rmk1_Vfiltration}ii) that $h_1\delta+\ldots+h\partial_t^{q-1}\delta\in V^1$, hence $h\in\widetilde{\cJ}(X,\mu H)$
for every $\mu$, with $q-1\leq\mu<q$. This completes the proof.
\end{proof}

\begin{prop}\label{triviality_microlocal}
Let $X$ be a smooth, irreducible variety and $H$ the hypersurface defined by $f\in\cO_X(X)$. If $\lambda\in\QQ_{\geq 0}$ is written as $\lambda=\alpha+q$,
where $q\in\ZZ_{\geq 0}$ and $\alpha\in [0,1)$, then $\widetilde{\cJ}(X,\lambda H)=\cO_X$ if and only if $\partial_t^q\delta\in V^{>\alpha}$.
\end{prop}

\begin{proof}
The ``if" part is clear from the definition, so we only need to prove the converse. We argue by induction on $q$, the case $q=0$ being clear. 
We may and will assume that $X$ is affine. Suppose that $1\in\Gamma\big(X,\widetilde{\cJ}(X,\lambda H)\big)$, so there are
$h_0,\ldots,h_{q-1}\in\cO_X(X)$ such that 
\begin{equation}\label{eq_triviality_microlocal}
h_0\delta+\ldots+h_{q-1}\partial_t^{q-1}\delta+\partial_t^q\delta\in V^{>\alpha}.
\end{equation}
By Proposition~\ref{inclusion_microlocal}, we know that $\widetilde{\cJ}\big(X,(\lambda-i)H\big)=\cO_X$ for $1\leq i\leq q$, hence
by induction we know that $\partial_t^j\delta\in V^{>\alpha}$ for $j<q$. We thus deduce from (\ref{eq_triviality_microlocal}) that
$\partial_t^q\delta\in V^{>\alpha}$, completing the proof of the induction step.
\end{proof}

\begin{eg}[The ideals $\widetilde{\cJ}(X,\lambda H)$ when $H$ is smooth]
If $H$ is a smooth hypersurface of $X$, then $\widetilde{\cJ}(X, \lambda H)=\cO_X$ for all $\lambda\geq 0$. Indeed, it follows from 
Remark~\ref{Vfilt_on_B_f_smooth} that 
$\partial_t^j\delta\in V^1\iota_+(\cO_X)$ for all $j\geq 0$. 
In fact, the converse also holds: if $\widetilde{\cJ}(X,\lambda H)=\cO_X$ for all $\lambda\geq 0$, then $H$ is smooth. 
However, the proof of this assertion requires tools beyond the ones we discuss in these notes:
the condition $\widetilde{\cJ}(X,\lambda H)=\cO_X$ for all $\lambda\geq 0$ is equivalent to $\widetilde{b}_H(s)=1$
(see Theorem~\ref{microlocal_and_min_exp}); the fact that this implies that $H$ is smooth is proved in 
\cite{BrianconMaisonobe} and \cite[Theorem~E(3)]{MP}.
\end{eg}

Our next goal is to describe $\min\big\{\lambda\in\QQ_{\geq 0}\mid \widetilde{\cJ}(X,\lambda H)\neq\cO_X\}$.
Note that by Remark~\ref{discreteness_microlocal_multiplier} and Proposition~\ref{inclusion_microlocal}, it makes sense to consider this minimum
(though this may be infinite, as seen in the above example). 

Let $H$ be a nonempty hypersurface in the smooth, irreducible variety $X$. Recall that in this case 
we write $b_H(s)=(s+1)\cdot\widetilde{b}_H(s)$.

\begin{defi}
The \emph{minimal exponent} $\widetilde{\alpha}(X,H)$ (also written $\widetilde{\alpha}(H)$ when $X$ is understood) of the hypersurface $H$
 is the negative of the largest root of $\widetilde{b}_H$,
with the convention that this is infinite if $\widetilde{b}_H=1$. 
\end{defi}


\begin{rmk}[The log canonical threshold vs. the minimal exponent]
It is a consequence of Theorem~\ref{thm_LichtinKollar} that
$$\lct(X,H)=\min\big\{\widetilde{\alpha}(X,H),1\big\}.$$
\end{rmk}

The minimal exponent was introduced by Saito in \cite{Saito_microlocal}, where it was related to the microlocal $V$-filtration. 
We now show that the minimal exponent governs the triviality of the refined multiplier ideals.

\begin{thm}\label{microlocal_and_min_exp}
If $H$ is a nonempty hypersurface in the smooth, irreducible, algebraic variety $X$, then
$$\min\big\{\lambda\in\QQ_{\geq 0}\mid \widetilde{\cJ}(X,\lambda H)\neq\cO_X\}=\widetilde{\alpha}(X,H).$$
\end{thm}

This result was proved by Saito in \cite{Saito_Hodge} using the connection with the microlocal $V$-filtration. We give a proof using the following
result, which is of independent interest:

\begin{thm}\label{equality_b_fcn}
If $X$ is a smooth, irreducible algebraic variety and $f\in\cO_X(X)$ is not invertible, then for every $q\in\ZZ_{\geq 0}$, we have
$$b_{\partial_t^q\delta}(s)=(s+1)\cdot \widetilde{b}_f(s-q).$$
\end{thm}

A slightly weaker assertion (the fact that $\widetilde{b}_f(s-q)$ divides $b_{\partial_t^q\delta}(s)$, which in turn divides $(s+1)\cdot \widetilde{b}_f(s-q)$) was proved in 
\cite[Proposition~6.11]{MP}, also making use of the microlocal $V$-filtration. In what follows we give a direct proof. 

\begin{proof}[Proof of Theorem~\ref{equality_b_fcn}]

 We may and will assume that $X$ is affine, with $R=\cO_X(X)$, and that we have coordinates
$x_1,\ldots,x_n$ on $X$. We begin by showing the easier assertion, that $b_{\partial_t^q\delta}(s)$ divides $(s+1)\cdot \widetilde{b}_f(s-q)$

The case $q=0$ is clear, hence from now on we assume $q\geq 1$. We put $\widetilde{b}=\widetilde{b}_{\delta}(s)$. 
By definition of $\widetilde{b}_{\delta}(s)$, there is $P(s)\in D_R[s]$ such that 
$$(s+1)\widetilde{b}(s)\delta=P(s)f\delta.$$
Since $s+1=-t\partial_t$ and $P(s)f\delta=P(s)t\delta=tP(s-1)\delta$ by Lemma~\ref{lem_s}, we have
$$-t\partial_t\widetilde{b}(s)\delta=tP(s-1)\delta.$$
Since the action of $t$ on $\iota_+(\cO_X)$ is injective (see Remark~\ref{rmk_M_no_ftorsion}), we deduce that
$$\partial_t\widetilde{b}(s)\delta=-P(s-1)\delta.$$
Using again Lemma~\ref{lem_s}, we have
$$(s+1)\widetilde{b}(s-q)\partial_t^q\delta=(s+1)\partial_t^q\widetilde{b}(s)\delta=-(s+1)\partial_t^{q-1}P(s-1)\delta$$
$$=t\partial_t^qP(s-1)\delta=P(s-q)t\partial_t^m\delta.$$
Since $P(s-q)t\in V^1\cD_{X\times {\mathbf A}^1}$, we conclude that 
the $b$-function $b_{\partial_t^q\delta}(s)$ exists and divides the polynomial
$(s+1)\widetilde{b}(s-q)$.

We now show that $(s+1)\cdot \widetilde{b}_f(s-q)$ divides $b_{\partial_t^q\delta}(s)$.
Let $p=b_{\partial_t^q\delta}$. By definition of $b$-functions, we have
$$p(s)\partial_t^q\delta\in V^1\cD_{X\times\AA^1}\cdot \partial_t^q\delta.$$
It will be convenient to use the isomorphism of $\cD_R\langle t,\partial_t\rangle$-modules $\iota_+\big(\cO_X[1/f]\big)\simeq \cO_X[1/f,s]f^s$ (see Proposition~\ref{prop_V_and_b}). 
Note that it follows from Lemma~\ref{lem_s}ii) that
$$\partial_t^qf^s=\partial_t^qt^tf^{s-q}=(-1)^qq!{{s}\choose {q}}f^{s-q},$$
and thus
$$t^j\partial_t^qf^s=(-1)^qq!{{s+j}\choose{q}}f^{s-q+j}\quad\text{for all}\quad j\geq 0.$$
Since $V^1\cD_{X\times\AA^1}=\sum_{j\geq 1}\cD_X[s]t^j$, we conclude that 
we have a positive integer $d$ and $A_1,\ldots,A_d\in D_R[s]$ such that 
\begin{equation}\label{eq2_equality_b_fcn}
p(s){s\choose q}f^{s-q}=\sum_{j=1}^dA_j(s){{s+j}\choose{q}}f^{s-q+j}.
\end{equation}

We first show by descending induction on $\ell\geq 2$ that we may assume that $A_j\in R[s]$ for all $j\geq \ell$. This is trivial if $p>d$, hence it is enough to show that if
$A_j\in R[s]$ for all $j\geq \ell+1$, with $\ell\geq 2$, then we can modify $A_{\ell}$ and $A_{\ell-1}$ so that also $A_{\ell}\in R[s]$. Let us write
$$A_{\ell}=A_{\ell,0}+\sum_{i=1}^nA_{\ell,i}\partial_{x_i}, \quad\text{with}\quad A_{\ell,0}\in R[s].$$
Since $\partial_{x_i}\cdot f^{s-q+\ell}=(s-q+\ell)\tfrac{\partial f}{\partial x_i}f^{s-q+\ell-1}$ and
$(s-q+\ell){{s+\ell}\choose{q}}=(s+\ell){{s+\ell-1}\choose{q}}$, it follows that 
$$A_{\ell,i}\partial_{x_i}{{s+\ell}\choose{q}}f^{s-q+\ell}\in D_R[s]{{s+\ell-1}\choose{q}}f^{s-q+\ell-1}\quad\text{for all}\quad 1\leq i\leq n.$$
We thus see that after modifying $A_{\ell-1}$, we may assume that $A_{\ell}=A_{\ell,0}\in R[s]$.
The conclusion is that we may and will assume that $A_2,\ldots,A_d\in R[s]$. 

The next step is to show, by descending induction on $\ell\geq 2$, that we may assume that $A_j\in R$ for all $j\geq \ell$. Again, this is clear if $\ell>d$.
We assume that $A_j\in R$ for all $j\geq\ell+1$, with $\ell\geq 2$, and show that we can modify $A_{\ell}$ and $A_{\ell-1}$ so that $A_{\ell}\in R$. 
Note that we can write
$$A_{\ell}(s)=B_{\ell}(s-q+\ell)+C_{\ell}\quad\text{with}\quad B_{\ell}\in R[s],\,\,C_{\ell}\in R$$
and we have
$$B_{\ell}(s-q+\ell){{s+\ell}\choose{q}}f^{s-q+\ell}\in R[s]{{s+\ell-1}\choose{q}}f^{s-q+\ell-1}.$$
We can thus modify $A_{\ell-1}$ so that $A_{\ell}=C_{\ell}\in R$. We conclude that we may and will assume that $A_j\in R$ for all $j\geq 2$. 

We next specialize $s$ to $s_0$ (see Remark~\ref{rmk_specialize_s}), where $s_0\in\{-1,0,\ldots,q-2\}$. In this case, since ${{s_0+1}\choose {q}}=0$, equation (\ref{eq2_equality_b_fcn}) becomes
\begin{equation}\label{eq3_equality_b_fcn}
p(s_0){{s_0}\choose {q}}f^{s_0-q}=\sum_{j=2}^dA_j{{s_0+j}\choose{q}}f^{s_0-q+j}.
\end{equation}
Note that if we have $\lambda+\sum_{j=1}^dg_jf^j=0$ for some $\lambda\in k$ and $g_1,\ldots,g_d\in R$, by evaluating at a point in $H$, we conclude that $\lambda=0$.
We deduce from (\ref{eq3_equality_b_fcn}) that $p(-1)=0$ and 
$$\sum_{j=2}^dA_j{{s_0+j}\choose{q}}f^{j}=0$$
for every $s_0\in \{-1,0,\ldots,q-2\}$. Let us write $p=(s+1)\widetilde{p}$.
Note that 
$$\sum_{j=2}^dA_j{{s+j}\choose{q}}f^{j-2}\in R[s]$$ has degree $\leq q$ in $s$ and it vanishes when $s=s_0\in \{-1,0,\ldots,q-2\}$. 
We thus conclude that $\sum_{j=2}^dA_j{{s+j}\choose{q}}f^{j-2}=Q\cdot {{s+1}\choose {q}}$ for some $Q\in R$. Therefore we have
$$p(s){s\choose q}f^{s-q}=A_1{{s+1}\choose q}f^{s-q+1}+Q{{s+1}\choose q}f^{s-q+2}=(A_1+Qf){{s+1}\choose{q}}f^{s-q+1}.$$
Since $p=(s+1)\widetilde{p}$ and $\cO_X[1/f,s]f^s$ is a free $\cO_X[1/f,s]$-module, we conclude that 
$$(s-q+1)\widetilde{p}(s)f^{s-q}\in\cD_X[s]\cdot f^{s-q+1}.$$
Multiplying on the left by $t^q$ (which has the effect of replacing $s$ by $s+q$), we obtain
$$(s+1)\widetilde{p}(s+q)f^s\in\cD_X[s]f^{s+1},$$
hence $b_f(s)=(s+1)\widetilde{b}_f(s)$ divides $(s+1)\widetilde{p}(s+q)$. We thus conclude that, indeed, $(s+1)\widetilde{b}_f(s-q)$ divides $p(s)$,
completing the proof of the theorem.
\end{proof}

The connection between the minimal exponent and the refined version of multiplier ideals now follows easily:

\begin{proof}[Proof of Theorem~\ref{microlocal_and_min_exp}]
By taking a suitable open cover of $X$, we may and will assume that $H$ is defined by $f\in\cO_X(X)$.
By Proposition~\ref{triviality_microlocal}, the assertion in the theorem is equivalent to the fact that for every $q\in\ZZ_{\geq 0}$ and $\alpha\in\QQ\cap [0,1)$, we have
$\partial_t^q\delta\in V^{>\alpha}$ if and only if $q+\alpha<\widetilde{\alpha}(X,H)$. Note also that by Corollary~\ref{prop_char_V_filt}, we have
$\partial_t^q\delta\in V^{>\alpha}$ if and only if all roots of $b_{\partial_t^q\delta}$ are $<-\alpha$. Since $\alpha<1$, it follows from Theorem~\ref{equality_b_fcn}
that this is the case if and only if all roots of $\widetilde{b}_f$ are $<-q-\alpha$, or equivalently, $\widetilde{\alpha}(X,H)>q+\alpha$. 
\end{proof}

We now give a lower bound for $\widetilde{\alpha}(X,H)$ in terms of a log resolution. This was proved in \cite{MP} using techniques involving mixed Hodge modules; the more elementary
proof that we give here, based on the Kashiwara-Lichtin estimate for roots of $b$-functions, is from \cite{DM}. Note that if $H$ is not reduced, then $\lct(X,H)<1$, hence
$\widetilde{\alpha}(X,H)=\lct(X,H)$ can be computed in terms of a log resolution via (\ref{eq_formula_lct}). 

\begin{thm}
Let $X$ be a smooth, irreducible algebraic variety and $H$ a nonempty reduced hypersurface in $X$.
Let $\pi\colon Y\to X$ be a log resolution of $(X,H)$ such that the strict transforms of the irreducible components of $H$ on $Y$ are disjoint.
If we write $\pi^*(H)=\sum_{i=1}^Na_iE_i$ and $K_{Y/X}=\sum_{i=1}^Nk_iE_i$, then 
\begin{equation}\label{eq_bound_min_exp}
\widetilde{\alpha}(X,H)\geq\min\big\{\tfrac{k_i+1}{a_i}\mid E_i\,\,\text{exceptional}\big\}.
\end{equation}
\end{thm}

\begin{proof}
After taking a suitable open cover of $X$, we may and will assume that $H$ is defined by $f\in\cO_X(X)$.
Let $\lambda$ be the right-hand side of (\ref{eq_bound_min_exp}) and let us write $\lambda=q+\beta$, with $q\in\ZZ_{\geq 0}$ and $\beta\in (0,1]$. 
We have $\widetilde{\alpha}(X,H)\geq \lambda$ if and only if every root of $\widetilde{b}_f(s-q)$ is $\leq -\beta$; by Theorem~\ref{equality_b_fcn},
this is equivalent to $b_{\partial_t^q\delta}(s)$ having all roots $\leq-\beta$. However, by Theorem~\ref{thm_DM_bfcn} that such a root is either $-m$,
for some positive integer $m$
(in which case it is clearly $\leq-\beta$) or of the form $q-\tfrac{k_i+\ell}{a_i}$, for some positive integer $\ell$ and some $i$ with $E_i$ exceptional (in which case, 
it is $\leq q-\lambda=-\beta$). This completes the proof.
\end{proof}

\begin{rmk}[Computation of the minimal exponent as the minimum on a log resolution]
Unlike in the case of formula (\ref{eq_formula_lct}) that computes $\lct(X,H)$ in terms of any log resolution, we can't hope that the inequality in (\ref{eq_bound_min_exp})
is an equality for an arbitrary resolution. Indeed, if the minimum on the right-hand side of (\ref{eq_bound_min_exp}) is $>1$ (so $\widetilde{\alpha}(X,H)>1$), then 
one can construct a sequence of blow-ups with smooth centers of codimension $2$, such that the corresponding minima converge to $1$.
Indeed, let $E_j$ be the strict transform of an irreducible component of $H$ (so that $a_j=1$ and $k_j=0$) and let $E_i$ be an exceptional divisor that intersects $E_j$.
We consider the blow-up $Y_1\to Y$ of $Y$ along a connected component of $E_i\cap E_j$, and let $G_1$ be the exceptional divisor. We consider next
the blow-up $Y_2\to Y_1$ of $Y_1$ along the intersection of
$G_1$ with the strict transform of $E_j$ and let $G_2$ be the exceptional divisor, etc. An easy computation shows that the coefficient of $G_{\ell}$ in the inverse image of $H$
on $Y_{\ell}$ is $a_i+\ell$ and its coefficient in $K_{Y_{\ell}/X}$ is $k_i+\ell$. Note that $\lim_{\ell\to\infty}\tfrac{k_i+1+\ell}{a_i+\ell}=1$. 

However, it is an interesting question whether given a reduced hypersurface $H$ on $X$, there is a log resolution $\pi\colon Y\to X$ of $(X,H)$
such that we have equality in (\ref{eq_bound_min_exp}).
\end{rmk}

One can show that certain properties of the log canonical threshold 
extend to the minimal exponent $\widetilde{\alpha}(X,H)$, see \cite{MP}. Two important ones concern the behavior
with respect to restriction to a smooth hypersurface and semicontinuity in families. These properties can be proved by following
the corresponding proofs for log canonical thresholds via multiplier ideals in \cite[Chapter~9]{Lazarsfeld} once we have a family of ideals that detect 
the minimal exponent. This was done in \cite{MP} using the \emph{Hodge ideals} $I_p(\alpha H)$ constructed in \cite{MP2} and \cite{MP3}
(here $p\in {\mathbf Z}_{\geq 0}$ and $\alpha$ is a positive rational number). 
One can attach such ideals to any reduced\footnote{Note that if $H$ is not reduced, then ${\rm lct}(X,H)<1$ and thus $\widetilde{\alpha}(X,H)={\rm lct}(X,H)$.
Therefore, in order to extend properties of the log canonical threshold to the minimal exponent, it is enough to consider the case when $H$ is reduced.}
hypersurface $H$ in $X$ by using the Hodge filtration on suitable Hodge modules associated to a local equation of $H$. 
By making use of results on mixed Hodge modules, it was shown in \cite{MP4} and \cite{MP3} 
that these ideals satisfy properties that parallel those of multiplier ideals. On the other hand, 
it was shown in \cite[Theorem~A']{MP} that if $\lambda=p+\alpha$, where $p\in {\mathbf Z}_{\geq 0}$ and $\alpha\in (0,1]$, then $I_p(\alpha H)\cdot\cO_H=\widetilde{\cJ}\big(X, (\lambda-\epsilon)H\big)\cdot\cO_H$
for $0<\epsilon\ll 1$. It thus follows from Proposition~\ref{triviality_microlocal} that $\widetilde{\alpha}(X,H)\geq p+\alpha$ if and only if $I_p(\alpha H)=\cO_X$. 
As we have already pointed out, Schnell and Yang have recently developed in \cite{SchnellYang}  a theory of twisted Hodge modules and 
used it to show that the ideals $\widetilde{\cJ}(X, \lambda H)$ themselves satisfy properties analogous to those of multiplier ideals, so one can also
use these invariants in order to prove the properties of the minimal exponent.

Finally, we mention that there is a notion of minimal exponent associated to local complete intersection subvarieties $Z$ inside a smooth variety $X$, see \cite{CDMO}. This is defined using the 
$V$-filtration associated to a local system of equations of $Z$ and many of the properties extend from the case of hypersurfaces to the case of arbitrary local complete intersections.

\end{document}